\documentclass[reqno, 10pt]{amsart}
\usepackage{amssymb,mathrsfs}
\usepackage[usenames, dvipsnames]{color}
\usepackage{esint}
\usepackage[colorlinks=true,linkcolor=blue,citecolor=red,urlcolor=black]{hyperref}
\usepackage{todonotes}
\usepackage{verbatim}
\usepackage{enumerate}
\usepackage{bbm}

\usepackage{cite}
\usepackage[nobysame,non-sorted-cites,initials]{amsrefs}
\def\MR#1{}

\theoremstyle{plain}
\newtheorem{theorem}{Theorem}[section]
\newtheorem{lemma}[theorem]{Lemma}

\newtheorem{proposition}[theorem]{Proposition}

\theoremstyle{definition}

\theoremstyle{remark}
\newtheorem{remark}[theorem]{Remark}

\newcommand{\dv}{\operatorname{div}}
\newcommand{\osc}{\operatorname{osc}}

\numberwithin{equation}{section}

\newcommand{\bR}{\mathbb{R}}

\newcommand{\bZ}{\mathbb{Z}}
\newcommand{\bS}{\mathbb{S}}

\newcommand\cD{\mathcal{D}}

\newcommand\cG{\mathcal{G}}

\newcommand\cQ{\mathcal{Q}}

\newcommand\sA{\mathscr{A}}

\makeatletter
\def\dashint{\operatorname%
	{\,\,\text{\bf--}\kern-.98em\DOTSI\intop\ilimits@\!\!}}
\makeatother

\newcommand\dashnorm[2]{\nparallel\kern-.2em #1 \Vert_{#2}}

\begin{document}
	\title[The conductivity problem with imperfect bonding interfaces]{The conductivity problem with imperfect bonding interfaces and finite internal conductivities}
	
	\author[H. Dong]{Hongjie Dong}
	
	\author[Z. Yang]{Zhuolun Yang}
	
	\author[H. Zhu]{Hanye Zhu}
	
	\address[H. Dong]{Division of Applied Mathematics, Brown University, 182 George Street, Providence, RI 02912, USA}
	\email{Hongjie\_Dong@brown.edu}

	\address[Z. Yang]{Department of Mathematics, The Ohio State University, 231 West 18th Avenue, Columbus, OH 43210, USA}
	\email{yang.8242@osu.edu}

	\address[H. Zhu]{Department of Mathematics, Duke University, 120 Science Drive, Durham, NC 27708, USA}
	\email{hanye.zhu@duke.edu}
	
	\thanks{H. Dong was partially supported by the NSF under grant DMS-2350129}
	\thanks{Z. Yang was partially supported by the NSF under grant DMS-2550221}
	\thanks{H. Zhu was partially supported by the NSF under grants DMS-2006372 and DMS-2306726}
	
	
	
	\begin{abstract}
		We study the field concentration phenomenon between two closely spaced inclusions with imperfect bonding interfaces of low conductivity type. The inclusions are assumed to have finite conductivities. The problem is governed by a system of elliptic equations coupled with Robin-type boundary conditions.  While it is known that finite-conductivity inclusions with ideal interfaces yield bounded gradients, in this paper we show that with imperfect bonding interfaces, the gradient of the solution may blow up as $\varepsilon$ (the distance between two inclusions) tends to zero when the bonding parameter $\gamma$ is large, while it is uniformly bounded independently of $\varepsilon$ when the bonding parameter $\gamma$ is small. Moreover, we identify the threshold of $\gamma$ and the optimal blow-up rates under certain symmetry assumptions. Compared to the case when the inclusions are perfect conductors, we find a novel logarithmic blow-up phenomenon at the critical value of $\gamma$.
	\end{abstract}
	
	\maketitle
	
	\section{Introduction}
	\subsection{Background}
	The study of stress and electric field concentrations in composite media with closely spaced inclusions has attracted extensive attention over the past few decades due to its strong physical background and mathematical challenges. When inclusions with extreme material properties are located close to each other, the narrow gap between them often gives rise to a significant build-up of stress or field intensity, which may lead to material failure (see \cites{BASL,Kel}) or field enhancement (see \cites{acsphotonics,KANG20191670}). Quantitative analysis of these field concentrations traces back to the pioneering work of Keller \cite{Keller} on the effective conductivity of dense arrays of cylinder or sphere packings. Over the past three decades, substantial progress has been made in quantitatively characterizing the field concentration for perfectly bonded inclusions (i.e., with continuity of potential and normal flux across interfaces): optimal estimates  and asymptotic characterizations near touching regimes have been established.
	
	However, in real-world applications, the interface between the matrix and the inclusions is rarely ideal. Physical phenomena such as thermal or electrical contact resistance, surface roughness, coating layers, or partial debonding give rise to imperfect bonding conditions. It is known in the physics literature that interfacial effects play a significant role in a variety of systems. For example, they may dramatically alter the effective properties (\cites{prl,LipVer,CHIEW198390,BenY,EVERY1992123,HASHIN1992767}). As emphasized in \cite{prl}, the study of such interfacial effects is challenging both experimentally and theoretically due to the complexity of the underlying microstructure. A widely accepted mathematical framework models these non-ideal interfaces via transmission conditions of low-conductivity type (LC-type, see \cites{fukushima2024finiteness,BenMil}). Under such transmission conditions, the flux remains continuous across the interfacial boundary while the potential experiences a jump discontinuity proportional to the flux. The field concentration phenomenon for perfect conductors with LC-type non-ideal bonding interfaces was systematically analyzed in our previous work \cite{DongYangZhu26} (see also \cites{fukushima2024finiteness,DongLiZhao25}), where a fundamental dichotomy was established: for small interfacial resistance, the local field remains uniformly bounded independently of the distance between the inclusions, whereas for large interfacial resistance, the local field may become arbitrarily large as the distance approaches 0. 
	Despite these advances, the full physical picture remains incomplete. Previous investigations into LC-type imperfect interfaces \cites{fukushima2024finiteness, DongYangZhu26, DongLiZhao25} assumed infinite internal conductivity. 
	It was unclear whether finite internal conductivities would completely eliminate field concentration under LC-type transmission conditions, or if a critical transition would persist.
	In this work, we address this fundamental problem by establishing optimal gradient estimates for inclusions with finite conductivities and LC-type imperfect interfaces.
	
	Let us describe the mathematical framework. Since our primary interest is the behavior of the electric field in the narrow region between two closely spaced inclusions, we formulate the problem in a domain containing two inclusions. Let $n\geq 2$, and let $\Omega\subset \bR^n$ contain two relatively strictly convex open sets $\cD_1$ and $\cD_2$ satisfying
	\[
	\operatorname{dist}(\cD_1\cup \cD_2,\partial\Omega)>1.
	\]
	We denote the distance between the two inclusions by
	\[
	\varepsilon:=\operatorname{dist}(\cD_1,\cD_2),
	\]
	and set
	\[
	\widetilde{\Omega}:=\Omega\setminus\overline{(\cD_1\cup \cD_2)}.
	\]
	When the inclusions are perfectly bonded to the surrounding matrix, the conductivity problem is described by the elliptic equation
	\begin{equation*}
		\begin{cases}
			\operatorname{div}(a(x)\nabla u)=0 & \text{in }\Omega,\\
			u=\varphi(x) & \text{on }\partial\Omega,
		\end{cases}
	\end{equation*}
	where the conductivity distribution is given by
	\[
	a=k_1\chi_{\cD_1}+k_2\chi_{\cD_2}+\chi_{\widetilde{\Omega}}.
	\]
	Here, $u$ denotes the voltage potential, while $-\nabla u$ represents the associated electric field. A fundamental characteristic of perfect bonding is the continuity of both the potential and the normal flux across the inclusion boundaries:
	\begin{equation}\label{transmission}
		u|_+=u|_-,\qquad
		\left.\frac{\partial u}{\partial\nu}\right|_+
		=k_i\left.\frac{\partial u}{\partial\nu}\right|_-
		\quad\text{on }\partial \cD_i.
	\end{equation}
	Throughout, the subscripts $+$ and $-$ indicate the limiting values from the exterior and interior of the inclusion, respectively, and $\nu$ denotes the inward normal vector on $\partial \cD_1\cup\partial \cD_2$.
	
	When the conductivities $k_i$ remain uniformly bounded away from both zero and infinity, it was established in \cites{BV,LV,LN} that the gradient of the solution remains uniformly bounded independently of the distance $\varepsilon$ between the inclusions. In contrast, when the conductivities degenerate toward either infinity or zero, the gradient may become unbounded as $\varepsilon\to0$. 
	
	In the perfectly conducting regime $k_1=k_2=\infty$, the problem reduces to
	\begin{equation*}
		\begin{cases}
			\Delta u=0 & \text{in }\widetilde{\Omega},\\
			u=U_j & \text{on }\partial \cD_j,\quad j=1,2,\\
			\displaystyle\int_{{\partial \cD_j}}\partial_\nu u\,d\sigma=0 & j=1,2,\\
			u=\varphi & \text{on }\partial\Omega,
		\end{cases}
	\end{equation*}
	where $U_j$ are constants determined by the zero-flux conditions.
	The optimal estimates for the gradient in the narrow region have been established in a number of works \cites{AKLLL,AKL,Y1,Y2,BLY1,BLY2}. More precisely,
	\begin{equation*}
		\begin{cases}
			\|\nabla u\|_{L^\infty(\widetilde{\Omega})}
			\le C\varepsilon^{-1/2}\|\varphi\|_{C^2(\partial\Omega)},
			& n=2,\\[2mm]
			\|\nabla u\|_{L^\infty(\widetilde{\Omega})}
			\le C|\varepsilon\ln\varepsilon|^{-1}\|\varphi\|_{C^2(\partial\Omega)},
			& n=3,\\[2mm]
			\|\nabla u\|_{L^\infty(\widetilde{\Omega})}
			\le C\varepsilon^{-1}\|\varphi\|_{C^2(\partial\Omega)},
			& n\ge4.
		\end{cases}
	\end{equation*}
	These estimates are known to be optimal.
	
	The opposite limiting regime, corresponding to insulating inclusions with $k_1=k_2=0$, leads to
	\begin{equation*}
		\begin{cases}
			\Delta u=0 & \text{in }\widetilde{\Omega},\\
			\partial_\nu u=0 & \text{on }\partial \cD_j,\quad j=1,2,\\
			u=\varphi & \text{on }\partial\Omega.
		\end{cases}
	\end{equation*}
	While the optimal gradient blow-up rate in two dimensions was established approximately two decades ago in \cites{AKL,AKLLL}, the corresponding optimal rates in dimensions $n\ge3$ have only recently been obtained in \cites{DLY,DLY2,li2024optimalgradientestimatesinsulated}. These higher-dimensional rates are closely related to the first nonzero eigenvalue of an elliptic operator on $\bS^{n-2}$, with the operator determined by the principal curvatures of the inclusions. More recently, for the mixed case $k_1=0$ and $k_2=\infty$, it was shown in \cites{JiKang,DonYan23} that the gradient remains uniformly bounded independently of $\varepsilon$. Thus, the field concentration phenomenon under perfect bonding, including the determination of optimal blow-up rates, is now relatively well understood. A more comprehensive account of the existing literature can be found in the survey articles \cites{Kang,DY26}.
	
	The situation becomes substantially different when the interfaces between the inclusions and the surrounding matrix are non-ideal. In this setting, at least one of the classical transmission conditions \eqref{transmission} fails. In this work, we focus on imperfect interfaces of low-conductivity type (LC-type). Such an interface can be viewed as the limiting model of an inclusion surrounded by a thin shell of low conductivity as the shell thickness tends to zero. The limiting transmission conditions take the form
	\begin{equation}\label{transmission_new}
		\left.\frac{\partial u}{\partial\nu}\right|_+
		={k_i}\left.\frac{\partial u}{\partial\nu}\right|_-
		=-\gamma^{-1}(u|_+-u|_-)\quad {\text{on }\partial\cD_i,\quad i=1,2}.
	\end{equation}
	We refer to this as an imperfect interface of LC-type, with $\gamma>0$ representing the bonding parameter. Unlike the perfectly bonded case, the potential is no longer continuous across the interface.
	There is also another type of imperfect interface, known as the high-conductivity type (HC-type), for which continuity of the flux is lost instead. In two dimensions, the HC-type problem is dual to the LC-type problem. Detailed derivations and discussions of these imperfect transmission conditions can be found in \cites{fukushima2024finiteness,BenMil}.
	
	The case when both inclusions $\cD_1$ and $\cD_2$ have {LC-type imperfect interfaces and} their conductivities $k_1=k_2=\infty$ has been studied very recently. Under these assumptions, the conductivity problem becomes
	\begin{equation*}
		\begin{cases}
			\Delta u=0 & \text{in }\widetilde{\Omega},\\
			u+\gamma\partial_\nu u=U_j & \text{on }\partial \cD_j,\quad j=1,2,\\
			\displaystyle\int_{\partial \cD_j}\partial_\nu u\,d\sigma=0 & j=1,2,\\
			u=\varphi & \text{on }\partial\Omega,
		\end{cases}
	\end{equation*}
	where the constants $U_j$ are determined by the zero-flux conditions. Fukushima et al. \cite{fukushima2024finiteness} investigated this model for two equal-sized disks centered on the $x_2$ axis in $\Omega=\mathbb{R}^2$, subject to the far-field condition $\varphi=x_2$. They established that the gradient remains uniformly bounded, independently of $\varepsilon$. This is in sharp contrast to the corresponding perfect conductivity problem with ideal interface, where the gradient blows up as $\varepsilon\to0$. In our recent work \cite{DongYangZhu26}, we discovered an unexpected {dichotomy in field concentration} depending on the size of $\gamma$. Specifically, we showed that the gradient of the solution is uniformly bounded independently of $\varepsilon$ when $\gamma$ is sufficiently small. However, the gradient may blow up when $\gamma$ is large. Moreover, under certain symmetry assumptions, the critical threshold was identified as $\gamma=1/\mu$, where $\mu$ is the relative curvature at the narrowest point of the gap defined as in \eqref{fg_2} below. The gradient was shown to be uniformly bounded when $\gamma\le 1/\mu$, and the optimal gradient blow-up rate was established when $\gamma>1/\mu$.
	
	However, in realistic physical applications, composite inclusions may possess finite conductivities as well as non-ideal interfaces.
	While it is known that finite-conductivity inclusions with ideal interfaces yield bounded gradients (\cites{BV,LV,LN}) and {perfectly conducting inclusions} with LC-type imperfect interfaces yield a threshold-governed blow-up (\cites{DongYangZhu26}{)}, the combined interaction between finite internal conductivities ($0 < k_1, k_2 < \infty$) and interfacial resistance ($\gamma>0$) remained unexplored. In particular, it is unclear whether the interaction between finite conductivity and
	interfacial resistance suppresses, modifies, or induces new resonance patterns in
	field concentration.
	
	In this work, we address this fundamental question by establishing optimal gradient estimates for inclusions with finite conductivities $k_1, k_2\in(0,\infty)$ subject to LC-type transmission conditions \eqref{transmission_new}. We reveal that while the critical threshold $\gamma = 1/\mu$ persists, the interplay between finite internal conductivity and interfacial resistance induces a novel logarithmic resonance phenomenon at $\gamma = 1/\mu$, bridging the gap between ideal interfaces, infinite conductivity models, and realistic composite systems.
	
	\subsection{Problem setup and notation}
	
	Throughout this article, we assume that
	$$
	\gamma \ge \gamma_0 > 0, \quad \mbox{and} \quad k_1,k_2 \ge k_0 > 0.
	$$
	We consider the conductivity problem
	\begin{equation}\label{robin}
		\begin{cases}
			\Delta u_+ =0&\mbox{in}~\widetilde\Omega,\\
			\Delta u_- =0&\mbox{in}~\mathcal{D}_1\cup\mathcal{D}_2,\\ \partial_\nu u_+=k_i \partial_\nu u_-= -\gamma^{-1}(u_+-u_-)&\mbox{on}~\partial \mathcal{D}_i, \quad i=1,2,\\
			u_+= \varphi&\mbox{on}~\partial\Omega.
		\end{cases}
	\end{equation}
	where $\nu$ is the unit normal vector {pointing into} $\mathcal{D}_i$ on $\partial\mathcal{D}_i$.  We say that $u$ is a piecewise $H^1$ solution to \eqref{robin} if
	$$
	u = \begin{cases}
		u_+ &\mbox{in}~\widetilde\Omega,\\
		u_- &\mbox{in}~\mathcal{D}_1\cup\mathcal{D}_2,
	\end{cases}
	$$
	with $u_+ \in H^1 (\widetilde\Omega)$ and $u_- \in H^1(\mathcal{D}_1\cup\mathcal{D}_2)$ satisfying equation \eqref{robin}.
	
	We use the notation $x = (x', x_n)$, where $x' \in \bR^{n-1}$. We {choose coordinates} so that near the origin, {the portions of $\partial \cD_1$ and $\partial \cD_2$}, denoted by $\Gamma_+$ and $\Gamma_-$, are respectively the graphs of two $C^{1,1}$ functions in terms of $x'$. That is, for some $R_0 > 0$,
	\begin{align*}
		\Gamma_+ =& \left\{ x_n = \frac{\varepsilon}{2}+h_1(x'),~|x'|<R_0\right\},\\
		\Gamma_- =& \left\{ x_n = -\frac{\varepsilon}{2}+h_2(x'),~|x'|<R_0\right\},
	\end{align*}
	where $h_1(x')$ and $h_2(x')$ are $C^{1,1}$ functions satisfying
	\begin{equation}\label{fg_0}
		h_1(0')=h_2(0')=0,\quad \nabla_{x'}h_1(0')=\nabla_{x'}h_2(0')=0,
	\end{equation}
	\begin{equation}\label{fg_1}
		c_1 |x'|^2 \le h_1(x')-h_2(x')\quad\mbox{for}~~0\le |x'|<R_0,
	\end{equation}
	\begin{equation}\label{def:c_2}
		\|h_1\|_{C^{1,1}}+\|h_2\|_{C^{1,1}}\leq c_2,
	\end{equation}
	{with positive constants $c_1$ and $c_2$.} For $x_0 \in \widetilde\Omega$ with $|x_0'|< R_0$, $0 < r\leq R_0 - |x_0'|$, we denote
	\begin{align*}
		\Omega_{x_0,r}:=\left\{(x',x_{n})\in \widetilde{\Omega}:~-\frac{\varepsilon}{2}+h_2(x')<x_{n}<\frac{\varepsilon}{2}+h_1(x'),~|x'-x_0' |<r\right\}
	\end{align*}
	and $\Omega_{r}:=\Omega_{0,r}$.  For any domain $\cD$, we denote the H\"older {seminorm} and H\"older norm of a function $u$ by
	$$
	[u]_{C^\alpha(\cD)}:= \sup_{\substack{x,y \in \cD\\x \neq y}}\frac{|u(x)-u(y)|}{|x-y|^\alpha}, \quad \|u\|_{C^\alpha(\cD)}:= [u]_{C^\alpha(\cD)} + \|u\|_{L^\infty(\cD)}.$$
	We denote a partial H\"older seminorm and the corresponding norm by
	$$
	[u]_{C_{x'}^\alpha(\cD)}:= \sup_{\substack{x,y \in \cD\\x \neq y, x_n = y_n}}\frac{|u(x)-u(y)|}{|x-y|^\alpha}, \quad \|u\|_{C_{x'}^\alpha(\cD)}:= [u]_{C_{x'}^\alpha(\cD)} + \|u\|_{L^\infty(\cD)}.
	$$
	We also denote the averaged $L^p$ norm of $u$ over $\mathcal{D}$ by
	$$
	\dashnorm{u}{L^p(\cD)}:= \left( \fint_{\cD} |u|^p \right)^{1/p},
	$$
	and the average of $u$ over $\mathcal{D}$ by
	$$
	(u)_{\cD} : = \fint_{\cD} u.
	$$
	{We also consider} the following problem near the origin:
	\begin{equation}\label{main_problem_narrow}
		\begin{cases}
			\Delta u =0&\mbox{in}~\Omega_{R_0},\\
			u + \gamma \partial_\nu u = g_1&\mbox{on}~\Gamma_+,\\
			u + \gamma \partial_\nu u = g_2&\mbox{on}~\Gamma_-,
		\end{cases}
	\end{equation}
	where $\nu$ is the unit normal vector pointing upward on $\Gamma_+$ and pointing downward on $\Gamma_-$. Since $g_1, g_2$ are only defined on $\Gamma_\pm$ respectively, they are functions only in terms of $x'$, and {we write}
	$$
	g_1(x') = g_1\Big(x',\frac{\varepsilon}{2} + h_1(x') \Big) \quad \mbox{and} \quad g_2(x') = g_2\Big(x',-\frac{\varepsilon}{2} + h_2(x') \Big)
	$$
	for simplicity. {For boundary data, all $C^\beta_{x'}$ seminorms compare the data at arbitrary distinct tangential points using their graph representations.}

	When considering the global problem \eqref{robin}, we use $g_1$ and $g_2$ to denote $u_-|_{\Gamma_+}$ and $u_-|_{\Gamma_-}$ near the origin respectively.
	
	Throughout this article, we say that a constant $C$ is universal if it depends only on $n$, $\gamma_0$, $k_0$, $\mu$, $\cD_1$, $\cD_2$, and $\Omega$. In particular, a universal constant $C$ is independent of $\varepsilon$.
	
	\subsection{Main results}
	
	In the main results below, we further assume ${\cD_1, \cD_2}$ to be $C^{2,\sigma}$ for some $\sigma \in (0,1)$. In addition to \eqref{fg_0} and \eqref{def:c_2}, the corresponding graph functions $h_1$ and $h_2$ will satisfy
	\begin{equation}\label{fg_2}
		h_1(x') - h_2(x') = \mu |x'|^2 + O(|x'|^{2+\sigma}),
	\end{equation}
	and
	\begin{equation}\label{fg_3}
		\nabla_{x'} h_1(x')-\nabla_{x'} h_2(x')=2\mu x'+O(|x'|^{1+\sigma}),
	\end{equation}
	for $0\le|x'|<R_0,~ \mu > 0$.
	
	For our first main result below, 
	we also assume that for some $1 \le j \le n-1$, the domain $\widetilde\Omega$ is symmetric with respect to $x_j$, {and the boundary datum $\varphi \in C^{1,1}$ is odd} in $x_j$. 
	We define
	\begin{equation}
		\label{alpha_n}
		\alpha=\alpha(n,\gamma, \mu):= \frac {-(n-1)+\sqrt{(n-1)^2+4(n-2 +  2/(\mu\gamma))}}{2}.
	\end{equation}
	Note that $\alpha$ is always positive, and $\alpha <1$ if and only if $\gamma >  1/\mu$.

	\begin{theorem}\label{thm_upperbound}
		For $n \ge 2$, $\varepsilon \in (0,R_0/4)$, $\gamma > 0$, $\mu > 0$, let $h_1,h_2$ satisfy \eqref{fg_0}, \eqref{def:c_2}, \eqref{fg_2}, and \eqref{fg_3}. { Let $\Omega_{R_0}$ be symmetric in $x_j$ and let $u$ be the solution of \eqref{robin} that is odd in $x_j$ for some $1 \le j \le n-1$.} Then there exists a positive constant $C$, depending only on $n$, $\gamma$, $k_0$, $\mu$, $R_0$, $\|h_1\|_{C^{2,\sigma}}$, and $\|h_2\|_{C^{2,\sigma}}$, such that 
		\begin{itemize}
			\item when $0 < \gamma < 1/\mu$,
			\begin{equation}\label{grad_u_upperbound1}
				|\nabla u(x)| \le C \|\varphi\|_{L^\infty(\partial\Omega)}\quad \mbox{in}~ \Omega_{R_0/2},
			\end{equation}
			\item when $\gamma =1/\mu$,
			\begin{equation}\label{grad_u_upperbound3}
				|\nabla u(x)| \le C \|\varphi\|_{L^\infty(\partial\Omega)}\left( |\log (\varepsilon+|x'|^2)| \Big( \frac{1}{k_1} + \frac{1}{k_2} \Big) + 1 \right) \quad \mbox{in}~ \Omega_{R_0/2},
			\end{equation}
			\item when $\gamma >1/\mu$,
			\begin{equation}\label{grad_u_upperbound2}
				|\nabla u(x)| \le C \|\varphi\|_{L^\infty(\partial\Omega)}(\varepsilon + |x'|^2)^{\frac{\alpha-1}{2}} \quad \mbox{in}~  \Omega_{R_0/2},
			\end{equation}
		\end{itemize}
		where $\alpha$ is given in \eqref{alpha_n}.
	\end{theorem}

	Estimates \eqref{grad_u_upperbound3} and \eqref{grad_u_upperbound2} indicate that when $\gamma \ge 1/\mu$, the gradient of the solution may blow up as $\varepsilon \to 0$. {The following example shows that these blow-up rates are attained.}
	
	\begin{theorem}
		\label{main_thm}
		For $n \ge 2$, $\mu > 0$, $\varepsilon \in (0,1/(4\mu))$, let $\Omega = B_{5/\mu}$, $\cD_1 , \cD_2$ be balls of radius $1/\mu$ centered at $(0', 1/\mu + \varepsilon/2)$ and $(0', -1/\mu - \varepsilon/2)$, respectively, $\varphi(x) = x_1$, and let $u$ be the solution of \eqref{robin}. Then there exist positive constants $c$ and $C$, depending only on $n$, $k_1$, $k_2$, and $\mu$, such that
		\begin{itemize}
			\item when $\gamma=1/\mu$,
			\begin{equation}
				\label{grad_u_lower_bound_new}
				\|\nabla u\|_{L^\infty(\Omega_{c \sqrt\varepsilon})} \ge \frac{1}{C}|\log\varepsilon|, 
			\end{equation}
			\item when $ \gamma > 1/\mu$,
			\begin{equation}
				\label{grad_u_lower_bound}
				\|\nabla u\|_{L^\infty(\Omega_{c \sqrt\varepsilon})} \ge \frac{1}{C}\varepsilon^{\frac{\alpha-1}{2}}, 
			\end{equation}
			where $\alpha$ is given in \eqref{alpha_n}.
		\end{itemize}
	\end{theorem}
	
	\begin{remark}
		As $k_1,k_2 \to \infty$, Theorem \ref{thm_upperbound} recovers our previous result \cite{DongYangZhu26}*{Theorem 1.4}. In particular, estimates \eqref{grad_u_upperbound3} and \eqref{grad_u_lower_bound_new} reveal a new blow-up phenomenon at the critical bonding parameter, caused by the resonance between two inclusions of finite conductivity.
		Moreover, the constant $C$ in \eqref{grad_u_upperbound1} depends on $\gamma$ through $\gamma_0$ and the positive lower bound of $1- \gamma\mu$, while the constant in \eqref{grad_u_upperbound2} depends on $\gamma$ only through the positive lower bound of $\gamma\mu - 1$ { but does not depend on the upper bound of $\gamma$}. As a consequence, estimate \eqref{grad_u_upperbound2} { is consistent with} the optimal gradient estimate for the insulated problem obtained in \cite{DLY} as $\gamma \to \infty$. We note that the discussion on $\gamma$ dependence does not apply to Theorems \ref{thm_upperbound_3} and \ref{thm_upperbound_2} below.
	\end{remark}
	
	We also define
	\begin{equation}
		\label{alpha_n_zero}
		\alpha_0=\alpha_0(n,\gamma, \mu):= \frac {-(n-1)+\sqrt{(n-1)^2+{8/(\mu\gamma)}}}{2}.
	\end{equation}
	Note that $0<\alpha_0\le \alpha$, $\alpha_0=\alpha$ if and only if $n=2$, and $\alpha_0 <1$ if and only if $\gamma >  2/(n\mu)$.
	
	In two dimensions, the optimal estimates in Theorem \ref{thm_upperbound} actually hold {for general domains and boundary data} without assumptions on symmetry.
	
	\begin{theorem}\label{thm_upperbound_3}
		For $n=2$, $\varepsilon \in (0,R_0/4)$, $\gamma > 0$, $\mu > 0$, let $h_1,h_2$ satisfy \eqref{fg_0}, \eqref{def:c_2}, \eqref{fg_2}, and \eqref{fg_3}, and let $u$ be the solution of \eqref{robin}. 
		{ Then the assertions of {Theorem} \ref{thm_upperbound}} hold.
	\end{theorem}
	
	In three and higher dimensions, we have the following estimates for general domains and boundary data, { but optimality is unclear.}
	\begin{theorem}\label{thm_upperbound_2}
		For $n \ge 3$, $\varepsilon \in (0,R_0/4)$, $\gamma > 0$, $\mu > 0$, let $h_1,h_2$ satisfy \eqref{fg_0}, \eqref{def:c_2}, \eqref{fg_2}, and \eqref{fg_3}, and let $u$ be the solution of \eqref{robin}. Then there exists a positive constant $C$, depending only on $n$, $\gamma$, $k_0$, $\mu$, $R_0$, $\|h_1\|_{C^{2,\sigma}}$, and $\|h_2\|_{C^{2,\sigma}}$, such that 
		\begin{itemize}
			\item when $0 < \gamma \le 2/(n\mu)$,
			\begin{equation}\label{grad_u_upperbound1_zero}
				|\nabla u(x)| \le C \|\varphi\|_{L^\infty(\partial\Omega)}\quad \mbox{in}~ \Omega_{R_0/2},
			\end{equation}
			\item when $\gamma >2/(n\mu)$,
			\begin{equation}\label{grad_u_upperbound2_zero}
				|\nabla u(x)| \le C \|\varphi\|_{L^\infty(\partial\Omega)}(\varepsilon + |x'|^2)^{\frac{\alpha_0-1}{2}} \quad \mbox{in}~  \Omega_{R_0/2},
			\end{equation}
		\end{itemize}
		where $\alpha_0$ is given in \eqref{alpha_n_zero}.
	\end{theorem}
	These estimates remain bounded at $\gamma=2/(n\mu)$ without a logarithmic loss. However, it is not clear to us whether the estimates are optimal or whether the threshold is indeed $2/(n\mu)$ in this case.

    \subsection{Novelty, main difficulties, and organization}
	The main novelty is that finite internal conductivities do not eliminate field concentration for LC-type imperfect interfaces and can produce the critical logarithmic blow-up. Compared with \cites{DongYangZhu26}, where the inclusions are perfect conductors so the interior potentials are constants, the present problem involves a coupled elliptic system for the interior and exterior potentials. The interior traces provide unknown, nonconstant Robin data for the exterior problem, and their regularity must be controlled simultaneously with the exterior solution. A key new ingredient is the derivation of local gradient estimates adapted to this coupling. A direct application of the reflection argument in \cites{DongYangZhu26} introduces huge jumps in the conormal derivative arising from the nonconstant Robin data. We overcome this difficulty by first subtracting an affine function to normalize the data at a reference point and then absorbing the jumps produced by reflection into a divergence-form forcing term. Moreover, we exploit the negative leading part of the zeroth-order coefficient generated by the Robin-to-conormal transformation to derive an estimate in terms of the oscillation of the solution. Together with control of the zero spherical harmonic mode, this stronger estimate yields optimal estimates without symmetry assumptions in two dimensions and upper bounds for general domains and data in higher dimensions, whose optimality remains open. The coupling also introduces a linear forcing term in the reduced equation that resonates with the first spherical harmonic mode at the critical value $\gamma=1/\mu$, producing the optimal logarithmic blow-up absent in \cites{DongYangZhu26}.
	
	The rest of the paper is organized as follows. Section \ref{sec:prelim} establishes the variational formulation, the maximum principle, and preliminary regularity estimates. Sections \ref{sec:local-estimates}
and \ref{sec4} develop the crucial local gradient estimates and the dimensional reduction argument, respectively, which are used in Section \ref{sec:upper-bounds} to prove the upper gradient bounds. Section \ref{sec:two-dimensional} treats the two-dimensional case without symmetry assumptions, while Section \ref{sec:optimality} establishes the optimality of the blow-up rates. Section \ref{sec:variable-bonding} discusses extensions to variable bonding parameters. Appendix \ref{sec_appendix} contains the construction and estimates for the normal-preserving flattening map used in the local analysis.

	\section{\texorpdfstring{{Preliminaries}}{Preliminaries}}\label{sec:prelim}
	
	The problem \eqref{robin} admits a variational structure. We shall show that the piecewise $H^1$ solution $u$ to equation \eqref{robin} is indeed the unique minimizer of a functional in an appropriate function space: $I[u] = \min_{v \in \sA}I[v]$, where
	\begin{equation}\label{minimizer}
		\begin{aligned}
			I[v] &:= \int_{\widetilde\Omega} |\nabla v_+|^2+ \sum_{j=1}^2 \int_{{\cD_j}} k_j |\nabla v_-|^{2} + \sum_{j=1}^2 \gamma^{-1}\int_{\partial \cD_j} |v_+ - v_-|^2,\\
			\sA &:= \{ v~ \mbox{piecewise}~ H^{1}:\, v_+ = \varphi~~\mbox{on}~~\partial\Omega \}.
		\end{aligned}
	\end{equation}
	
	{We first prove the maximum principle} for the equation \eqref{robin}.

	\begin{lemma}\label{lem:maximum_principle}
		Let $n\ge 2$, $\Omega$, $\cD_1$, and $\cD_2$ be Lipschitz domains in $\bR^n$. Let $u$ be a piecewise $H^1$ subsolution to the equation \eqref{robin}, that is,
		$u_+\in H^1(\widetilde\Omega)$ and $u_- \in H^1(\cD_1 \cup \cD_2)$ satisfy
		\begin{equation}\label{subrobin}
			\begin{cases}
				\Delta u_+ \ge 0&\mbox{in}~\widetilde\Omega,\\
				\Delta u_- \ge 0&\mbox{in}~\mathcal{D}_1\cup\mathcal{D}_2,\\ 
				\partial_\nu u_+ \le  -\gamma^{-1}(u_+-u_-)&\mbox{on}~\partial \mathcal{D}_i, \quad i=1,2,\\
				k_i \partial_\nu u_- \ge -\gamma^{-1}(u_+-u_-)&\mbox{on}~\partial \mathcal{D}_i, \quad i=1,2,\\
				u_+ \le \varphi&\mbox{on}~\partial\Omega.
			\end{cases}
		\end{equation}
		Then it holds that
		\begin{equation*}
			u \le \sup_{\partial \Omega} \varphi \quad \mbox{in } \Omega.
		\end{equation*}
	\end{lemma}
	
	\begin{proof}
		By considering $u_\pm -\sup_{\partial \Omega} \varphi$ in place of $u_\pm$, we may assume $u_+\le 0$ on $\partial\Omega$, and it suffices to show that
		\begin{equation}\label{v_upper}
			u \le 0 \quad \mbox{in } \Omega.
		\end{equation}
		We denote $v_1:=\max\{u_+,0\}$ and $v_2:=\max\{u_-,0\}$. We multiply the first equation in \eqref{subrobin} by $v_1$ and integrate by parts in $\widetilde\Omega$. We have
		\begin{align*}
			0& \ge \int_{\widetilde\Omega} |\nabla v_1|^{2}- \sum_{j=1}^2 \int_{\partial \cD_j} v_1\,\partial_\nu u_+ \\
			& \ge \int_{\widetilde\Omega} |\nabla v_1|^{2}+ \gamma^{-1}\sum_{j=1}^2 \int_{\partial \cD_j} (v_1^2 -v_1 u_- )\\
			& \ge \int_{\widetilde\Omega} |\nabla v_1|^{2} +  \gamma^{-1}\sum_{j=1}^2 \int_{\partial \cD_j} (v_1^2 -v_1 v_2).
		\end{align*}
		We then multiply the second equation in \eqref{subrobin} by $k_j v_2$ and integrate by parts in $\cD_j$, $j = 1,2$. We have
		\begin{align*}
			0& \ge \sum_{j=1}^2 \int_{\cD_j} k_j |\nabla v_2|^{2} + \sum_{j=1}^2 k_j \int_{\partial \cD_j} v_2\,\partial_\nu u_- \\
			& \ge \sum_{j=1}^2 \int_{\cD_j} k_j |\nabla v_2|^{2}+ \gamma^{-1}\sum_{j=1}^2 \int_{\partial \cD_j} (v_2^2 -v_2 u_+) \\
			& \ge \sum_{j=1}^2 \int_{\cD_j} k_j |\nabla v_2|^{2} +  \gamma^{-1}\sum_{j=1}^2 \int_{\partial \cD_j} (v_2^2 -v_1 v_2).
		\end{align*}
		Summing up the two inequalities above yields
		$$
		\int_{\widetilde\Omega} |\nabla v_1|^{2} + \sum_{j=1}^2 \int_{\cD_j} k_j |\nabla v_2|^{2} +  \gamma^{-1}\sum_{j=1}^2 \int_{\partial \cD_j} (v_1- v_2)^2 \le 0.
		$$
		Therefore,
		$$
		\nabla v_1 \equiv 0 ~\mbox{in}~\widetilde\Omega, \quad \nabla v_2 \equiv 0 ~\mbox{in}~\cD_1 \cup \cD_2, \quad v_1 \equiv v_2 ~\mbox{on}~ \partial \cD_1 \cup \partial \cD_2.
		$$
		Since $v_1 = 0$ on $\partial\Omega$, it follows that $v_1 \equiv 0$ in $\widetilde \Omega$, and hence $v_2 \equiv 0$ in $\cD_1 \cup \cD_2$. Therefore, \eqref{v_upper} is proved.
	\end{proof}
	
	Using the maximum principle Lemma \ref{lem:maximum_principle}, we can show that the equation \eqref{robin} is equivalent to the variational problem \eqref{minimizer}.
	
	\begin{lemma}
		Let $n\ge 2$, $\Omega$, $\cD_1$, and $\cD_2$ be Lipschitz domains in $ \bR^n$, {and let $\varphi\in H^{1/2}(\partial\Omega)$. Then} $u$ is the minimizer of \eqref{minimizer} if and only if $u$ is a piecewise $ H^1$ solution to \eqref{robin}. Therefore, there exists a unique piecewise $ H^1$ solution $u$ to the equation \eqref{robin}.
	\end{lemma}
	
	\begin{proof}
		First, by Lemma \ref{lem:maximum_principle}, we know that the equation \eqref{robin} has at most one solution. {For $\varphi\in H^{1/2}(\partial\Omega)$, a trace lifting makes $\sA$ nonempty. The boundary Poincar\'e and trace inequalities, together with the jump terms, 
        give coercivity on variations with zero outer trace. Weak lower semicontinuity yields a minimizer by the direct method, and strict positivity of the quadratic form on nonzero variations gives uniqueness.} Therefore, it suffices to show that the minimizer $u$ of \eqref{minimizer} satisfies \eqref{robin}.
		
		By a standard argument, one can see that the minimizer $u$ of \eqref{minimizer} is harmonic in $\widetilde\Omega$ and $\cD_1 \cup \cD_2$. It remains to show that $u$ satisfies the transmission conditions \eqref{robin}$_3$.
		
		We choose an arbitrary piecewise $H^1$ function $v$ such that $v = 0$ in $\cD_1 \cup \cD_2$ and on ${\partial (\Omega\setminus \overline \cD_2)}$. Then the equation
		\begin{equation*}
			0 = \left.\frac{d}{dt} \right|_{t = 0}  I[u + t v]
		\end{equation*}
		{gives}
		$$
		0 = \int_{\widetilde\Omega} \nabla u_+ \cdot \nabla v_+ + \gamma^{-1} \int_{\partial \cD_1} (u_+ - u_-) v_+.
		$$
		Using the harmonicity of $u_+$ and integration by parts, we have
		$$
		0 = \int_{\partial \cD_1} \big( \partial_\nu u_+ + \gamma^{-1}  (u_+ - u_-) \big) v_+,
		$$
		which implies 
		$$
		\partial_\nu u_+ =  -\gamma^{-1}(u_+-u_-)\qquad\mbox{on}~\partial \mathcal{D}_1.
		$$
		The rest of the transmission conditions follow similarly.
	\end{proof}
	
	Next, we show that the solution is H\"older continuous in the inclusions $\cD_1$ and $\cD_2$, and its gradient is also H\"older continuous outside of the narrow region.
	
	\begin{lemma}
		Let $n\ge 2$, $\beta \in (0,1)$, $\Omega$, $\cD_1$, and $\cD_2$ be $C^{1,\beta}$ domains in $\bR^n$. Let { $\varphi\in C^{1,\beta}(\partial \Omega)$ and} $u$ be the piecewise $H^1$ solution to the equation \eqref{robin}. Then there exists a positive constant $C$, depending only on $n$, $\beta$, $\gamma_0$, $R_0$, $\cD_1$, $\cD_2$, and  $\|\partial \Omega\|_{C^{1,\beta}}$, such that 
		\begin{align}
			\label{u_C0}
			\|u\|_{L^\infty(\Omega)}  & \le \|\varphi\|_{L^\infty(\partial \Omega)},\\
			\label{u_holder_inside}
			[u]_{C^\beta(\mathcal{D}_i)} &\le \frac{C}{k_i}\|\varphi\|_{L^\infty(\partial \Omega)}, \quad i = 1,2,\\
			\label{u_C1_outside}
			\|u\|_{{C^{1,\beta}}(\widetilde\Omega \setminus \Omega_{R_0/2})} &\le C \Big( 1 + \frac{1}{k_1} + \frac{1}{k_2}\Big)\|\varphi\|_{C^{1,\beta}(\partial \Omega)}.
		\end{align}
		
	\end{lemma}
	
	\begin{proof}
		First, \eqref{u_C0} simply follows from the maximum principle Lemma \ref{lem:maximum_principle}. For $i = 1, 2$, $u_-$ satisfies
		$$
		\begin{cases}
			\Delta u_- = 0&\mbox{in}~\mathcal{D}_i,\\ 
			\partial_\nu u_- = -(k_i\gamma)^{-1}(u_+-u_-)&\mbox{on}~\partial \mathcal{D}_i.
		\end{cases}
		$$
		By the $W^{1,p}$ estimate for the Neumann boundary problem (see \cite{DongLi21}*{Theorem 2.6}), we have
		$$
		\| u_- - (u_-)_{\cD_i} \|_{W^{1,p}(\cD_i)} \le  \frac{C}{k_i} \| u_+ - u_-\|_{ L^{p} (\partial \cD_i)},
		$$
		which implies \eqref{u_holder_inside} by \eqref{u_C0} and Morrey's inequality. Recall that $u_+$ satisfies
		$$
		\begin{cases}
			\Delta u_+ = 0&\mbox{in}~\widetilde\Omega,\\ 
			\partial_\nu u_+ = {-}\gamma^{-1}(u_+-u_-)&\mbox{on}~\partial \mathcal{D}_1 \cup \partial \mathcal{D}_2,\\
			u_+ = \varphi &\mbox{on}~\partial \Omega.
		\end{cases}
		$$
		Away from the narrow gap region, one can apply a similar { local $W^{1,p}$ estimate} to see that
		$$
		[u_+]_{{C^{\beta}}(\widetilde\Omega \setminus \Omega_{R_0/4})} \le C \| \varphi\|_{C^\beta (\partial \Omega)}.
		$$
		Then one can apply the $C^{1,\beta}$ estimate for the Neumann boundary problem (see \cite{Lieberman87}*{Theorem 1.2}) to obtain
		$$
		\|u\|_{{C^{1,\beta}}(\widetilde\Omega \setminus \Omega_{R_0/2})} \le C \Big( [u_+]_{C^\beta(\widetilde\Omega \setminus \Omega_{R_0/4})} + {\|u_-\|_{C^\beta(\mathcal{D}_1)} + \|u_-\|_{C^\beta(\mathcal{D}_2)}} + \| \varphi\|_{C^{1,\beta} (\partial \Omega)} \Big),
		$$
		which implies \eqref{u_C1_outside}.
	\end{proof}

	\section{\texorpdfstring{$\varepsilon^{-1/2}$}{epsilon to the minus one-half} estimate}\label{sec:local-estimates}
	
	{The main result of this section is the following result on the $C^{1,\beta}$ estimate of solutions to \eqref{main_problem_narrow} in a narrow region.}
	
	\begin{theorem}
		\label{thm-1/2}
		Let $h_1$ and $h_2$ be $C^{1,1}$ functions satisfying \eqref{fg_0}--\eqref{def:c_2}, $n \ge 2$, $\gamma \ge \gamma_0 > 0$, $\varepsilon \in(0,R_0/4)$, and $u \in H^1({\Omega}_{R_0})$ be a solution of \eqref{main_problem_narrow} for some  functions $g_1,g_2 \in C^{\beta}_{x'}$ for some $\beta \in (0,1)$. Then there exists  a constant $C_1>0$ depending only on $n$, $\gamma_0$, $\beta$, $R_0$, $c_1$, and $c_2$, and a constant $C_2 >0$ further depending on $\gamma$, such that for any $x_0\in \Omega_{R_0/2}$ and $r=\frac{1}{4}(\varepsilon+|x'_0|^2)^{\frac{1}{2}}$, we have
        		\begin{equation}\label{gradient-1/2a}
			\begin{aligned}
				&
                r\|\nabla u\|_{L^\infty(\Omega_{x_0,r})} + r^{1+\beta} [\nabla u]_{C^\beta(\Omega_{x_0,r})} \\
				& \le  C_1 \Big( \dashnorm{u-(g_1(0') + g_2(0'))/2}{L^2(\Omega_{x_0,2r})}  + r|g_1(0') - g_2(0')|\\ 
				&+ r^{\beta}([g_1]_{C^\beta_{x'}(\Gamma^+_{x_0,4r})} + [g_2]_{C^\beta_{x'}(\Gamma^-_{x_0,4r})}) 
				\Big),
			\end{aligned}
		\end{equation}
        and 
		\begin{equation}\label{gradient-1/2}
			\begin{aligned}
				&r\|\nabla u\|_{L^\infty(\Omega_{x_0,r})} + r^{1+\beta} [\nabla u]_{C^\beta(\Omega_{x_0,r})} \\
				& \le  C_2 \Big( \dashnorm{u- (u)_{\Omega_{x_0,2r}}}{L^2(\Omega_{x_0,2r})}  + r|g_1(0') - g_2(0')|\\ 
				&+ r^{\beta}([g_1]_{C^\beta_{x'}(\Gamma^+_{x_0,4r})} + [g_2]_{C^\beta_{x'}(\Gamma^-_{x_0,4r})}) 
				{+r\|\partial_{x_n} u\|_{L^\infty(\Omega_{x_0,2r})}}
				\Big),
			\end{aligned}
		\end{equation}
		where $\Gamma^\pm_{x_0,4r} := \Gamma_\pm \cap \overline{\Omega_{x_0,4r}}$.
	\end{theorem}
	
		The proof is based on a local flattening and reflection argument in the spirit of \cite{DongYangZhu26}. A direct adaptation, however, is obstructed by the inhomogeneous Robin data: after the boundary conditions are converted into co-normal form, reflection would produce jumps that are too large to control. To overcome this difficulty, we first subtract an affine function $Ax_n+B$ so that {the new Robin data vanish} at $x_0'$. After flattening the boundaries, a suitable exponential change of variables is used to convert the Robin conditions into co-normal conditions and introduce a favorable negative zeroth-order coefficient with a Lipschitz leading term $c_L$. We then reflect evenly and periodically across the flattened boundaries and, following the idea of \cites{DongXu1,DongXu2}, absorb the remaining jumps of the co-normal derivative into a divergence-form forcing. The desired estimates follow from an interior $C^{1,\beta}$ estimate for laminate domains established in \cite{Dong2012}. The negative zeroth-order coefficient also yields an important improvement over \cite{DongYangZhu26}: it allows us to prove the estimate \eqref{gradient-1/2} with the oscillation of $u$ on the right-hand side {and an additional vertical derivative which will be shown to be bounded}. With this stronger estimate, we can treat more general domains and boundary data without imposing symmetry assumptions (see Theorems \ref{thm_upperbound_3} and \ref{thm_upperbound_2}).

	\begin{proof}
		It suffices to prove estimates \eqref{gradient-1/2a} and \eqref{gradient-1/2} for $r < r_0$, where $r_0$ is a small $\varepsilon$-independent constant to be specified during the proof.
		Without loss of generality, we assume $g_2(x_0')=-g_1(x_0')$, since otherwise we can consider the equation for $u-(g_1(x_0')+g_2(x_0'))/2$ instead of $u$.
		
		\textbf{Step 1:} Subtract a linear function in $x_n$.
		
		{Fix $x_0 \in \Omega_{R_0/2}$ and define}
		$$
		\hat{u}:= u - (A x_n + B).
		$$
		The first step is to find $A$ and $B$ such that $\hat u + \gamma \partial_\nu \hat u = 0$ at $x' = x_0'$ on $\Gamma_\pm$. By direct computations, we have
		on $\Gamma_+$,
		\begin{align*}
			\hat{u} + \gamma \partial_\nu \hat{u} &= g_1(x') - A \left( \frac{\varepsilon}{2} + h_1(x') + \frac{\gamma}{\sqrt{1 + |\nabla h_1|^2}} \right) - B;
		\end{align*}
		on $\Gamma_-$,
		\begin{align*}
			\hat{u} + \gamma \partial_\nu \hat{u} &= g_2(x') - A \left( -\frac{\varepsilon}{2} + h_2(x') - \frac{\gamma}{\sqrt{1 + |\nabla h_2|^2}} \right) - B.
		\end{align*}
		Set
		$$
		\left\{
		\begin{aligned}
			A \left( \frac{\varepsilon}{2} + h_1(x_0') + \frac{\gamma}{\sqrt{1 + |\nabla h_1(x_0')|^2}} \right) {+} B &= g_1(x_0'),\\
			A \left( -\frac{\varepsilon}{2} + h_2(x_0') - \frac{\gamma}{\sqrt{1 + |\nabla h_2(x_0')|^2}} \right) {+} B &= {g_2(x_0')}.
		\end{aligned}
		\right.
		$$
		Then one can solve for $A$ and $B$ explicitly, and see that
		\begin{equation}\label{est:ab}
			|A| \le C |g_1(x_0')|, \quad |B| \le C |x_0'|^2 |g_1(x_0')|.
		\end{equation}
		Thus, we conclude that $\hat{u}$ satisfies
		\begin{equation}\label{main_problem_narrow_2}
			\begin{cases}
				\Delta \hat u =0&\mbox{in}~\Omega_{R_0},\\
				\hat u + \gamma \partial_\nu \hat u = \tilde g_1&\mbox{on}~\Gamma_+,\\
				\hat u + \gamma \partial_\nu \hat u = \tilde g_2&\mbox{on}~\Gamma_-,
			\end{cases}
		\end{equation}
		with $\tilde g_1(x_0') = \tilde g_2(x_0') = 0$.

		\textbf{Step 2:} Flatten the boundaries.
		
		For any $x_0 \in \Omega_{R_0/2}$, we denote
		$$r:=\frac{1}{4}(\varepsilon+|x_0'|^2)^{1/2},$$
		and
		$$
		Q_{s,t} := \{ y = (y', y_n) \in \bR^n:~ |y'-x_0'| < s,
        \,|y_n| < t \},
		$$
		Using the flattening map $\Phi$ defined in \eqref{Phi} {which satisfies
$\Phi(Q_{1.9r, r^2})\subset \Omega_{x_0,2r}$ and $\Omega_{x_0,r} \subset \Phi(Q_{1.1r, r^2})$}, and defining $\tilde{u}(y) = \hat u(\Phi(y))$, where $ \hat u$ is a solution of \eqref{main_problem_narrow_2}, we know that $\tilde{u}(y)$ satisfies the equation
		\begin{equation}\label{tilde_u_equation}
			\left\{
			\begin{aligned}
				\partial_{y_i}(a^{ij}\partial_{y_j}\tilde u)&=0  &&\mbox{in }Q_{1.9r, r^2},\\
				\tilde u+\gamma \,F_1^{-1} \partial_{y_n}\tilde u&= \tilde g_1 &&\mbox{on } \{y_n=r^2\}\cap \overline{Q_{1.9r, r^2}},\\
				\tilde u-\gamma \, F_2^{-1} \partial_{y_n}\tilde u&= \tilde g_2 &&\mbox{on } \{y_n=-r^2\}\cap \overline{Q_{1.9r, r^2}},
			\end{aligned}
			\right.
		\end{equation}
		where 
		$a:=(a^{ij})_{n\times n}=(\nabla \Phi)^{-1}((\nabla\Phi)^{-1})^T{|\det(\nabla \Phi)|}\in C^{0,1}(\overline{Q_{1.9r, r^2}}:\bR^{n\times n})$ satisfies
		\begin{equation}\label{est:a}
			\frac{I}{C} \le a \le C I, \quad |\nabla_y a| \le \frac{C}{r},
		\end{equation}
		and
		\begin{equation}\label{ortho}
			{a^{nj}=a^{jn}}=0 \quad \mbox{on } \{y_n=\pm r^2\}\cap \overline{Q_{1.9r, r^2}}
		\end{equation}
		for any  $j\in \{1,2,\ldots, n-1\}$,
		and $F_i\in C^{0,1}(\overline{Q_{1.9r, r^2}})\cap C^{1,1}(Q_{1.9r, r^2})$ ($i=1,2$) are defined as
		\begin{equation*}
			\left\{
			\begin{aligned}
				&F_1(y):=\frac{1}{2r^2}(\varepsilon+h_1(y')-h_2(y'))\sqrt{1+|\nabla_{y'} \widehat h_1(y)|^2},\\
				&F_2(y):=\frac{1}{2r^2}(\varepsilon+h_1(y')-h_2(y'))\sqrt{1+|\nabla_{y'} \widehat h_2(y)|^2},
			\end{aligned}
			\right.
		\end{equation*}
		$\widehat{h}_i$ denotes the mollification defined in \eqref{mollification}. It follows from \eqref{est:mollification} that 
		\begin{equation}\label{F_bound}
			\left\{
			\begin{aligned}
				&\frac{1}{C}\le F_i \le C, \quad |\nabla_{y'} F_i|\le \frac{C}{r}, \quad |\partial_{y_n} F_i|\le C\, r^2,\\
				& |\nabla^2_{y'}F_i(y)|\le \frac{C\,r^2}{r^4-y_n^2},\quad |\partial_{y_n}\nabla_{y'}F_i(y)|\le \frac{C\,r^3}{r^4-y_n^2}, \quad 
				|\partial^2_{y_n}F_i(y)|\le \frac{C\,r^4}{r^4-y_n^2},
			\end{aligned}
			\right.
		\end{equation}
		Furthermore, we have
		\begin{equation}\label{est:a2}
			{a^{ii} \sim 1, \qquad |a^{ij}|} \le Cr \quad \mbox{for}~i \neq j. 
		\end{equation}
		
		\textbf{Step 3:} {Reduce to conormal boundary conditions}.
		
		Let $\tilde v= \tilde u \, e^{\psi}$, where $\psi=\psi(y)$ is a function to be determined. Note that by the chain rule, 
		$$\partial_{y_n}\tilde v=e^{\psi}(\partial_{y_n}\tilde u+\tilde u\,\partial_{y_n} \psi).$$
		We choose $$\psi=(y_n+r^2)(y_n-r^2)(\tilde{A}\,y_n+\tilde{B}),$$ where $\tilde A$ and $\tilde B$ are functions such that
		\begin{equation*}
			\left\{
			\begin{aligned}
				&\partial_{y_n}\psi(y',r^2)=2r^2(\tilde A r^2+\tilde B)=\gamma^{-1} F_1,\\
				&\partial_{y_n}\psi(y',-r^2)=-2r^2(-\tilde A r^2+\tilde B)=-\gamma^{-1} F_2,
			\end{aligned}
			\right.
		\end{equation*}
		namely, 
		$$\tilde A:=(4\gamma r^4)^{-1}(F_1-F_2) \quad \mbox{and}\quad \tilde B:=(4\gamma r^2)^{-1}(F_1+F_2).$$
		The equation \eqref{tilde_u_equation} becomes
		\begin{equation}\label{tilde_v_equation}
			\left\{
			\begin{aligned}
				\partial_{y_i}(a^{ij}\partial_{y_j}\tilde v) + b^i \partial_{y_i}\tilde v + c \tilde v &=0  &&\mbox{in }Q_{1.9r, r^2},\\
				\partial_{y_n}\tilde v&= e^\psi \frac{F_1}{\gamma} \tilde g_1 &&\mbox{on } \{y_n=r^2\}\cap \overline{Q_{1.9r, r^2}},\\
				\partial_{y_n}\tilde v&= -e^\psi \frac{F_2}{\gamma} \tilde g_2 &&\mbox{on } \{y_n=-r^2\}\cap \overline{Q_{1.9r, r^2}},
			\end{aligned}
			\right.
		\end{equation}
		where
		$$
		b^i = - 2a^{ij} \partial_{y_j} \psi,
		$$
		and
		$$
		c = -(\partial_{y_i}a^{ij})(\partial_{y_j} \psi) + a^{ij} (\partial_{y_i} \psi) (\partial_{y_j} \psi) - a^{ij} \partial_{y_i y_j} \psi.
		$$
		In light of \eqref{ortho}, the boundary conditions in \eqref{tilde_v_equation} can also be written as
		$$
		a^{nj} \partial_{y_j}\tilde v =  a^{nn} e^\psi \frac{F_1}{\gamma} \tilde g_1 \quad \mbox{on } \{y_n=r^2\}\cap \overline{Q_{1.9r, r^2}},
		$$
		and
		$$
		a^{nj} \partial_{y_j}\tilde v =  -a^{nn} e^\psi \frac{F_2}{\gamma} \tilde g_2 \quad \mbox{on } \{y_n=-r^2\}\cap \overline{Q_{1.9r, r^2}}.
		$$
		{Following the same computations} as in \cite{DongYangZhu26}, we have
		\begin{equation}\label{est:psi}
			|\psi| \le Cr^2, \quad |\nabla_y \psi| \le C, \quad |\nabla_{y'}^2 \psi| \le C, \quad |\nabla_{y'}\partial_{y_n} \psi| \le Cr^{-1}, \quad |\partial_{y_n}^2 \psi| \le C r^{-2}.
		\end{equation}
		Furthermore, by a direct computation, we have
		$$
		{ \partial_{y_n}^2 \psi  = 6 \tilde A y_n + 2\tilde B + 4y_n (y_n\partial_{y_n} \tilde A  + \partial_{y_n} \tilde B) + (y_n^2 - r^4) (2 \partial_{y_n} \tilde A + \partial_{y_n}^2 \tilde A y_n + \partial_{y_n}^2 \tilde B)}.
		$$
		{
			By \eqref{est:a}, \eqref{F_bound}, \eqref{est:a2}, and \eqref{est:psi}, we can conclude that the leading order term from $c$, denoted by $c_L$, is
			$$
			c_L:= - 2 a^{nn} \tilde B.
			$$
			It satisfies
			$$
			{\frac{1}{C\gamma r^2} \le -c_L \le \frac{C}{\gamma r^2}} \quad \mbox{and} \quad |\nabla c_L| \le \frac{C}{\gamma r^3} .
			$$
			The remaining terms from $c$, denoted by $c_R:= c- c_L$, satisfy
			$$
			|c_R| \le Cr^{-1}.
			$$
			Furthermore, direct computation yields
			$$
			|b| \le C.
			$$
		}
		
		\textbf{Step 4:} Extend and remove the discontinuities of the co-normal boundary conditions.
		
		We define
		$$
		\cQ_k : =  \{ y = (y', y_n) \in \bR^n:~ |y'-x_0'| < 1.9r,\, (2k-1)r^2 < y_n < (2k+1)r^2 \}
		$$
		for $k \in \bZ$,
		and denote the upper boundary of $\cQ_k$ by
		$$
		\Gamma_k := \{ y = (y', y_n) \in \bR^n:~ |y'-x_0'| < 1.9r,\, y_n = (2k+1)r^2 \}.
		$$
		Now for $i,j = 1,2,\ldots, n-1$, we extend $ a^{ij}$, $ a^{nn}$, $ b^i$, $c$, and $\tilde{v}$ evenly with respect to $y_n = r^2$, $ a^{in}$, $ a^{ni}$, and $b^n$ oddly with respect to $y_n = r^2$, and then periodically along $y_n$ so that the period is equal to $4r^2$. Such extension creates discontinuities of the co-normal derivatives across interfaces $\Gamma_k$. We adapt the idea in \cites{DongXu1,DongXu2} to remove the discontinuities. We denote
		$$
		[a^{nj} \partial_{y_j}\tilde v]_{\Gamma_k}:= a^{nj} \partial_{y_j}\tilde v|_{\Gamma_k}^+ - a^{nj} \partial_{y_j}\tilde v|_{\Gamma_k}^-
		$$
		to be the difference between the {upper and lower traces} of the co-normal derivative of $\tilde{v}$ across $\Gamma_k$. Then it follows from \eqref{tilde_v_equation} that $\tilde{v}$ satisfies
		\begin{equation}\label{tilde_v_equation_extend}
			\left\{
			\begin{aligned}
				\partial_{y_i}(a^{ij}\partial_{y_j}\tilde v) + b^i \partial_{y_i}\tilde v + c \tilde v &=0  &&\mbox{in }\cup_{k=- \infty}^\infty \cQ_k,\\
				[a^{nj} \partial_{y_j}\tilde v]_{\Gamma_k}&= \cG_k &&\mbox{on } \Gamma_k,
			\end{aligned}
			\right.
		\end{equation}
		where
		$$
		\cG_k(y') = \left\{
		\begin{aligned}
			&-2a^{nn} e^\psi \frac{F_1}{\gamma} \tilde g_1, &&\mbox{when}~k~\mbox{is even},\\
			&-2a^{nn} e^\psi \frac{F_2}{\gamma} \tilde g_2, &&\mbox{when}~k~\mbox{is odd}.
		\end{aligned}
		\right.
		$$
		For $m \sim \frac{1}{r}$, we define
		$$
		f^i(y):=  \delta_{in} \sum_{k=-m}^m \mathbbm{1}_{\{y_n > (2k{+}1)r^2\}} \cG_k(y') \quad \mbox{for}~i=1,\ldots,n,
		$$
		where $\delta$ is the  Kronecker delta and $\mathbbm{1}_{\bullet}$ is the indicator function. Then it follows from \eqref{tilde_v_equation_extend} that $\tilde{v}$ satisfies
		\begin{equation}\label{tilde_v_equation_extend_new}
			\left\{
			\begin{aligned}
				\partial_{y_i}(a^{ij}\partial_{y_j}\tilde v - f^{i}) + b^i \partial_{y_i}\tilde v + c \tilde v &=0  &&\mbox{in }\cup_{k=-m}^m \cQ_k,\\
				[a^{nj} \partial_{y_j}\tilde v - f^{{n}}]_{\Gamma_k}&= 0 &&\mbox{on } \Gamma_k.
			\end{aligned}
			\right.
		\end{equation}

	\textbf{Step 5:} Rescale and estimate.
		
		We consider the following rescaling
		{
			$$
			\begin{aligned}
				&\bar{v}(y)=\tilde v(x_0'+ry',ry_n), \quad \bar{a}(y)=a(x_0'+ry',ry_n), \quad \bar{b}(y)=r b(x_0'+ry',ry_n),\\
				&\bar{c}(y)=r^2 c(x_0'+ry',ry_n), \quad \bar{c}_L(y)=r^2 c_L(x_0'+ry',ry_n), \\
				&\bar{c}_R(y)=r^2 c_R(x_0'+ry',ry_n), \quad
				\bar{f}(y)=r f(x_0'+ry',ry_n).
			\end{aligned}
			$$
			Then by \eqref{tilde_v_equation_extend_new} $\bar{v}$ satisfies the equation
			\begin{equation}\label{eqn:extend}
				\partial_{y_i}(\bar a^{ij}\partial_{y_j}\bar v - \bar f^i) +\bar b^i \partial_{y_i}\bar v + \bar c \bar v =0  \quad \mbox{in}~ {B_{1.9}}:=\{ |y|< 1.9 \},
			\end{equation}
			with the coefficients satisfying
			\begin{equation}\label{est:coefficient}
				\begin{aligned}
					&\frac{I}{C} \le \bar a \le C I, \quad |\nabla_y \bar a| \le C, \quad |\bar{b}| \le Cr, \\
					&\frac{1}{C\gamma} \le -\bar{c}_L \le \frac{C}{\gamma}, \quad |\bar c_R| \le Cr, \quad \mbox{and} \quad |\nabla_y \bar c_L| \le \frac{C}{\gamma}.
				\end{aligned}
			\end{equation}
			By \cite{Dong2012}*{Theorem 2 and Remark 3}, we have
			\begin{equation*}\label{est:2012a}
				\begin{aligned}
					&
                    \| \bar{v}\|_{W^{1,\infty}({B_{1.1}})} 
					+ [\nabla \bar{v}]_{C^\beta_{{y'}}(B_{1.1})}
					+ [\nabla \bar{v}]_{C^\beta({\cQ_0})}
					\le C(\|\bar{v} \|_{L^2(B_{1.9})} + \| \bar f \|_{C_{y'}^\beta(B_{1.9})})\\
					&\le C (\|\bar{v} \|_{L^2(B_{1.9})} + \| \bar f \|_{C_{y'}^\beta(B_{1.9})}.
				\end{aligned}
                \end{equation*}
			Since $\bar{f}(0',y_n) = 0$, we have
			\begin{equation}\label{est:f_bar}
				\| \bar{f}\|_{L^\infty(B_{1.9})} + \| \bar f \|_{C_{y'}^\beta(B_{1.9})}  \le C r^\beta ([g_1]_{C^\beta_{x'}} + [g_2]_{C^\beta_{x'}}) {+ Cr^2|g_1(x_0')-g_2(x_0')|}.    
			\end{equation}
		{Reversing the scaling and the change of variables gives}
		$$
		\begin{aligned}
			&\|\hat u\|_{L^\infty(\Omega_{x_0,r})} +r\|\nabla\hat u\|_{L^\infty(\Omega_{x_0,r})} + r^{1+\beta} [\nabla \hat u]_{C^\beta(\Omega_{x_0,r})} \\
			&\le  C \left( \dashnorm{\hat u }{L^2(\Omega_{x_0,2r})} + r^{\beta}([g_1]_{C^\beta_{x'}} + [g_2]_{C^\beta_{x'}})+ r^2|g_1(x_0')-g_2(x_0')| \right).
		\end{aligned}
		$$
		From \eqref{est:ab}, we can compute that 
		$$\|\nabla(Ax_n + B)\|_{L^\infty({\Omega_{x_0,r}})} \le C |g_1(x_0')|,$$ and 
        $$\dashnorm{Ax_n + B}{L^2(\Omega_{x_0,2r})} \le C r^2|g_1(x_0')|.$$
        Finally, 
        $$
        |g_i(x_0') - g_i(0')| \le r^{\beta}[g_i]_{C^\beta_{x'}}, \quad i = 1,2.
        $$
        Hence \eqref{gradient-1/2a} follows.
        
Next, we show \eqref{gradient-1/2}. By the triangle inequality,
            \begin{equation}\label{est:2012}
				\begin{aligned}
					&\|\bar{v}\|_{W^{1,\infty}({B_{1.1}})} 
					+ [\nabla \bar{v}]_{C^\beta_{{y'}}(B_{1.1})}
					+ [\nabla \bar{v}]_{C^\beta({\cQ_0})}
					\le C(\|\bar{v} \|_{L^2(B_{1.9})} + \| \bar f \|_{C_{y'}^\beta(B_{1.9})})\\
					&\le C (\|\bar{v} - (\bar{v})_{B_{1.1}}\|_{L^2(B_{1.9})} + \| \bar f \|_{C_{y'}^\beta(B_{1.9})} + |(\bar{v})_{B_{1.1}}|).
				\end{aligned}
			\end{equation}
			We can also estimate
			\begin{equation}\label{est:mean-oscillation}
				\begin{aligned}
					&\| \bar{v} - (\bar{v})_{B_{1.1}}\|_{L^2(B_{1.9})}\\ & = \left( \int_{B_{1.9}} \big| \fint_{B_{1.1}} \bar{v}(y) - \bar{v}(x) \, dx  \big|^2 \, dy \right)^{1/2}\\
					&\le C \left( \int_{B_{1.9}}  \fint_{B_{1.9}} \big|{\bar{v}(y)-\bar{v}(x)}\big|^2 \,dx\,dy \right)^{1/2}\\
					&\le C \left( \int_{B_{1.9}} \big|\bar{v}(y) - (\bar{v})_{B_{1.9}} \big|^2  \, dy \right)^{1/2}\\
					&= C \| \bar{v} - (\bar{v})_{B_{1.9}}\|_{L^2(B_{1.9})}.
				\end{aligned}
			\end{equation}
	On the other hand, since $\bar c_L$ is negative and Lipschitz, we can divide the equation \eqref{eqn:extend} by $\bar c_L$ and obtain
			\begin{equation}\label{eqn:extend_divide}
				\partial_{y_i}\Big( \frac{\bar a^{ij}}{\bar{c}_L}\partial_{y_j}\bar v - \frac{\bar f^i}{\bar c_L}\Big) + \Big( \frac{\bar b^i}{\bar c_L} + \frac{\bar a^{ji} \partial_{y_j} \bar c_L}{\bar c_L^2} \Big) \partial_{y_i}\bar v  +  \bar v = \frac{\bar f^i \partial_{y_i} \bar c_L}{\bar {c}_L^2} {-} \frac{\bar{c}_R}{\bar{c}_L} \bar{v}  \quad \mbox{in}~ B_{1.9}.
			\end{equation}
			Now integrating the equation \eqref{eqn:extend_divide} in $B_{1.1}$ gives
			\begin{align*}
				\int_{B_{1.1}} \bar{v} =& - \int_{\partial B_{1.1}} \Big( \frac{\bar a^{ij}}{\bar{c}_L}\partial_{y_j}\bar v - \frac{\bar f^i}{\bar c_L}\Big) \nu^i - \int_{B_{1.1}} \Big( \frac{\bar b^i}{\bar{c}_L} + \frac{\bar a^{ji} \partial_{y_j} \bar c_L}{\bar c_L^2} \Big) \partial_{y_i}\bar v\\
				&+ \int_{B_{1.1}} \frac{\bar f^i \partial_{y_i} \bar c_L}{\bar c_L^2} {-} \int_{B_{1.1}} \frac{\bar{c}_R}{\bar{c}_L} \bar{v} .
			\end{align*}
			By \eqref{est:coefficient} and interpolation, for $\delta > 0$ small, we have
			\begin{equation}\label{est:interpolate}
				\begin{aligned}
					|(\bar{v})_{B_{1.1}}| \le &C \gamma \Big( \|\nabla \bar{v}\|_{L^\infty(B_{1.1})} + \|\bar f\|_{L^\infty(B_{1.1})} + r \| \bar v \|_{L^2(B_{1.1})} \Big)\\
					\le &C \gamma \Big( \|\nabla_{y'} \bar{v}\|_{L^\infty(B_{1.1})}
						+\|\partial_{y_n} \bar{v}\|_{L^\infty(B_{1.1})}
						+ \|\bar f\|_{L^\infty(B_{1.1})} + r \| \bar v \|_{L^2(B_{1.1})} \Big)\\
					\le &\delta [\nabla\bar{v}]_{C^\beta_{{y'}}(B_{1.1})} + C_{\delta,\gamma} \| \bar{v} - (\bar{v})_{B_{1.9}}\|_{L^2(B_{1.9})} 
					{+C\gamma\|\partial_{y_n} \bar{v}\|_{L^\infty(B_{1.1})}}\\
					&+ C \gamma \|\bar f\|_{L^\infty(B_{1.1})}
					+ C \gamma r_0 \| \bar{v} - (\bar{v})_{B_{1.1}}\|_{L^2(B_{1.1})} + C\gamma r_0|(\bar{v})_{B_{1.1}}|.
				\end{aligned}
			\end{equation}
			Now we choose $r_0$ satisfying $C\gamma r_0 < 1$ so that the last term can be absorbed by the left-hand side. 
			Therefore, combining \eqref{est:f_bar}, \eqref{est:mean-oscillation}, \eqref{est:interpolate}, {and choosing a sufficiently small $\delta$,} the estimate \eqref{est:2012} implies}
	\begin{equation}
	                           \label{eq1.15}
        \begin{aligned}
			&\|\bar{v}\|_{W^{1,\infty}(B_{1.1})} + [\nabla \bar{v}]_{C^\beta_{y'}(B_{1.1})} 
			+ [\nabla \bar{v}]_{C^\beta({\cQ_0})}
			\\
			&\le  C_\gamma \big( \| \bar{v} - (\bar{v})_{B_{1.9}}\|_{L^2(B_{1.9})} + r^{\beta}([g_1]_{C^\beta_{x'}} + [g_2]_{C^\beta_{x'}}) \\
			&\qquad + r^2|g_1(x_0')-g_2(x_0')| {+\|\partial_{y_n} \bar{v}\|_{L^\infty(B_{1.1})}}\big).    
		\end{aligned}
    \end{equation} 
        Here we bound the first term on the right-hand of \eqref{eq1.15} by
\begin{align*}
&\dashnorm{\tilde{v} - (\tilde{v})_{Q_{1.9r,r^2}}}{L^2(Q_{1.9r,r^2})}\\
&\le C\dashnorm{\tilde{u} - (\tilde{u})_{Q_{1.9r,r^2}}}{L^2(Q_{1.9r,r^2})}
+ C|(\tilde{u})_{Q_{0.9r,r^2}}|
\osc_{Q_{1.9r,r^2}}\psi.
\end{align*}
Since $\osc_{Q_{1.9r,r^2}}\psi\le Cr^2$, for sufficiently small $r_0$, the last term on the right-hand side above can be absorbed to the left-hand side of \eqref{eq1.15}. 
        Reversing the scaling and the change of variables gives
		\begin{align*}
			&\|\hat u\|_{L^{\infty}(\Omega_{x_0,r})} +r\|\nabla\hat u\|_{L^{\infty}(\Omega_{x_0,r})} + r^{1+\beta} [\nabla \hat u]_{C^\beta(\Omega_{x_0,r})} \\
			&\le   C_\gamma \left( \dashnorm{\hat u - (\hat{u})_{\Omega_{x_0,2r}}}{L^2(\Omega_{x_0,2r})} + r^{\beta}([g_1]_{C^\beta_{x'}} + [g_2]_{C^\beta_{x'}}) \right.\\
			&\qquad\left. + r^2|g_1(x_0')-g_2(x_0')|{+r\|\partial_{x_n} \hat{u}\|_{L^\infty(\Omega_{x_0,2r})}} \right)
		\end{align*}
for $r_0$ sufficiently small. 
Arguing as before, we obtain \eqref{gradient-1/2}.
	\end{proof}
	
	\section{A dimensional reduction argument}\label{sec4}
	
	Note that the weak formulation of \eqref{main_problem_narrow} is given by 
	$$
	\int_{\Omega_{R_0}} \nabla u \cdot \nabla \varphi \, dx + \frac{1}{\gamma} \int_{\Gamma_+} (u-g_1) \varphi \, d\sigma+ \frac{1}{\gamma} \int_{\Gamma_-} (u-g_2) \varphi \, d\sigma = 0
	$$
	for all smooth $\varphi$ that {vanish} on $\{|x'| = R_0\}$.
	First, we apply the same flattening map and change of variables as in \cite{DongYangZhu26}. Let \begin{equation}\label{x_to_y}
		\left\{
		\begin{aligned}
			y' &= x' ,\\
			y_n &= 2 \varepsilon \left( \frac{x_n - h_2(x') + \varepsilon/2}{\varepsilon + h_1(x') - h_2(x')} - \frac{1}{2} \right),
		\end{aligned}\right.
		\quad \forall (x',x_n) \in \Omega_{R_0},
	\end{equation}
	and let $v(y) = u(x)$. Then one can deduce as in \cite{DongYangZhu26}*{Section 5}  that $v$ satisfies 
	\begin{equation}\label{equation_v}
		\left\{
		\begin{aligned}
			-\partial_i(a^{ij}(y) \partial_j v(y)) &=0 \quad \mbox{in } Q_{R_0, \varepsilon},\\
			\frac{2\varepsilon}{\gamma} \sqrt{1+ |\nabla_{y'}h_1|^2} v(y) + a^{nj}(y) \partial_j v(y) &= \frac{2\varepsilon}{\gamma} \sqrt{1+ |\nabla_{y'}h_1|^2} g_1(y') \quad \mbox{on } \{y_n = \varepsilon\},\\
			\frac{2\varepsilon}{\gamma} \sqrt{1+ |\nabla_{y'}h_2|^2} v(y) - a^{nj}(y) \partial_j v(y) &=\frac{2\varepsilon}{\gamma} \sqrt{1+ |\nabla_{y'}{h_2}|^2} g_2(y')  \quad \mbox{on } \{y_n = -\varepsilon\},
		\end{aligned}
		\right.
	\end{equation}
	where 
	\begin{equation}\label{a_ij_formula}
		\begin{aligned}
			(a^{ij}(y)) =& \frac{ 2 \varepsilon(\partial_x y)(\partial_x y)^t}{\det (\partial_x y)} \\
			=& \begin{pmatrix}
				\varepsilon + \mu|y'|^2 &0 &\cdots &0 &a^{1n}\\
				0 &\varepsilon + \mu|y'|^2 &\cdots &0 &a^{2n}\\
				\vdots &\vdots &\ddots &\vdots &\vdots\\
				0 &0 &\cdots &\varepsilon + \mu|y'|^2 &a^{n-1,n}\\
				a^{n1} &a^{n2} &\cdots &a^{n,n-1} &  \frac{4\varepsilon^2 + \sum_{i=1}^{n-1}|a^{in}|^2}{\varepsilon + h_1(y') - h_2(y')}
			\end{pmatrix}\\
			&+
			\begin{pmatrix}
				e^1 &0 &\cdots &0 &0\\
				0 &e^2 &\cdots &0 &0\\
				\vdots &\vdots &\ddots &\vdots &\vdots\\
				0 &0 &\cdots & e^{n-1} &0\\
				0 &0 &\cdots & 0 &0
			\end{pmatrix},
		\end{aligned}
	\end{equation}
	
	\begin{align*}
		a^{ni} = a^{in} &= -2\varepsilon \partial_i h_2(y') - (y_n + \varepsilon)\partial_i(h_1(y') - h_2(y'));\\
		e^{i} &= h_1(y')-h_2(y')-\mu|y'|^2 \quad \mbox{for}~i= 1,2,\ldots, n -1,
	\end{align*}
	and
	\begin{equation}\label{Q_s_t}
		Q_{s,t}:= \{ y = (y',y_n) \in \bR^n ~\big|~  |y'| < s,  |y_n| < t\}.
	\end{equation}
	Then we define the vertical average
	\begin{equation}
		\label{v_bar_def}
		\bar{v}(y') := \fint_{-\varepsilon}^\varepsilon v(y',y_n)\, dy_n.
	\end{equation}
	By \eqref{equation_v}, $\bar{v}$ satisfies
	\begin{equation*}
		\begin{aligned}
			&\dv((\varepsilon+\mu|y'|^2)\nabla \bar v)\\
			=& - \frac{1}{2\varepsilon}\sum_{j=1}^{n} (a^{nj} \partial_j v)(y',\varepsilon) + \frac{1}{2\varepsilon}\sum_{j=1}^{n} (a^{nj} \partial_j v)(y',-\varepsilon) \\
			&-\sum_{i=1}^{n-1}\partial_i(\overline{a^{in}\partial_n v}) - \sum_{i=1}^{n-1} \partial_i(e^i \partial_i \bar v)\\
			=& \gamma^{-1}\sqrt{1+ |\nabla_{y'}h_1|^2} (v(y', \varepsilon)-g_1(y'))+ \gamma^{-1}\sqrt{1+ |\nabla_{y'}h_2|^2} (v(y', -\varepsilon)-g_2(y'))\\
			&-\sum_{i=1}^{n-1}\partial_i(\overline{a^{in}\partial_n v}+e^i \partial_i \bar{v}) \\
			=& 2\gamma^{-1} \bar{v}(y')    \underbrace{-\sum_{i=1}^{n-1}\partial_i(\overline{a^{in}\partial_n v}+e^i \partial_i \bar{v})}_{:= \dv F}\\
			&+\underbrace{\gamma^{-1}\sqrt{1+ |\nabla_{y'}h_1|^2} v(y', \varepsilon) + \gamma^{-1}\sqrt{1+ |\nabla_{y'}h_2|^2} v(y', -\varepsilon) - 2\gamma^{-1} \bar{v}(y')}_{:= G_1}\\&+
			\underbrace{-\gamma^{-1}\sqrt{1+ |\nabla_{y'}h_1|^2} g_1(y') - \gamma^{-1}\sqrt{1+ |\nabla_{y'}h_2|^2} g_2(y')}_{:= G_2},
		\end{aligned}
	\end{equation*}
	where $\overline{a^{in}\partial_n v}$ denotes the vertical average of $a^{in}\partial_n v$ with respect to $y_n \in (-\varepsilon, \varepsilon)$. 
	{We use this notation} throughout this article. 
	
	Summarizing this dimension reduction argument, we have the following lemma.
	
	\begin{lemma}\label{lemma_reduction}
		Let $h_1,h_2{\in C^{2,\sigma}}$ satisfy \eqref{fg_0} -- \eqref{def:c_2}, $u \in H^1(\Omega_{R_0})$ be a solution of \eqref{main_problem_narrow}. Perform  the change of variables \eqref{x_to_y}, and let $v(y) = u(x)$, $\bar{v}$ be defined as in \eqref{v_bar_def}. Then $\bar{v}$ satisfies the equation
		\begin{equation}\label{equation_v_bar}
			\dv((\varepsilon+\mu|y'|^2)\nabla \bar v) - 2\gamma^{-1} \bar{v} = \dv F + G_1+G_2 \quad \mbox{in}~B_{R_0} \subset \bR^{n-1},
		\end{equation}
		where
		\begin{equation}\label{FG_bound}
			\begin{aligned}
				|F(y')| &\le C\Big(\varepsilon|y'| \overline{|\partial_n v|} + |e^1| |\nabla_{y'}\overline{v}|\Big), \\
				|G_1(y')| &\le C\Big(\varepsilon \max_{y_n \in (-\varepsilon, \varepsilon)} |\partial_n v| + |y'|^2 |\overline{v}| \Big),\\
				G_2(y')&=-\gamma^{-1}\sqrt{1+ |\nabla_{y'}h_1|^2} g_1(y') - \gamma^{-1}\sqrt{1+ |\nabla_{y'}h_2|^2} g_2(y'),
			\end{aligned}
		\end{equation}
		and $C$ is some positive constant depending only on $n$, $\gamma$, $\|h_1\|_{C^{2,\sigma}}$, and $\|h_2\|_{C^{2,\sigma}}$.
	\end{lemma}
	Note that the estimates for $F$ and $G_1$ follow from \cite{DongYangZhu26}*{Lemma 5.1}, and they will be higher order terms compared to the homogeneous part. The term $G_2$ plays an important role in the critical case $\gamma = 1/\mu$ in Theorems \ref{thm_upperbound} and \ref{thm_upperbound_3}.

	\section{Proofs of Theorems \ref{thm_upperbound} and \ref{thm_upperbound_2}}\label{sec:upper-bounds}
	
	In this section, we prove Theorems \ref{thm_upperbound} and \ref{thm_upperbound_2}. Without loss of generality, we may assume $\mu = 1$. Namely, we consider
	\begin{equation*}
		h_1(x')- h_2(x') =|x'|^2 + O(|x'|^{2+\sigma}) \quad\mbox{for}~~|x'|<R_0,
	\end{equation*}
	and 
	$$\nabla h_1(x')-\nabla h_2(x')=2x'+O(|x'|^{1+\sigma}) \quad\mbox{for}~~|x'|<R_0.$$
	In this case,
	\begin{equation}
		\label{alpha_new}
		\alpha = \frac {-(n-1)+\sqrt{(n-1)^2+4(n-2 +  2/\gamma)}}{2},
	\end{equation}
	\begin{equation}\label{alpha_new_zero}
		\alpha_0 = \frac{-(n-1) + \sqrt{(n-1)^2 + 8/\gamma}}{2}.
	\end{equation}
	For general $\mu > 0$, we only need to replace $\varepsilon$ and $\gamma$ in the proof by $\varepsilon/\mu$ and $\mu\gamma$ respectively, in view of \eqref{equation_v_bar}. {The strategy of the proof is based on the proof of \cite{DongYangZhu26}*{Theorem 1.4}. However, the extra nonhomogeneous term $G_2$ in the equation \eqref{equation_v_bar} introduces an additional level of difficulty and a new blow-up phenomenon in the critical case $\gamma=1/\mu$. See \eqref{grad_u_upperbound3}.}
	
	For $\varepsilon > 0$, $\sigma, s \in \bR$, an open set $\cD \subset \bR^{n-1}$, we introduce the following norm
	\begin{equation}\label{def_norm}
		\| F \|_{\varepsilon, \sigma,s,\cD}: = \sup_{y' \in \cD} \frac{|F(y')|}{|y'|^{\sigma} (\varepsilon + |y'|^2)^{1-s}}.
	\end{equation}
	Throughout this section, $B_R$ always lies in $\bR^{n-1}$. For any $0 < R < R_0$, we denote
	$$
	(u)_{\partial B_R}:= \fint_{\partial B_R} u \, dS
	$$
	to be the average of $u$ over $\partial B_R$. 
	
	First, we establish estimates for $\bar{v}$.
	\begin{proposition}
		\label{prop_grad_v_bar_control}
		For $n \ge 2$, $0 \le s \le 1/2$, $\gamma > 0$, $\sigma> 0$, and $\varepsilon > 0$,
		let $\bar{v} \in H^1({B_{R_0}})$ be a solution of
		\begin{equation*}
			\dv((\varepsilon+|y'|^2)\nabla \bar v) -2\gamma^{-1}\bar{v} = \dv F + G \quad\text{in}\,\,B_{R_0}\subset \bR^{n-1},
		\end{equation*}
		where $F,G$ satisfy
		\begin{equation}
			\label{FG_bound3}
			\|F\|_{\varepsilon, \sigma,s,B_{R_0}}< \infty , \quad \|G\|_{\varepsilon, \sigma-1,s,B_{R_0}} < \infty.
		\end{equation}
		Let $\alpha$, $\alpha_0$ be the constants given in \eqref{alpha_new}, \eqref{alpha_new_zero}, and $\tilde\alpha := \min \{ \alpha, 1, 1 + \sigma - 2s \}$, $\tilde\alpha_0 := \min \{ \alpha_0, 1, 1 + \sigma - 2s \}$ . There exists a positive constant $C$ depending only on $n$, $\sigma$, $s$, but independent of $\varepsilon$, such that for any $R \in (0, R_0/2)$, we have:
		\begin{itemize}
			\item[(i)] If $1+\sigma - 2s \neq \min\{\alpha,1\}$, then
			\begin{equation}
				\label{v_bar_L2_sphere_old}
				\begin{aligned}
					&\dashnorm{\bar{v} - (\bar{v})_{\partial B_R}}{L^2(\partial B_R)}\\
					&\le C(\|F\|_{\varepsilon, \sigma,s,B_{R_0}} + \|G\|_{\varepsilon, \sigma-1,s,B_{R_0}} +  \dashnorm{\bar v - (\bar{v})_{\partial B_{R_0}}}{L^2(\partial B_{R_0})})  R^{\tilde\alpha}.
				\end{aligned}
			\end{equation}
			\item[(ii)] If $1+\sigma - 2s = \min\{\alpha,1\}$, then
			\begin{equation}
				\label{v_bar_L2_sphere}
				\begin{aligned}
					&\dashnorm{\bar{v} - (\bar{v})_{\partial B_R}}{L^2(\partial B_R)}\\
					&\le C(\|F\|_{\varepsilon, \sigma,s,B_{R_0}} + \|G\|_{\varepsilon, \sigma-1,s,B_{R_0}} +  \dashnorm{\bar v - (\bar{v})_{\partial B_{R_0}}}{L^2(\partial B_{R_0})})  R^{\tilde\alpha}\log(R_0/R).
				\end{aligned}
			\end{equation}
			
			\item[(iii)] If $1+\sigma - 2s \neq \min\{\alpha_0,1\}$, then
			\begin{equation}
				\label{v_bar_L2_sphere_new}
				\begin{aligned}
					&\dashnorm{\bar{v} - \bar{v}(0)}{L^2(\partial B_R)}\\
					&\le C(\|F\|_{\varepsilon, \sigma,s,B_{R_0}} + \|G\|_{\varepsilon, \sigma-1,s,B_{R_0}} +  \dashnorm{\bar v - \bar{v}(0)}{L^2(\partial B_{R_0})})  (\varepsilon+R^2)^{\tilde\alpha_0/2}.
				\end{aligned}
			\end{equation}
		\end{itemize}
		
	\end{proposition}
	
	For the proof, we first recall the following lemmas for the homogeneous and inhomogeneous equations.

	\begin{lemma}[\cite{DongYangZhu26}*{Lemma 6.2}]
		\label{lemma_v1}
		For $n \ge 2$, $\gamma > 0$, and $\varepsilon > 0$, let $v_1 \in H^1(B_{R_0})$ satisfy
		\begin{equation}
			\label{equation_v1}
			\dv((\varepsilon+|y'|^2)\nabla v_1) -2\gamma^{-1}v_1=0\quad\text{in}\,\,B_{R_0} \subset \bR^{n-1}.
		\end{equation}
		Then for any $0 < \rho < R \le R_0$, we have
		$$
		\left( \fint_{\partial B_\rho} |v_1(y') - (v_1)_{\partial B_\rho}|^2 \, dS\right)^{\frac{1}{2}} \le \left(\frac{\rho}{R} \right)^{\widehat\alpha} \left( \fint_{\partial B_R} |v_1(y') - (v_1)_{\partial B_R}|^2 \, dS \right)^{\frac{1}{2}},
		$$
		where $\widehat{\alpha}:= \min\{\alpha,1 \}$, and $\alpha$ is given in \eqref{alpha_new}.
	\end{lemma}

	\begin{lemma}[\cite{DongYangZhu26}*{Lemma 6.3}]
		\label{lemma_v2}
		For $n \ge 2$, $s \le 1/2$, $\sigma > 0$, and $\varepsilon > 0$, suppose that $F,G$ satisfy \eqref{FG_bound3} {with $B_1$ in place of $B_{R_0}$}, and $v_2 \in H^1_0(B_{1})$ satisfies
		\begin{equation*}
			\dv((\varepsilon+|y'|^2)\nabla v_2) - 2\gamma^{-1} v_2=\dv F + G\quad\text{in}\,\,B_{1}\subset \bR^{n-1}.
		\end{equation*}
		Then we have
		\begin{equation*}
			\|v_2\|_{L^\infty(B_{1})}\le C(\|F\|_{\varepsilon,\sigma,s, B_{1}} + \|G\|_{\varepsilon,\sigma-1,s,B_{1}}),
		\end{equation*}
		where $C>0$ depends only on $n$, $\sigma$, and $s$, and is in particular independent of $\varepsilon$.
	\end{lemma}
	We also need the following lemma to estimate the average of $v_1$.
	\begin{lemma}
		\label{lemma_v1_zero_mode}
		For $n \ge 2$, $\gamma > 0$, and $\varepsilon > 0$, let $v_1 \in H^1(B_{R_0})$ satisfy \eqref{equation_v1}.
		Then for any $0 < \rho < R < R_0$, we have
		\begin{equation}\label{v1_control_zero}
			|(v_1)_{\partial B_\rho}-v_1(0)| \le \left(\frac{\varepsilon+\rho^2}{\varepsilon+R^2} \right)^{\alpha_0/2} |(v_1)_{\partial B_R}-v_1(0)|,
		\end{equation}
		where $\alpha_0$ is given in \eqref{alpha_new_zero}.
	\end{lemma}
	
	\begin{proof}
		Let $V_0(r):=(v_1)_{\partial B_r}$. By continuity of $v_1$, we have $V_0(0)=v_1(0)$. Without loss of generality, assume $V_0(0)\ge 0$. {Otherwise, consider $-v_1$ in place of $v_1$.}
		Using the spherical harmonic decomposition as in the proof of \cite{DongYangZhu26}*{Lemma 6.2}, one can see that $V_0$ satisfies the following equation
		\begin{equation}\label{zero-mode}
			\begin{aligned}
				L_0 V_{0}:= V_{0}''(r) +\left(\frac {n-2}r +\frac{2r}{\varepsilon+r^2} \right)V_{0}'(r)- \frac{2}{\gamma(\varepsilon + r^2)}  V_{0}(r)=0 \quad \mbox{in}~~(0,R_0).
			\end{aligned}
		\end{equation} 
		Then it follows that $v_0(y):=V_0(|y|)$ is a radial solution of \eqref{equation_v1}. 
		Let $\rho\in (0,R_0)$.
		By the weak maximum principle for $v_0$ in $B_\rho$, we know that $$0\le v_0(0)\le \underset{\partial{B_\rho}}{\max} \,v_0.$$
		This directly implies that
		\begin{equation}\label{V0-lower}
			V_0(\rho) - V_0(0)\ge0, \quad \rho\in(0,R_0).
		\end{equation}
		Next, we prove the upper bound for $V_0(\rho)-V_0(0)$. First, we show that $w(r):=(\varepsilon+r^2)^{\alpha_0/2}$ is a supersolution. By direct computations,
		\begin{equation*}
			w'(r) =  \alpha_0 r (\varepsilon + r^2)^{\alpha_0/2-1},
		\end{equation*}
		\begin{equation*}
			w''(r) = \alpha_0 (\varepsilon + r^2)^{\alpha_0/2-1} +  \alpha_0 (\alpha_0 - 2) r^2 (\varepsilon + r^2)^{\alpha_0/2-2}.
		\end{equation*}
		Thus
		\begin{align*}
			L_0 w&=\alpha_0^2r^2(\varepsilon + r^2)^{\alpha_0/2-2}+\big((n-1)\alpha_0-2\gamma^{-1}\big)(\varepsilon + r^2)^{\alpha_0/2-1}
			\\&=-\alpha_0^2\varepsilon(\varepsilon + r^2)^{\alpha_0/2-2}\le 0,
		\end{align*}
		where in the last line, we used the fact that 
		$$\alpha_0^2+(n-1)\alpha_0-2\gamma^{-1}=0$$
		in view of the definition of $\alpha_0$ in \eqref{alpha_new_zero}.
		
		Next we construct a barrier
		\begin{equation*}
			W(r) := V_0(0) + B\,w(r)=V_0(0) + B(\varepsilon + r^2)^{\alpha_0/2},
		\end{equation*}
		where the constant $B$ is defined as
		$$B:=\frac{V_0(R) - V_0(0)}{(\varepsilon + R^2)^{\alpha_0/2}}\ge 0.$$
		Then
		\begin{equation*}
			{L_0}W = - \frac{2}{\gamma(\varepsilon + r^2)}  V_{0}(0)+B\,L_0w\le 0=L_0 V_0 \quad \mbox{in}~~(0,R_0) .  
		\end{equation*}
		Note that $W(R)=V_0(R)$, { and $W(\delta)\ge V_0(\delta)$ for small $\delta>0$ by the continuity of $V_0$.} Then by the comparison principle, $V_0\le W$ in $(0,R)$. This directly implies
		\begin{equation}\label{V0-upper}
			V_0(\rho) - V_0(0) \leq B(\varepsilon + \rho^2)^{\alpha_0/2}=\Big(\frac{\varepsilon + \rho^2}{\varepsilon + R^2}\Big)^{\alpha_0/2} (V_0(R) - V_0(0)) \quad \forall \rho \in (0, R).
		\end{equation}
		The lemma is proved by combining \eqref{V0-lower} and \eqref{V0-upper}.
	\end{proof}

	\begin{proof}[Proof of Proposition \ref{prop_grad_v_bar_control}] The first case where $1+\sigma - 2s \neq \min\{\alpha,1\}$ has been proved in \cite{DongYangZhu26}. Therefore, we only give the proof for the cases (ii) and (iii).

		We denote
		$$
		{M:}=\|F\|_{\varepsilon, \sigma,s,B_{R_0}} + \|G\|_{\varepsilon, \sigma-1,s,B_{R_0}}.
		$$
		For $0 < \rho \le R/2 \le R_0/2$, we write $\bar v=v_1+v_2$ in $B_R$, where $v_2$ satisfies
		$$
		\dv((\varepsilon+|y'|^2)\nabla v_2) - 2\gamma^{-1}v_2=\dv F + G\quad\text{in}\,\,B_R
		$$
		and $v_2=0$ on $\partial B_R$. Thus $v_1$ satisfies
		$$
		\dv((\varepsilon+|y'|^2)\nabla v_1)- 2\gamma^{-1}v_1=0\quad\text{in}\,\,B_R,
		$$
		and $v_1=\bar v$ on $\partial B_R$.
		Since $\tilde v_2(y'):=v_2(Ry')$ satisfies
		$$
		\dv((R^{-2}\varepsilon+|y'|^2)\nabla \tilde v_2) - 2\gamma^{-1}\tilde v_2
		=\dv \tilde F + \tilde G\quad\text{in}\,\,B_1,
		$$
		where $\tilde F(y'):=R^{-1}F(Ry')$ and $\tilde{G}(y') := G(Ry')$ satisfy
		\begin{align*}
			\|\tilde{F}\|_{R^{-2}\varepsilon,\sigma,s,B_1} &= R^{1+ \sigma - 2s} \|F\|_{\varepsilon, \sigma,s,B_{R}}, \\
			\|\tilde G\|_{R^{-2}\varepsilon, \sigma-1,s,B_{1}} &= R^{1+ \sigma - 2s} \|G\|_{\varepsilon, \sigma-1,s,B_{R}},
		\end{align*}
		we can apply Lemma \ref{lemma_v2} to $\tilde{v}_2$ with $R^{-2}\varepsilon$ in place of $\varepsilon$ to obtain
		\begin{equation}
			\label{v2_control}
			\|v_2\|_{L^\infty(B_R)}\le C{M}R^{1+\sigma-2s}.
		\end{equation}
		By Lemma \ref{lemma_v1},
		\begin{equation}
			\label{v1_control}
			\left( \fint_{\partial B_\rho} |v_1(y') - (v_1)_{\partial B_\rho}|^2 \, dS \right)^{\frac{1}{2}} \le \left(\frac{\rho}{R} \right)^{\widehat\alpha} \left( \fint_{\partial B_R} |v_1(y') - (v_1)_{\partial B_R}|^2 \, dS \right)^{\frac{1}{2}},
		\end{equation}
		where $\widehat\alpha=\min\{\alpha,1\}$.
		By \eqref{v1_control} and \eqref{v1_control_zero}, using orthogonality and the fact that $\alpha_0\le \alpha$, we have
		\begin{equation}
			\label{v1_control_new}
			\left( \fint_{\partial B_\rho} |v_1(y') - v_1(0)|^2 \, dS \right)^{\frac{1}{2}} \le \left(\frac{\varepsilon+\rho^2}{\varepsilon+R^2} \right)^{\widehat\alpha_0/2} \left( \fint_{\partial B_R} |v_1(y') - v_1(0)|^2 \, dS \right)^{\frac{1}{2}},
		\end{equation}
		where $\widehat\alpha_0=\min\{\alpha_0,1\}$.
		
		{ We denote
			$$
			\omega(\rho):= \dashnorm{\bar{v} - (\bar{v})_{\partial B_{\rho}}}{L^2(\partial B_\rho)}
			$$
			and 
			$$
			\omega_0(\rho):= \dashnorm{\bar{v} - \bar{v}(0)}{L^2(\partial B_\rho)}.
			$$}
		Combining \eqref{v1_control} and \eqref{v2_control} yields, using $\bar{v} = v_1$ on $\partial B_R$ and $\bar{v} = v_1 + v_2$,
		\begin{equation}
			\begin{aligned}
				\label{omega_iteration}
				\omega(\rho)&\le \left( \fint_{\partial B_\rho} |v_1(y') - (v_1)_{\partial B_\rho}|^2 \, dS \right)^{\frac{1}{2}} + \left( \fint_{\partial B_\rho} |v_2(y') - (v_2)_{\partial B_\rho}|^2 \, dS \right)^{\frac{1}{2}}\\
				&\le \left(\frac{\rho}{R} \right)^{\widehat\alpha} \left( \fint_{\partial B_R} |v_1(y') -  (v_1)_{\partial B_R}|^2 \, dS \right)^{\frac{1}{2}} + 2\|v_2\|_{L^\infty(B_R)}\\
				&\le \left(\frac{\rho}{R} \right)^{\widehat\alpha} \omega(R) + C{M}R^{1+ \sigma - 2s}.
			\end{aligned}
		\end{equation}
		{Similarly, combining} \eqref{v1_control_new} and \eqref{v2_control} yields
		\begin{equation}
			\begin{aligned}
				\label{omega_iteration_zero}
				\omega_0(\rho)&\le \left( \fint_{\partial B_\rho} |v_1(y') - v_1(0)|^2 \, dS \right)^{\frac{1}{2}} + \left( \fint_{\partial B_\rho} |v_2(y') - v_2(0)|^2 \, dS \right)^{\frac{1}{2}}\\
				&\le \left(\frac{\varepsilon+\rho^2}{\varepsilon+R^2} \right)^{\widehat\alpha_0/2} \left( \fint_{\partial B_R} |v_1(y') -  v_1(0)|^2 \, dS \right)^{\frac{1}{2}} + 2\|v_2\|_{L^\infty(B_R)}\\
				&\le \left(\frac{\varepsilon+\rho^2}{\varepsilon+R^2} \right)^{\widehat\alpha_0/2} \omega_0(R) + C{M}R^{1+ \sigma - 2s}.
			\end{aligned}
		\end{equation}
		
		We first treat case (ii) where $1 + \sigma - 2s = \widehat\alpha$.
		For a positive integer $k$, we take $\rho = 2^{-i-1}R_0$ and $R = 2^{-i}R_0$ in \eqref{omega_iteration}, and iterate from $i = 0$ to $k-1$. 
		\begin{align*}
			\omega(2^{-k}R_0) &\le 2^{-k\widehat\alpha} \omega(R_0) + C{M}\sum_{i=1}^k 2^{-(k-i)\widehat\alpha} (2^{1-i})^{1+\sigma - 2s}\\
			&= 2^{-k\widehat\alpha} \omega(R_0) + C{M}\sum_{i=1}^k 2^{-(k-i)\widehat\alpha} (2^{1-i})^{\widehat\alpha}\\
			&\le 2^{-k\widehat\alpha} \omega(R_0) + C{M}k2^{-k\widehat\alpha}.
		\end{align*}
		For any $\rho \in (0, R_0/2)$, let $k$ be the integer such that $2^{-k-1}R_0 < \rho \le 2^{-k}R_0$. Then by \eqref{omega_iteration} and the previous inequality,
		\begin{equation*}
			\begin{aligned}
\omega(\rho) &\le \omega(2^{-k}R_0)+C{M}2^{-k\widehat\alpha} \\ &\le 2^{-k\widehat\alpha} \omega(R_0) + C{M}k2^{-k\widehat\alpha}\\ &\le C{(\omega(R_0)+M)}\rho^{\widehat\alpha}\log(R_0/\rho).
\end{aligned}
		\end{equation*}
		Therefore, \eqref{v_bar_L2_sphere} is proved.

	Now we consider case (iii), where $1+\sigma - 2s \neq \widehat\alpha_0$.
		
		For a positive integer $k$ such that $2^{-k}R_0\ge \sqrt{\varepsilon}$, we take $\rho = 2^{-i-1}R_0$ and $R = 2^{-i}R_0$ in \eqref{omega_iteration_zero}, and iterate from $i = 0$ to $k-1$ to get
		\begin{align*}
			\omega_0(2^{-k}R_0) &\le  \left(\frac{\varepsilon+2^{-2k}R_0^2}{\varepsilon+R_0^2} \right)^{\widehat\alpha_0/2}\omega_0(R_0)\\ &\quad + C{M}\sum_{i=1}^k 2^{-(i-1)(1+\sigma-2s)} \left(\frac{\varepsilon+2^{-2k}R_0^2}{\varepsilon+2^{-2i}R_0^2} \right)^{\widehat\alpha_0/2}\\
			&\le C2^{-k\widehat\alpha_0} \omega_0(R_0) + C{M}\sum_{i=1}^k 2^{-(i-1)(1+\sigma-2s)} 2^{-(k-i)\widehat\alpha_0}\\
			&\le C2^{-k\widehat\alpha_0} \omega_0(R_0) + C{M}2^{-k\widehat\alpha_0}\frac{1-2^{k(\widehat\alpha_0-1-\sigma+2s)}}{1-2^{\widehat\alpha_0-1-\sigma+2s}}.
		\end{align*}
		It follows that
		\begin{equation}\label{w0_1}
			\omega_0(2^{-k}R_0) \le C2^{-k\tilde\alpha_0} \left( \omega_0(R_0) + M\right), \quad \mbox{where } \tilde{\alpha}_0 = \min\{\widehat{\alpha}_0, 1 + \sigma - 2s \}.
		\end{equation}
		For any $\rho \in [\sqrt{\varepsilon}, R_0/2)$, let $k$ be the integer such that $2^{-k-1}R_0 < \rho \le 2^{-k}R_0$. Then by \eqref{omega_iteration_zero} and \eqref{w0_1},
		\begin{equation}\label{w0_2}
			\omega_0(\rho) \le \omega_0(2^{-k}R_0)+C{M}2^{-k(1+\sigma-2s)}\le C \left( \omega_0(R_0) + M\right)\rho^{\tilde\alpha_0}.
		\end{equation}
		For any $\rho \in (0,\sqrt{\varepsilon})$, by \eqref{omega_iteration_zero} and \eqref{w0_2}, we have 
		\begin{equation}\label{w0_3}
			\omega_0(\rho) \le \omega_0(\sqrt{\varepsilon})+C{M}
			\varepsilon^{(1+\sigma-2s)/2}
			\le C \left( \omega_0(R_0) + M\right)\varepsilon^{\tilde\alpha_0/2}.
		\end{equation}
		Therefore, \eqref{v_bar_L2_sphere_new} is proved by combining \eqref{w0_2} and \eqref{w0_3}.
	\end{proof}
	In the following lemma, we first show that $u_+$ has a uniform $C^{\beta}$ bound for some $\beta\in(0,1)$ independent of $\varepsilon$. As a consequence of the Schauder estimates for Neumann boundary problems, $\nabla u_-$ also has a uniform $C^{\beta}$ bound.
	\begin{lemma}\label{lem:u_holder}
		Let $u$ be the solution of \eqref{robin}. Then there exist some $\beta\in(0,\sigma]\subset (0,1)$ and a universal positive constant $C$ depending also on $\gamma$, such that 
		\begin{equation}\label{u_holder}
			\|u\|_{C^\beta(\Omega_{R_0/4})}\le C \,\|\varphi\|_{L^\infty(\partial\Omega)},
		\end{equation}
		and 
		\begin{equation}\label{u-_1+holder}
			\|\nabla u\|_{C^{\beta}(\mathcal{D}_j)}\le \frac{C}{k_j} \,\|\varphi\|_{L^\infty(\partial\Omega)},\quad \mbox{for } j=1,2.
		\end{equation}
        Moreover, if the domain $\widetilde\Omega$ is symmetric in $x_j$ and the solution $u$ is odd in $x_j$ for some $1 \le j \le n-1$, then the constant $C$ depends on $\gamma$ only through $\gamma_0$.
	\end{lemma}
	\begin{proof}
		Without loss of generality, we assume that $\|\varphi\|_{L^\infty(\partial\Omega)} = 1$. Recall that $g_i=u_-|_{\partial \cD_i}$. Similar to the proof of Theorem \ref{thm-1/2}, we first subtract a linear function of $x_n$ such that $g_1(0')=g_2(0')=0$. We define \begin{equation}\label{subtr_0}
			\hat u:=u-(a+b_n x_n),
		\end{equation}
		where
		$$a=\frac{g_1(0')+g_2(0')}{2}, \quad b_n=\frac{g_1(0')-g_2(0')}{2\gamma+\varepsilon}.$$
		Note that $|\nabla (a+b_nx_n)|= |b_n|\le C \|\varphi\|_{L^{\infty}(\partial\Omega)}=C$.
		Then $\hat u$ satisfies the equation 
		\begin{equation}\label{main_problem_narrow_4}
			\begin{cases}
				\Delta \hat u_+ =0&\mbox{in}~\Omega_{R_0},\\
				\hat u_+ + \gamma \partial_\nu \hat u_+ = \hat g_1&\mbox{on}~\Gamma_+,\\
				\hat u_+ + \gamma \partial_\nu \hat u_+ = \hat g_2&\mbox{on}~\Gamma_-,
			\end{cases}
		\end{equation}
		{ where
			$$\hat{g}_1(x')=g_1(x')-\big(a+b_n(\frac{\varepsilon}{2}+h_1(x'))+\gamma\frac{b_n}{\sqrt{1+|\nabla h_1(x')|^2}}\big),$$
			$$\hat{g}_2(x')=g_2(x')-\big(a+b_n(-\frac{\varepsilon}{2}+h_2(x'))-\gamma\frac{b_n}{\sqrt{1+|\nabla {h_2}(x')|^2}}\big).$$
			Then we have $\hat{g}_1(0')=\hat{g}_2(0')=0$, and 
			$$\hat{g}_1(x')=\hat{g}_1(x')-\hat{g}_1(0')=g_1(x')-g_1(0')-b_n h_1(x')-\gamma b_n\big(\frac{1}{\sqrt{1+|\nabla h_1(x')|^2}}-1\big),$$
			$$\hat{g}_2(x')=\hat{g}_2(x')-\hat{g}_2(0')=g_2(x')-g_2(0')-b_n h_2(x')+\gamma b_n\big(\frac{1}{\sqrt{1+|\nabla h_2(x')|^2}}-1\big).$$
			Therefore, by \eqref{u_holder_inside},  \eqref{fg_0}, \eqref{def:c_2}, and \eqref{fg_2}, for any $\theta\in(0,1)$ fixed,
			\begin{equation}\label{g12}
				|\hat g_1(x')|+|\hat g_2(x')|\le C \big( \frac{1}{k_1} + \frac{1}{k_2} \big) |x'|^{\theta}+C|x'|^2\le C|x'|^\theta.
		\end{equation}}  We then apply the dimension reduction argument to \eqref{main_problem_narrow_4} in Section \ref{sec4} with $\hat u_+$ in place of $u_+$ and $\hat{g}_i$ in place of $g_i$.
		Recall that $\mu = 1$.  We perform the subtraction \eqref{subtr_0}, make the change of variables \eqref{x_to_y}, let $v(y) = \hat{u}(x)$, and let $\bar{v}$ be the vertical average defined as in \eqref{v_bar_def}. Then by Lemma \ref{lemma_reduction}, $\bar{v}$ satisfies \eqref{equation_v_bar} with $\mu=1$, $F$, $G_1$, and $G_2$ satisfying \eqref{FG_bound} with $\hat{g}_i$ in place of $g_i$. Since $|\nabla(a+b_nx_n)|=|b_n|\le C$, it suffices to prove the estimate \eqref{u_holder} for $\hat{u}$. For simplicity of notation, we still denote $\hat u$ by $u$.
		By {\eqref{gradient-1/2}}, we have
	\begin{equation}\label{grad_u_start}
			|\nabla u(x)|\le C(\varepsilon+|x'|^2)^{-1/2}, \quad x\in \Omega_{R_0},
		\end{equation}
		and consequently,
		\begin{equation}\label{grad_v_bar_start}
			|\nabla \bar{v}(y')|\le C(\varepsilon+|y'|^2)^{-1/2}, \quad y\in Q_{R_0,\varepsilon}.
		\end{equation}
		On the other hand, since $\partial_\nu u =  - \gamma^{-1}u+\gamma^{-1}\hat{g}_i$ on $\Gamma_+$ and $\Gamma_-$, we have, by
		\eqref{fg_0} 
		and \eqref{grad_u_start} that
		$$|\partial_n u(x)| \le {C|x'|\,|\nabla_{x'}u|} + C|u|+C|\hat{g}_i| \le C, \quad \forall x \in \Gamma_+ \cup \Gamma_-.$$
		By the harmonicity of $\partial_n u$, the estimate \eqref{u_C1_outside}, and the maximum principle,
		\begin{equation*}
			|\partial_n u| \le C \quad \mbox{in}~~\Omega_{R_0},
		\end{equation*}
		and consequently,
		\begin{equation}
			\label{grad_n_v_bound}
			|\partial_n v| \le C(\varepsilon+|y'|^2)/\varepsilon \quad \mbox{in}~~Q_{R_0,\varepsilon},
		\end{equation}
		where $Q_{R_0,\varepsilon}$ is defined as in \eqref{Q_s_t}. 
		Thus,
		\begin{equation}
			\label{v-v_bar}
			|v(y', y_n) - \bar{v}(y')| \le 2\varepsilon \max_{y_n\in (- \varepsilon , \varepsilon)} |\partial_n v(y',y_n)| \le C (\varepsilon + |y'|^2) \quad \mbox{in}~~Q_{R_0,\varepsilon}.
		\end{equation}
		On the other hand, $e^i$ in \eqref{a_ij_formula} can be bounded by $C|y'|^{2+\sigma}$. 
		
		Therefore, by \eqref{FG_bound}, \eqref{g12},  \eqref{grad_v_bar_start}, and \eqref{grad_n_v_bound}, $F$, ${G}_1$, and $G_2$  satisfy
		$$
		\|F\|_{\varepsilon, \sigma, s_0, B_{R_0}} + \|{G}_1\|_{\varepsilon, \sigma-1, s_0, B_{R_0}}+\|G_2\|_{\varepsilon, \theta-1, s_0, B_{R_0}}
		\le C,
		$$
		where $s_0:=1/2$. Here we recall the definition of the norm in \eqref{def_norm}.
		Let $G:={G}_1+{G}_2$. By decreasing $\sigma$ if necessary, we may assume that $\sigma\le \theta$, $1+\sigma-2s_0=\sigma\ne\widehat\alpha_0$, where $\widehat{\alpha}_0 = \min\{\alpha_0, 1\}$, and therefore 
		\begin{equation}\label{F_G0}
			\|F\|_{\varepsilon, \sigma, s_0, B_{R_0}} + \|G\|_{\varepsilon, \sigma-1, s_0, B_{R_0}}\le C.
		\end{equation}
		Then by \eqref{F_G0}, the estimate \eqref{v_bar_L2_sphere_new} in Proposition \ref{prop_grad_v_bar_control} implies that 
		\begin{equation}\label{v_bar_decay_1}
			\dashnorm{\bar{v} - \bar{v}(0)}{L^2(\partial B_{2\varepsilon^{1/2}})}\le C \varepsilon^{\beta/2},
		\end{equation}
		where $\beta = \min\{\widehat\alpha_0 , 1+ \sigma - 2s_0 \}=\min\{\alpha_0, \sigma\}$.
		Then by \eqref{v_bar_decay_1}, \eqref{v-v_bar}, and the triangle inequality, we have
		\begin{equation}
			\begin{aligned}\label{v_L2_control_1}
				\dashnorm{u-(u)_{\Omega_{2\varepsilon^{1/2}}}}{L^2(\Omega_{2\varepsilon^{1/2}})} \le &\dashnorm{u-\bar{v}(0)}{L^2(\Omega_{2\varepsilon^{1/2}})}  
				\\ \le &  C \dashnorm{v-\bar v(0)}{L^2(Q_{2\varepsilon^{1/2} , \varepsilon})} \\
				\le& C( \dashnorm{v-\bar{v}}{L^2(Q_{2\varepsilon^{1/2} , \varepsilon})} + \dashnorm{\bar v-\bar{v}(0)}{L^2(Q_{2\varepsilon^{1/2} , \varepsilon})})\\
				\le& C \varepsilon^{\beta/2}.
			\end{aligned}
		\end{equation}
		By \eqref{v_L2_control_1}, \eqref{gradient-1/2} with $g_1(0') = g_2(0') = 0$, and the fact that $\beta\le \sigma\le \theta<1$, we have
		$$
		\| \nabla u\|_{L^\infty (\Omega_{\varepsilon^{1/2}})} \le C\varepsilon^{\frac{\beta-1}{2}}+C+C \big( \frac{1}{k_1} + \frac{1}{k_2} \big)\varepsilon^\frac{\theta-1}{2}\le C\varepsilon^{\frac{\beta-1}{2}}.
		$$
		Similarly, for any $R \in (\varepsilon^{1/2}, R_0/4)$,  the estimates \eqref{v_bar_L2_sphere_new} in Proposition \ref{prop_grad_v_bar_control} with $s=s_0$ and \eqref{v-v_bar} imply
		$$
		\dashnorm{u-(u)_{\Omega_{4R}\setminus \Omega_{R/4}}}{L^2(\Omega_{4R}\setminus \Omega_{R/4})} \le C R^{\beta}, 
		$$
		which implies, by \eqref{gradient-1/2} with $g_1(0') = g_2(0') = 0$,
		$$
		\| \nabla u\|_{L^\infty ({\Omega_{2R}\setminus \Omega_{R/2}})} \le CR^{\beta-1}
		$$
		for any $R \in (\varepsilon^{1/2}, R_0/4)$. Thus, we obtain a better bound for the gradient
		\begin{equation}\label{Du-1st}
			|\nabla u(x)|\le C(\varepsilon+|x'|^2)^\frac{\beta-1}{2}, \quad x\in \Omega_{R_0}.
		\end{equation}
		Next, we derive the $C^\beta$ estimate \eqref{u_holder} directly from \eqref{Du-1st}.
		For $x,y\in\Omega_{2\varepsilon^{1/2}}$, by \eqref{Du-1st}, \eqref{fg_0}, and \eqref{fg_1}, we have
		$$|u(x)-u(y)|\le \|Du\|_{L^\infty(\Omega_{2\varepsilon^{1/2}})}|x-y|^\beta (|x|+|y|)^{1-\beta}\le C|x-y|^\beta,
		$$
		and therefore
		\begin{equation}\label{holder1}   [u]_{C^\beta(\Omega_{2\varepsilon^{1/2}})}\le C.
		\end{equation}
		Similarly, for $R\in (\varepsilon^{1/2}, R_0/4)$, \eqref{Du-1st} also implies that
		\begin{equation}\label{holder2}   [u]_{C^\beta(\Omega_{2R}\setminus\Omega_{R/2})}\le C.
		\end{equation}
		Now we show that 
		\begin{equation}\label{holder3}   [u]_{C^\beta(\Omega_{R_0/4}\setminus\Omega_{\varepsilon^{1/2}})}\le C.
		\end{equation}
		For $x,y\in \Omega_{R_0/4}\setminus \Omega_{\varepsilon^{1/2}}$, assume without loss of generality that $|x'|\ge |y'|$. Let $R=|x'|$ and $m$ be the nonnegative integer such that $2^{-m-1}R< |y'|\le 2^{-m}R$. If $m=0$,  $x,y\in \Omega_{2R}\setminus \Omega_{R/2}$. Then \eqref{holder2} implies that 
		$$|u(x)-u(y)|\le C|x-y|^\beta.$$ Otherwise, we have $m\ge 1$. Then $$|x-y|\ge |x'|-|y'|\ge (1-2^{-m})R\ge R/2.$$
		We choose $z_k\in \Omega_R$ such that $|z_k'|=2^{-k}R$ for $k=1,2,\dots,m$, and denote $z_0=x$, $z_{m+1}=y$. For $k=0,1,\dots, m$, we know that $z_k, z_{k+1}\in \Omega_{2R_k}\setminus\Omega_{R_k/2}$, where $R_k=2^{-k}R$. Thus by \eqref{holder2}, \eqref{fg_0}, and \eqref{fg_1}, 
		$$
		\begin{aligned}
			&|u(x)-u(y)|
			\\&\le 
			\sum_{k=0}^m |u(z_k)-u(z_{k+1})|
			\le C\sum_{k=0}^m |z_k-z_{k+1}|^\beta\le C\sum_{k=0}^m 2^{-k\beta}R^\beta\le C|x-y|^\beta.
		\end{aligned}
		$$
		The estimate \eqref{holder3} is proved. Finally, for $x\in \Omega_{R_0/4}\setminus\Omega_{2\varepsilon^{1/2}}$, $y\in \Omega_{\varepsilon^{1/2}}$,
		$$|x-y|\ge |x'|-\varepsilon^{1/2}\ge |x'|/2\ge \varepsilon^{1/2}.$$
		{We choose} $z\in \Omega_{R_0/4}$ such that $|z'|=\varepsilon^{1/2}$. Then by \eqref{holder1}, \eqref{holder3}, \eqref{fg_0}, and \eqref{fg_1}, 
		\begin{equation}\label{holder4}
			\begin{aligned}
				&|u(x)-u(y)|\\&\le |u(x)-u(z)|+|u(z)-u(y)|\le C |x-z|^\beta+C|z-y|^\beta \le C |x'|^\beta\le C|x-y|^\beta.
			\end{aligned}
		\end{equation}
		The H\"older estimate \eqref{u_holder} follows from \eqref{holder1}, \eqref{holder3}, and \eqref{holder4}. By \eqref{u_holder} and \eqref{u_C1_outside}, \eqref{u-_1+holder} follows again from the $C^{1,\beta}$ estimate for the Neumann boundary problem. See \cite{Lieberman87}*{Theorem 1.2}. The {first assertion of the} lemma is proved. 
        
        In the case where the domain $\widetilde\Omega$ is symmetric in $x_j$ and the solution $u$ is odd in $x_j$ for some $1 \le j \le n-1$, we have
        $$\bar{v}(0) = g_1(0') = g_2(0') = 0.$$
        We may replace the gradient estimate \eqref{gradient-1/2} with \eqref{gradient-1/2a} so that the constant $C$ depends on $\gamma$ only through $\gamma_0$.
	\end{proof}
	
	Next, we subtract a linear function so that {the new $g_1 + g_2$ and $G_2$ are} sufficiently small around $x'=y'=0$.
	
	Consider \begin{equation}\label{subtr}
		\tilde u_+=u_+-(a+b\cdot x).
	\end{equation} By direct computations, we have 
	on $\Gamma_+$,
	$$\tilde u_+ + \gamma \partial_\nu \tilde u_+=g_1-\gamma  \Big(  \frac{-b'\cdot\nabla h_1(x')+b_n}{\sqrt{1 + |\nabla h_1(x')|^2}} \Big) - a-b'\cdot x'-b_n\Big(\frac{\varepsilon}{2}+h_1(x')\Big) := \tilde g_1(x'),$$
	and on $\Gamma_-$
	$$\tilde u_+ + \gamma \partial_\nu \tilde u_+=g_2-\gamma  \Big(  \frac{ b'\cdot\nabla h_2(x')-b_n}{\sqrt{1 + |\nabla h_2(x')|^2}} \Big) - a-b'\cdot x'-b_n\Big(-\frac{\varepsilon}{2}+h_2(x')\Big) := \tilde g_2(x').$$
	First, similar to Lemma \ref{lem:u_holder}, we choose
	$$a=\frac{g_1(0')+g_2(0')}{2}, \quad b_n=\frac{g_1(0')-g_2(0')}{2\gamma+\varepsilon},$$
	so that $\tilde g_1(0')=\tilde g_2(0')=0$.
	By \eqref{fg_0}, \eqref{def:c_2} and $\|\nabla g_i\|_{C^{\beta}}\le C/k_i$ as a consequence of \eqref{u-_1+holder} , we have 
	\begin{equation}\label{g1+g2}
		\begin{aligned}
			&\tilde g_1(x')+\tilde g_2(x')=\tilde g_1(x')-\tilde g_1(0)+\tilde g_2(x')-\tilde g_2(0)\\
			&=g_1(x')-g_1(0')+g_2(x')-g_2(0')-2b'\cdot x'\\
			&\quad+\gamma\Big(\frac{b'\cdot\nabla h_1(x')}{\sqrt{1 + |\nabla h_1(x')|^2}}-\frac{b'\cdot\nabla h_2(x')}{\sqrt{1 + |\nabla h_2(x')|^2}}\Big)+O(|x'|^2)
			\\&=(\nabla g_1(0')+\nabla g_2(0')-2b')\cdot x'\\
			&\quad +\gamma\Big(\frac{b'\cdot\nabla h_1(x')}{\sqrt{1 + |\nabla h_1(x')|^2}}-\frac{b'\cdot\nabla h_2(x')}{\sqrt{1 + |\nabla h_2(x')|^2}}\Big)+O(|x'|^{1+\beta}) \Big( \frac{1}{k_1} + \frac{1}{k_2} \Big)+O(|x'|^2).
		\end{aligned}
	\end{equation}
	Note that $\nabla h_1(x')- \nabla h_2(x')=2 x'+O(|x'|^{1+\sigma})$ and $0<\beta\le \sigma<1$. We then take 
	$$
	b'=\left\{
	\begin{aligned}
		&-\frac{1}{2(\gamma-1)}(\nabla g_1(0')+\nabla g_2(0')), &&\mbox{if } \gamma\neq 1,\\
		&0, && \mbox{if } \gamma= 1,
	\end{aligned}
	\right.
	$$ and 
	thus
	\begin{equation}\label{tildeg}\left\{
		\begin{aligned}
			&|\tilde g_1(x')+\tilde g_2(x')|\le  C|x'|^{1+\beta}, &&\mbox{if } \gamma\neq 1,\\
			&|\tilde g_1(x')+\tilde g_2(x')|\le  C\Big( \frac{1}{k_1} + \frac{1}{k_2} \Big)|x'|{ +C|x'|^2}, && \mbox{if } \gamma= 1,
		\end{aligned}
		\right.
	\end{equation}
	where $C>0$ is a universal constant independent of $\varepsilon$.

	Since $\tilde u_+$ satisfies
	\begin{equation*}
		\begin{cases}
			\Delta \tilde u_+ =0&\mbox{in}~\Omega_{R_0},\\
			\tilde u_+ + \gamma \partial_\nu \tilde u_+ = \tilde g_1&\mbox{on}~\Gamma_+,\\
			\tilde u_+ + \gamma \partial_\nu \tilde u_+ = \tilde g_2&\mbox{on}~\Gamma_-,
		\end{cases}
	\end{equation*}
	applying the dimension reduction argument in Section \ref{sec4} with $\tilde u_+$ in place of $u_+$ and using \eqref{tildeg}, we obtain
	\begin{lemma}\label{lemma_reduction_2}
		Let $h_1,h_2$ satisfy \eqref{fg_0}, \eqref{def:c_2}, \eqref{fg_2}, \eqref{fg_3}, {$u$ be the restriction to $\Omega_{R_0}$ of the global solution of \eqref{robin}, with $g_i$ its interior trace on $\Gamma_\pm$}. Perform the subtraction \eqref{subtr}, the change of variables \eqref{x_to_y}, and let $v(y) = \tilde u_+(x)$, $\bar{v}$ be defined as in \eqref{v_bar_def}. Then $\bar{v}$ satisfies the equation
		\begin{equation}\label{equation_v_bar_2}
			\dv((\varepsilon+|y'|^2)\nabla \bar v) - 2\gamma^{-1} \bar{v} = \dv F + \tilde{G}_1+\tilde{G}_2 \quad \mbox{in}~B_{R_0} \subset \bR^{n-1},
		\end{equation}
		where
		\begin{equation}\label{FG_bound_b}
			\begin{aligned}
				|F(y')| &\le C\Big(\varepsilon|y'| \overline{|\partial_n v|} + |e^1| |\nabla_{y'}\overline{v}|\Big), \\
				|\tilde{G}_1(y')| &\le C\Big(\varepsilon \max_{y_n \in (-\varepsilon, \varepsilon)} |\partial_n v| + |y'|^2 |\overline{v}| +|y'|^2\Big),\\
				|\tilde{G}_2(y')|&\le C  |y'|^{1+\beta},\quad \mbox{if } \gamma\neq 1,\\
				|\tilde{G}_2(y')|&\le C  \Big( \frac{1}{k_1} + \frac{1}{k_2} \Big) |y'|,\quad \mbox{if } \gamma= 1,
			\end{aligned}
		\end{equation}
		and $C$ is some positive constant depending only on  $n$, $k_0$, $\gamma$, $\|h_1\|_{C^{2,\sigma}}$, $\|h_2\|_{C^{2,\sigma}}$, and $\|\varphi\|_{L^\infty}$.
	\end{lemma}
	{
		\begin{proof}
			By Lemma \ref{lemma_reduction}, $\bar{v}$ satisfies \eqref{equation_v_bar} with $\mu=1$, $F$, $G_1$, and $G_2$ satisfying \eqref{FG_bound} with $\tilde{g}_i$ in place of $g_i$. In particular,
			$$G_2(y')=-\gamma^{-1}\sqrt{1+ |\nabla_{y'}h_1|^2} \tilde{g}_1(y') - \gamma^{-1}\sqrt{1+ |\nabla_{y'}h_2|^2} \tilde{g}_2(y').$$
			We then define
			$$
			\tilde{G}_2(y'):=
			\left\{
			\begin{aligned}
				&G_2(y'), &&\mbox{if } \gamma\neq 1,\\
				&-\gamma^{-1} ({g}_1(y')-g_1(0')+{g}_2(y')-g_2(0')), && \mbox{if } \gamma= 1,
			\end{aligned}
			\right.
			$$   
			and
			$$\tilde{G}_1(y'):=G_1(y')+G_2(y')-\tilde{G}_2(y').$$
			For the case when $\gamma\ne 1$, \eqref{FG_bound_b} follows from \eqref{FG_bound}, \eqref{tildeg}, \eqref{fg_0}, and \eqref{def:c_2}.
			For the case when $\gamma= 1$, by \eqref{u-_1+holder}, we have $\|\nabla g_i\|_{C^{\beta}}\le C/k_i$ and the estimate for $\tilde{G}_2$ follows. It remains to prove the estimate for $\tilde{G}_1$ when $\gamma=1$. Note that by the triangle inequality, \eqref{fg_0}, \eqref{def:c_2}, and \eqref{g1+g2}, 
			$$
			\begin{aligned}
				&|G_2(y')-\tilde{G}_2(y')|\le |G_2(y')+\gamma^{-1} (\tilde g_1(y')+\tilde{g}_2(y'))|
				\\&\qquad+|-\gamma^{-1} (\tilde g_1(y')+\tilde{g}_2(y'))+\gamma^{-1} ({g}_1(y')-g_1(0')+{g}_2(y')-g_2(0'))|\\
				&\le C |y'|^2.
			\end{aligned}
			$$
			Therefore, the estimate for ${\tilde G_1}$ follows from \eqref{FG_bound} and the last inequality. The lemma is proved.
		\end{proof}
	}
	
	Now we are ready to prove Theorems \ref{thm_upperbound} and \ref{thm_upperbound_2}.
	\begin{proof} [Proof of Theorem \ref{thm_upperbound}]
		Without loss of generality, we assume that $\|\varphi\|_{L^\infty(\partial\Omega)}= 1$. Recall that $\mu = 1$. We will prove \eqref{grad_u_upperbound1}--\eqref{grad_u_upperbound2} with $\alpha$ given as in \eqref{alpha_new}. We perform the subtraction \eqref{subtr}, make the change of variables \eqref{x_to_y}, let $v(y) = \tilde{u}(x)$, and let $\bar{v}$ be the vertical average defined as in \eqref{v_bar_def}. { Then by Lemma \ref{lemma_reduction_2}, $\bar{v}$ satisfies \eqref{equation_v_bar_2} with $F$, $\tilde{G}_1$, $\tilde{G}_2$ satisfying \eqref{FG_bound_b}. Since $|\nabla(a+b\cdot x)|=|b|\le C$, it suffices to prove the gradient estimates {\eqref{grad_u_upperbound1}, \eqref{grad_u_upperbound3}, and \eqref{grad_u_upperbound2}} for $\tilde{u}$. For simplicity of notation, we still denote $\tilde u$ by $u$, $\tilde{g}_i$ by $g_i$, and $\tilde{G}_i$ by $G_i$.
			By Lemma \ref{lem:u_holder}, $g_i$ is Lipschitz and therefore
			\begin{equation}\label{gi_c_beta}
				[g_1]_{C^\beta_{x'}(\Gamma^+_{x_0,2r})} + [g_2]_{C^\beta_{x'}(\Gamma^-_{x_0,2r})}\le Cr^{1-\beta}.
			\end{equation}
		}
	{Following the proof of} Lemma \ref{lem:u_holder}, one can obtain the following estimates
		
		\begin{equation}\label{grad_v_bar_start_2}
			\|\nabla \bar{v}\|_{\varepsilon,0,s_0+1,B_{R_0}} \le C,
		\end{equation}
		where $s_0 = \frac{1}{2}$,
		\begin{equation}
			\label{grad_n_v_bound_2}
			|\partial_n v| \le C(\varepsilon+|y'|^2)/\varepsilon \quad \mbox{in}~~Q_{R_0,\varepsilon},
		\end{equation}
		where $Q_{R_0,\varepsilon}$ is defined as in \eqref{Q_s_t}, and thus
		\begin{equation}
			\label{v-v_bar_2}
			|v(y', y_n) - \bar{v}(y')| \le 2\varepsilon \max_{y_n\in (- \varepsilon , \varepsilon)} |\partial_n v(y',y_n)| \le C (\varepsilon + |y'|^2) \quad \mbox{in}~~Q_{R_0,\varepsilon}.
		\end{equation}
		On the other hand, $e^i$ in \eqref{a_ij_formula} can be bounded by $C|y'|^{2+\sigma}$. 
		
		Therefore, by Lemma \ref{lemma_reduction_2}, \eqref{grad_v_bar_start_2}, and \eqref{grad_n_v_bound_2}, $F,G_1,G_2$ satisfy
		\begin{equation}\label{F_G1}
			\|F\|_{\varepsilon, \sigma, s_0, B_{R_0}} + \|G_1\|_{\varepsilon, \sigma-1, s_0, B_{R_0}}\le C,
		\end{equation}
		and for $s>0$,
		\begin{equation}\label{G2-1}
			\|G_2\|_{\varepsilon, \sigma-1, s, B_{R_0}}\le C , \quad \mbox{when } \gamma\neq 1, \sigma\le 2s+\beta,
		\end{equation}
		
		\begin{equation}\label{G2-2}
			\|G_2\|_{\varepsilon, \sigma-1, s, B_{R_0}}\le C \Big( \frac{1}{k_1} + \frac{1}{k_2} \Big), \quad \mbox{when } \gamma=1, \sigma\le 2s.
		\end{equation}
		Let $G:=G_1+G_2$.
		
		Next, we consider two separate cases: $\gamma\ne 1$ and $\gamma=1$.
		
		\textit{Case 1:} When $\gamma\ne 1$,
		by decreasing $\sigma$ if necessary, we may assume that $\sigma\le \beta$ (so that \eqref{G2-1} is satisfied for any $s>0$), and $1+\sigma-2s_0=\sigma\ne\widehat\alpha$, where $\widehat{\alpha} = \min\{\alpha, 1\}$.
		Then by \eqref{F_G1} and \eqref{G2-1} with $s=s_0$, the estimate \eqref{v_bar_L2_sphere_old} in Proposition \ref{prop_grad_v_bar_control} implies that 
		\begin{equation}\label{v_bar_decay}
			\dashnorm{\bar{v} - (\bar{v})_{\partial B_{2\varepsilon^{1/2}}}}{L^2(\partial B_{2\varepsilon^{1/2}})}\le C \varepsilon^{\tilde{\alpha}/2},
		\end{equation}
		where $\tilde{\alpha} = \min\{\widehat\alpha , 1+ \sigma - 2s_0 \}$.
		
		Since $u$ is odd in $x_j$ for some $1 \le j \le n-1$, $\bar{v}$ is also odd in $x_j$. In particular, this implies
		$$
		(\bar{v})_{\partial B_R} = 0 \quad \forall R \in(0,R_0).
		$$
		Then by \eqref{v_bar_decay} and \eqref{v-v_bar_2}, we have
		\begin{equation}
			\begin{aligned}\label{v_L2_control}
				\dashnorm{u}{L^2(\Omega_{2\varepsilon^{1/2}})} \le& C \dashnorm{v}{L^2(Q_{2\varepsilon^{1/2} , \varepsilon})} \\
				\le& C( \dashnorm{v-\bar{v}}{L^2(Q_{2\varepsilon^{1/2} , \varepsilon})} + \dashnorm{\bar v}{L^2(Q_{2\varepsilon^{1/2} , \varepsilon})})\\
				\le& C \varepsilon^{\tilde\alpha/2}.
			\end{aligned}
		\end{equation}
		By \eqref{v_L2_control}, \eqref{gi_c_beta}, and \eqref{gradient-1/2a} with $g_1(0') = g_2(0') = 0$, we have
		$$
		\| \nabla u\|_{L^\infty (\Omega_{\varepsilon^{1/2}})} \le C\varepsilon^{\frac{\tilde\alpha-1}{2}}.
		$$
		Similarly, for any $R \in (\varepsilon^{1/2}, R_0/4)$,  the estimates \eqref{v_bar_L2_sphere_old} in Proposition \ref{prop_grad_v_bar_control} with $s=s_0$ and \eqref{v-v_bar} imply
		$$
		\dashnorm{u}{L^2(\Omega_{4R}\setminus \Omega_{R/2})} \le C R^{\tilde{\alpha}}, 
		$$
		which implies, by \eqref{gradient-1/2a} with $g_1(0') = g_2(0') = 0$ and \eqref{gi_c_beta},
		$$
		\| \nabla u\|_{L^\infty (\Omega_{4R}\setminus \Omega_{R/2})} \le CR^{\tilde\alpha-1}
		$$
		for any $R \in (\varepsilon^{1/2}, R_0/4)$. Therefore, we have improved the upper bound $|\nabla u(x)| \le C(\varepsilon + |x'|^2)^{-s_0}$ to $|\nabla u(x)| \le C(\varepsilon + |x'|^2)^{\frac{\tilde{\alpha}-1}{2}}$, where $\frac{\tilde{\alpha}-1}{2} = \min\left\{0,\frac{\alpha - 1}{2}, -s_0 + \frac{\sigma}{2} \right\}$. If $-s_0 + \frac{\sigma}{2} <\min\left\{0, \frac{\alpha - 1}{2}\right\}$, we take $s_1 = s_0 - \frac{\sigma}{2}$ and repeat the argument above. Note that the improved upper bound $|\nabla u(x)| \le C(\varepsilon + |x'|^2)^{-s_1}$ implies \eqref{F_G1} with $s_1$ in place of $s_0$. Let $s_k=s_0-k\frac{\sigma}{2}$ for $k\in \mathbb{Z}_+$. We may decrease $\sigma$ if necessary so that $1+\sigma -2s_k\ne \widehat\alpha$ for any $k\in \mathbb{Z}_+$.
		After {repeating the argument finitely many times}, we obtain the estimate \eqref{grad_u_upperbound1} or \eqref{grad_u_upperbound2}.
		
		\textit{Case 2:} When $\gamma= 1$, we know that $\widehat\alpha=\min\{\alpha,1\}=1$. By decreasing $\sigma$ if necessary, we may assume that $\sigma=1/(N+1)$ for some large integer $N$. Let $s_0=1/2$ and $s_k=s_0-k\frac{\sigma}{2}$ for $k=1,2,\dots, N+1$. Then clearly $1+\sigma-2s_k<1$ 
		for $k=0,1,\dots,N-1$ { and ${s_{N+1}=0}$}. Similar to case 1, by \eqref{G2-2} in place of \eqref{G2-1} with $s=s_k$, $k=0,1,\dots, N-1$, we can iterate the estimate \eqref{v_bar_L2_sphere_old} in Proposition \ref{prop_grad_v_bar_control} to improve the upper bound
		$|\nabla u(x)| \le C(\varepsilon + |x'|^2)^{-s_0}$ to $|\nabla u(x)| \le C(\varepsilon + |x'|^2)^{-s_N}$. From the improved upper bound $|\nabla u(x)| \le C(\varepsilon + |x'|^2)^{-s_N}$, one can obtain estimate \eqref{F_G1} with $s_N$ in place of $s_0$. Namely, we have
		\begin{equation}\label{F_G11}
			\|F\|_{\varepsilon, \sigma, s_N, B_{R_0}} + \|G_1\|_{\varepsilon, \sigma-1, s_N, B_{R_0}}\le C.
		\end{equation}
		
		We split $\bar v = v_1 + v_2$ such that $v_1$ satisfies
		\begin{equation*}
			\left\{
			\begin{aligned}
				\dv((\varepsilon+|y'|^2)\nabla  v_1) - 2\gamma^{-1} v_1 &= \dv F + G_1 \quad \mbox{in}~B_{R_0} \subset \bR^{n-1},\\
				v_1 &= \bar{v}  \quad \mbox{on}~ \partial B_{R_0},
			\end{aligned}\right.
		\end{equation*}
		while $v_2$ satisfies
		\begin{equation*}
			\left\{
			\begin{aligned}
				\dv((\varepsilon+|y'|^2)\nabla  v_2) - 2\gamma^{-1} v_2 &= G_2 \quad \mbox{in}~B_{R_0} \subset \bR^{n-1},\\
				v_2 &= 0  \quad \mbox{on}~ \partial B_{R_0}.
			\end{aligned}\right.
		\end{equation*}

		Note that $1+\sigma-2s_{N}=1$. Then 
		by \eqref{F_G11}, \eqref{G2-2} with $s=s_N$, and \eqref{v-v_bar}, the estimate \eqref{v_bar_L2_sphere} in Proposition \ref{prop_grad_v_bar_control} implies that for any $R \in (0, R_0/2)$,
		\begin{equation}\label{est:v1v2}
			\left\{
			\begin{aligned}
				\dashnorm{v_1 - (v_1)_{\partial B_R}}{L^2(\partial B_R)} &\le C R |\log R|,\\
				\dashnorm{v_2 - (v_2)_{\partial B_R}}{L^2(\partial B_R)} &\le C \Big( \frac{1}{k_1} + \frac{1}{k_2} \Big) R |\log R|.
			\end{aligned}\right.
		\end{equation}
		Then
		$$
		\dashnorm{{\bar v}}{L^2(\partial B_R)} \le \dashnorm{v_1 - (v_1)_{\partial B_R}}{L^2(\partial B_R)} + \dashnorm{v_2 - (v_2)_{\partial B_R}}{L^2(\partial B_R)} \le C R |\log R|.
		$$
		Therefore,
		\begin{equation*}
			\begin{aligned}
				\dashnorm{u}{L^2(\Omega_{2\varepsilon^{1/2}})} \le& C \dashnorm{v}{L^2(Q_{2\varepsilon^{1/2} , \varepsilon})} \\
				\le& C( \dashnorm{v-\bar{v}}{L^2(Q_{2\varepsilon^{1/2} , \varepsilon})} + \dashnorm{\bar v}{L^2(Q_{2\varepsilon^{1/2} , \varepsilon})})\\
				\le& C \varepsilon^{1/2} |\log \varepsilon|,
			\end{aligned}
		\end{equation*}
		and similarly
		$$
		\dashnorm{u}{L^2(\Omega_{4R}\setminus \Omega_{R/2})} \le C R|\log R|
		$$
		for any $R \in (\varepsilon^{1/2}, R_0/4)$. Then by \eqref{gradient-1/2a} with $g_1(0') = g_2(0') = 0$ and \eqref{gi_c_beta}, we have
		\begin{equation}\label{est:Du_temp}
			|\nabla u (x)| \le C |\log (\varepsilon + |x'|^2)|.
		\end{equation}
		
		Recall that our goal is the estimate \eqref{grad_u_upperbound3}, which has a more precise dependence on $k_1, k_2$ than that in \eqref{est:Du_temp}. This can be achieved by improving the estimate for $v_1$ in \eqref{est:v1v2}. In fact, due to \eqref{est:Du_temp}, we have
		$$
		\|F\|_{\varepsilon, \sigma, s, B_{R_0}} + \|G_1\|_{\varepsilon, \sigma-1, s, B_{R_0}}\le C
		$$
		for any $s > 0$. We choose $s = \frac{\sigma}{4}$ so that $1 + \sigma - 2s =  1 + \sigma/2 > 1$. Then we can apply the estimate \eqref{v_bar_L2_sphere_old} in Proposition \ref{prop_grad_v_bar_control} to obtain
		$$
		\dashnorm{v_1 - (v_1)_{\partial B_R}}{L^2(\partial B_R)} \le C R,
		$$
		which improves \eqref{est:v1v2} by removing the logarithmic factor. Then the estimate \eqref{grad_u_upperbound3} follows by repeating the same argument above.
	\end{proof}
	
	\begin{proof} [Proof of Theorem \ref{thm_upperbound_2}]
		Without loss of generality, we assume that $\|\varphi\|_{L^\infty(\partial\Omega)}= 1$. Recall that $\mu = 1$. We will prove \eqref{grad_u_upperbound1_zero} and \eqref{grad_u_upperbound2_zero} with $\alpha_0$ given as in \eqref{alpha_new_zero}. As in the proof of Theorem \ref{thm_upperbound}, we perform the subtraction \eqref{subtr}, make the change of variables \eqref{x_to_y}, and let $v(y) = \tilde{u}(x)$, and let $\bar{v}$ be the vertical average defined as in \eqref{v_bar_def}. { Then by Lemma \ref{lemma_reduction_2}, $\bar{v}$ satisfies \eqref{equation_v_bar_2} with $F$, $\tilde{G}_1$, $\tilde{G}_2$ satisfying \eqref{FG_bound_b}. Since $|\nabla(a+b\cdot x)|=|b|\le C$, it suffices to prove the gradient estimates {\eqref{grad_u_upperbound1}, \eqref{grad_u_upperbound3}, and \eqref{grad_u_upperbound2}} for $\tilde{u}$. For simplicity of notation, we still denote $\tilde u$ by $u$, $\tilde{g}_i$ by $g_i$, and $\tilde{G}_i$ by $G_i$.} Note that
		the estimates
		\eqref{gi_c_beta}
		--
		\eqref{v-v_bar_2} still hold.
		On the other hand, $e^i$ in \eqref{a_ij_formula} can be bounded by $C|y'|^{2+\sigma}$. 
		
		Therefore, by Lemma \ref{lemma_reduction_2}, \eqref{grad_v_bar_start_2}, and \eqref{grad_n_v_bound_2}, $F,G_1,G_2$ satisfy
		\begin{equation}\label{F_G1_zero}
			\|F\|_{\varepsilon, \sigma, s_0, B_{R_0}} + \|G_1\|_{\varepsilon, \sigma-1, s_0, B_{R_0}}\le C,
		\end{equation}
		and 
		\begin{equation}\label{G2-1_zero}
			\|G_2\|_{\varepsilon, \sigma-1, s, B_{R_0}}\le C, \quad \mbox{when } \gamma\neq 1, \sigma\le 2s+\beta,
		\end{equation}
		
		\begin{equation}\label{G2-2_zero}
			\|G_2\|_{\varepsilon, \sigma-1, s, B_{R_0}}\le C, \quad \mbox{when } \gamma=1, \sigma\le 2s.
		\end{equation}
		Let $G:=G_1+G_2$.
		
		Next, we consider two separate cases: $\gamma\ne 1$ and $\gamma=1$.
		
		\textit{Case 1:} When $\gamma\ne 1$,
		by decreasing $\sigma$ if necessary, we may assume that $\sigma\le \beta$ (so that \eqref{G2-1_zero} is satisfied for any $s>0$), and $1+\sigma-2s_0=\sigma\ne\widehat\alpha_0$, where $\widehat{\alpha}_0 = \min\{\alpha_0, 1\}$.
		Then by \eqref{F_G1_zero} and \eqref{G2-1_zero} with $s=s_0$, the estimate \eqref{v_bar_L2_sphere_new} in Proposition \ref{prop_grad_v_bar_control} implies that 
		\begin{equation}\label{v_bar_decay_new}
			\dashnorm{\bar{v} - \bar{v}(0)}{L^2(\partial B_{2\varepsilon^{1/2}})}\le C \varepsilon^{\tilde{\alpha}_0/2},
		\end{equation}
		where $\tilde{\alpha}_0 := \min\{\widehat\alpha_0 , 1+ \sigma - 2s_0 \}$.
		
		Then by \eqref{v_bar_decay_new}, \eqref{v-v_bar_2}, and the triangle inequality, we have
		\begin{equation}
			\begin{aligned}\label{v_L2_control_1_zero}
				\dashnorm{u-(u)_{\Omega_{2\varepsilon^{1/2}}}}{L^2(\Omega_{2\varepsilon^{1/2}})} \le &\dashnorm{u-\bar{v}(0)}{L^2(\Omega_{2\varepsilon^{1/2}})}  
				\\ \le &  C \dashnorm{v-\bar v(0)}{L^2(Q_{2\varepsilon^{1/2} , \varepsilon})} \\
				\le& C( \dashnorm{v-\bar{v}}{L^2(Q_{2\varepsilon^{1/2} , \varepsilon})} + \dashnorm{\bar v-\bar{v}(0)}{L^2(Q_{2\varepsilon^{1/2} , \varepsilon})})\\
				\le& C \varepsilon^{\tilde\alpha_0/2}.
			\end{aligned}
		\end{equation}
		By \eqref{v_L2_control_1_zero}, \eqref{gi_c_beta}, and \eqref{gradient-1/2} with $g_1(0') = g_2(0') = 0$, we have
		$$
		\| \nabla u\|_{L^\infty (\Omega_{\varepsilon^{1/2}})} \le C\varepsilon^{\frac{\tilde\alpha_0-1}{2}}.
		$$
		Similarly, for any $R \in (\varepsilon^{1/2}, R_0/4)$,  the estimates \eqref{v_bar_L2_sphere_new} in Proposition \ref{prop_grad_v_bar_control} with $s=s_0$ and \eqref{v-v_bar} imply
		$$
		\dashnorm{u-(u)_{\Omega_{4R}\setminus \Omega_{{R/4}}}}{L^2(\Omega_{4R}\setminus \Omega_{{R/4}})} \le C R^{\tilde\alpha_0}, 
		$$
		which implies, by \eqref{gradient-1/2} with $g_1(0') = g_2(0') = 0$ and \eqref{gi_c_beta},
		$$
		\| \nabla u\|_{L^\infty (\Omega_{{2R}}\setminus \Omega_{R/2})} \le CR^{\tilde\alpha_0-1}
		$$
		for any $R \in (\varepsilon^{1/2}, R_0/4)$. Therefore, we have improved the upper bound $|\nabla u(x)| \le C(\varepsilon + |x'|^2)^{-s_0}$ to $|\nabla u(x)| \le C(\varepsilon + |x'|^2)^{\frac{\tilde{\alpha}_0-1}{2}}$, where $\frac{\tilde{\alpha}_0-1}{2} = \min\left\{0,\frac{\alpha_0 - 1}{2}, -s_0 + \frac{\sigma}{2} \right\}$. If $-s_0 + \frac{\sigma}{2} <\min\left\{0, \frac{\alpha_0 - 1}{2}\right\}$, we take $s_1 = s_0 - \frac{\sigma}{2}$ and repeat the argument above. Note that the improved upper bound $|\nabla u(x)| \le C(\varepsilon + |x'|^2)^{-s_1}$ implies \eqref{F_G1_zero} with $s_1$ in place of $s_0$. Let $s_k=s_0-k\frac{\sigma}{2}$ for $k\in \mathbb{Z}_+$. We may decrease $\sigma$ if necessary so that $1+\sigma -2s_k\ne \widehat\alpha_0$ for any $k\in \mathbb{Z}_+$.
		After repeating the argument finitely many times, we obtain the estimate \eqref{grad_u_upperbound1_zero} or \eqref{grad_u_upperbound2_zero}.
		
		\textit{Case 2:} When $\gamma= 1$, {in contrast to Case 2 in the proof of Theorem \ref{thm_upperbound},} we know that $\widehat\alpha_0=\min\{\alpha_0,1\}=\alpha_0<1$. Let $s_k=s_0-k\frac{\sigma}{2}$ for $k\in \mathbb{Z}_+$. We may decrease $\sigma$ if necessary so that $\sigma< 1-\alpha_0$ and $1 -2s_k\ne \alpha_0$ for any $k\in \mathbb{Z}_+$. Note that we can repeat the argument as in Case 1 as long as $\sigma\le 2s_k$, so that \eqref{G2-2_zero} is satisfied. Let $j\in \mathbb{Z}_+$ such that $1 -2s_{j-1}< \alpha_0$ and $1 -2s_{j}> \alpha_0$. Then $\sigma<1-\alpha_0<2s_{j-1}\le 2s_{k}$ for $k\le j-1$ and one can repeat the argument as in Case 1, {using \eqref{G2-2_zero} instead of \eqref{G2-1_zero},} to obtain the upper bound $|\nabla u(x)| \le C(\varepsilon + |x'|^2)^\theta$,
		where $\theta:=\min\{\frac{\alpha_0-1}{2},-s_{j-1}+\frac{\sigma}{2}\}=\min\{\frac{\alpha_0-1}{2},-s_j\}=\frac{\alpha_0-1}{2}$.
		Therefore, the estimate {\eqref{grad_u_upperbound2_zero}} is proved.
	\end{proof}
	
	\section{\texorpdfstring{{Two-dimensional case}}{Two-dimensional case}}\label{sec:two-dimensional}
	In this section, we give an outline of the proof for Theorem \ref{thm_upperbound_3} by highlighting the differences in {the two-dimensional case} compared with the three- or {higher-dimensional cases} {at the end of} the last section. {We use the same notation} as in the last section.

	First, by noticing that $\alpha_0=\alpha$ when $n=2$, and $r^{{\widehat \alpha}}$ is a supersolution of \eqref{zero-mode} when $n=2$ {and $\widehat \alpha=\min\{\alpha,1\}$}, we have the following better estimate for the zero mode instead of {Lemma} \ref{lemma_v1_zero_mode}.
	\begin{lemma}
		\label{lemma_v1_zero_mode_2d}
		For $n {=} 2$, $\gamma > 0$, and $\varepsilon > 0$, let $v_1 \in H^1(B_{R_0})$ satisfy \eqref{equation_v1}.
		Then for any $0 < \rho < R < R_0$, we have
		\begin{equation*}
			|(v_1)_{\partial B_\rho}-v_1(0)| \le \Big(\frac{\rho}{R} \Big)^{{\widehat \alpha}} |(v_1)_{\partial B_R}-v_1(0)|,
		\end{equation*}
		where $\alpha$ is given in \eqref{alpha_new}.
	\end{lemma}
	
	Then a similar iteration argument (as in the proof of Proposition \ref{prop_grad_v_bar_control}) using Lemmas \ref{lemma_v1}, \ref{lemma_v2}, and \ref{lemma_v1_zero_mode_2d} (instead of \ref{lemma_v1_zero_mode}) yields 
	\begin{proposition}
		\label{prop_grad_v_bar_control_2d}
		For $n = 2$, $0 \le s \le 1/2$, $\gamma > 0$, $\sigma> 0$, and $\varepsilon > 0$,
		let $\bar{v} \in H^1(B_{R_0})$ be a solution of
		\begin{equation*}
			\dv((\varepsilon+|y'|^2)\nabla \bar v) -2\gamma^{-1}\bar{v} = \dv F + G \quad\text{in}\,\,B_{R_0}\subset \bR^{n-1},
		\end{equation*}
		where $F,G$ satisfy
		\begin{equation*}
			\|F\|_{\varepsilon, \sigma,s,B_{R_0}}< \infty , \quad \|G\|_{\varepsilon, \sigma-1,s,B_{R_0}} < \infty.
		\end{equation*}
		Let $\alpha$ be the constant given in \eqref{alpha_new}, and $\tilde\alpha := \min \{ \alpha, 1, 1 + \sigma - 2s \}$. There exists a positive constant $C$ depending only on $n$, $\sigma$, $s$, $R_0$, but independent of $\varepsilon$, such that for any $R \in (0, R_0/2)$, we have:
		\begin{itemize}
			\item[(i)] If $1+\sigma - 2s \neq \min\{\alpha,1\}$, then
			\begin{equation*}
				\begin{aligned}
					&\dashnorm{\bar{v} - \bar{v}(0)}{L^2(\partial B_R)}\\
					&\le C(\|F\|_{\varepsilon, \sigma,s,B_{R_0}} + \|G\|_{\varepsilon, \sigma-1,s,B_{R_0}} +  \dashnorm{\bar v - {\bar{v}(0)}}{L^2(\partial B_{R_0})})  R^{\tilde\alpha}.
				\end{aligned}
			\end{equation*}
			\item[(ii)] If $1+\sigma - 2s = \min\{\alpha,1\}$, then
			\begin{equation*}
				\begin{aligned}
					&\dashnorm{\bar{v} - \bar{v}(0)}{L^2(\partial B_R)}\\
					&\le C(\|F\|_{\varepsilon, \sigma,s,B_{R_0}} + \|G\|_{\varepsilon, \sigma-1,s,B_{R_0}} +  \dashnorm{\bar v -  {\bar{v}(0)}}{L^2(\partial B_{R_0})})  R^{\tilde\alpha}\log(R_0/R).
				\end{aligned}
			\end{equation*}
		\end{itemize}
	\end{proposition}
	
	With Proposition \ref{prop_grad_v_bar_control_2d} in place of \ref{prop_grad_v_bar_control}, one can use a similar bootstrap argument (as in the proof of Theorem \ref{thm_upperbound}) to prove Theorem \ref{thm_upperbound_3}. We leave the details to the interested reader.

	\section{Optimality}\label{sec:optimality}
	In this section, we prove Theorem \ref{main_thm}. Without loss of generality, we assume $\mu = 1$ again throughout this section. The following ordinary differential operator
	\begin{equation}\label{def_L}
		Lh :=h''(r) + \left( \frac{n-2}{r} + \frac{2r}{\varepsilon + r^2} \right) h'(r) - \left( \frac{n-2}{r^2} + \frac{2}{\gamma} \frac{1}{\varepsilon+r^2} \right)  h(r)
	\end{equation}
	will play an important role in our proof. First, we recall the following lemma in \cite{DongYangZhu26}.
	\begin{lemma}[\cite{DongYangZhu26}*{Lemma 7.1}]
		\label{ODE_lemma}
		For $\varepsilon > 0$, $\gamma > 1$, $n \ge 2$, let $L$ be the operator defined as in \eqref{def_L}. There exists a unique solution $h \in C([0,1/2]) \cap C^\infty((0,1/2])$ of
		\begin{equation}
			\label{Homogeneous_ODE}
			Lh = 0, \quad 0 < r <1/2
		\end{equation}
		satisfying  $h(0) = 0$ and $h(1/2) = 1/2$.
		Moreover, there exists a positive constant $C(\varepsilon)$ depending only on $n$ and $\varepsilon$, such that
		\begin{equation}
			\label{h_bounds}
			r  < h(r) < \min\{ C(\varepsilon)r, r^\alpha\} \quad \mbox{for}~~0 < r < 1/2,
		\end{equation}
		where $\alpha$ is given by \eqref{alpha_new} and $h$ is strictly increasing in $[0,1/2]$.
	\end{lemma}
	Next, we establish better estimates of $h$ using a rescaling argument and the maximum principle.
	\begin{lemma}\label{lemma_subsolution}
		For $\varepsilon > 0$, $\gamma > 1$, $n \ge 2$, let $h \in C([0,1/2]) \cap C^\infty((0,1/2])$ be the solution to \eqref{Homogeneous_ODE} and $\alpha$ be given by \eqref{alpha_new}. Then
		\begin{equation}\label{h_lowerbound}
			h(r) > {2^{\alpha-1}}\Big(r^\alpha - ({2}\sqrt\varepsilon)^{\alpha} \Big)_+, \quad r \in(0,1/2).
		\end{equation} Moreover, for some positive constants $C_1$, $C_2$ depending only on $n$, 
		\begin{equation}\label{h_upper_lower}
			C_1\varepsilon^{\frac{\alpha-1}{2}}r\le h(r)\le C_2\varepsilon^{\frac{\alpha-1}{2}}r, \quad r\in(0,{\min
            \{1/2,{4}\sqrt{\varepsilon}\}}).
		\end{equation}
	\end{lemma}
	\begin{proof}
		Denote $v(r) = \Big(r^\alpha - (c\sqrt\varepsilon)^{\alpha} \Big)_+$ for some positive constant $c$ to be determined. When $r > c\sqrt\varepsilon$, a direct computation gives
\begin{align*}
 Lv =& - 2 \varepsilon \left( \alpha - \frac{1}{\gamma} \right) \frac{ r^{\alpha - 2}}{\varepsilon+r^2} + \left( \frac{n-2}{r^2} + \frac{2}{\gamma} \frac{1}{\varepsilon+r^2} \right) (c\sqrt\varepsilon)^{\alpha} \\
 \ge & \frac{c^{\alpha-2}\varepsilon^{\frac{\alpha}{2}}}{\varepsilon + r^2} \left( -2\alpha + \frac{2}{\gamma} + c^2\left( n-2 + \frac{2}{\gamma} \right) \right).
\end{align*}
When $n \ge 3$, it is clear that one can choose $c = 2$ such that $Lv > 0$. When $n = 2$,
$$
\alpha = \frac{-1 + \sqrt{1+ 8/\gamma}}{2} = \frac{8/\gamma}{2(1 + \sqrt{1+ 8/\gamma})} \le \frac{2}{\gamma}.
$$
Again, one can choose $c = 2$ such that $Lv > 0$. Then \eqref{h_lowerbound} follows from the maximum principle.
		
        For the proof of \eqref{h_upper_lower}, first we assume that  $0<\varepsilon<1/64$. We consider the following rescaling of $h$. For $s\in{[0,(2\sqrt{\varepsilon})^{-1}]}$, let $r=\sqrt{\varepsilon}s$ and $$\hat{h}(s):=\varepsilon^{-\alpha/2} h(r):= \varepsilon^{-\alpha/2} h(\sqrt{\varepsilon}s).$$ 
		Then $\hat{h}$ satisfies
		\begin{equation*}
			\hat{L}\hat{h}:=\hat{h}''(s)+\Big(\frac{n-2}{s}+\frac{2s}{1+s^2}\Big)\hat{h}'(s)-\Big(\frac{n-2}{s^2}+\frac{2}{\gamma}\frac{1}{1+s^2}\Big)\hat{h}(s)=0.
		\end{equation*}
		By Lemma \ref{ODE_lemma} and \eqref{h_lowerbound}, we have 
		$\hat{h}(0)=0$, $2^{2\alpha-1}(2^\alpha-1)\le\hat{h}(4)\le 4^\alpha$, and $\hat{h}(s)\le s^\alpha$ for $s\in[0,(2\sqrt{\varepsilon})^{-1}]$. 
		Since $\hat{L}s=\frac{2s}{1 + s^2} \left( 1 - \frac{1}{\gamma} \right) > 0$, by the maximum principle, one can obtain 
		\begin{equation}\label{low}
			\hat{h}(s)\ge \frac{\hat{h}(4)}{4}s\ge C_1s
            \quad \mbox{for } s\in(0,4)
		\end{equation}
		where $C_1$ is some positive constant depending only on $n$. Let $v = s - s^{3/2}/2$, by a direct computation,
		$$
		\left| {\hat L}v + \frac{1}{4} \Big( n - \frac{1}{2} \Big) s^{-\frac{1}{2}}  \right| \le 2 \Big( 1-\frac{1}{\gamma} \Big) s + \Big( 3-\frac{2}{\gamma} \Big) s^{\frac{3}{2}}{\le 2s + 3s^{3/2}}.
		$$
		Hence $\hat L v < 0$ in $(0, s_0)$, for some small $s_0=s_0(n)$.  By the maximum principle, we have 
		\begin{equation}\label{up1}
			\hat{h} \le \frac{\hat{h}(s_0)}{v(s_0)} v \le Cs\quad \text{in}\,\, (0,s_0),
		\end{equation}
		where $C = \frac{s_0^\alpha}{v(s_0)}$. On the other hand, for $s\in[s_0, 4)$,
		\begin{equation}\label{up2}
			\hat{h}(s)\le s^\alpha\le s_0^{\alpha-1}s.
		\end{equation}
		Combining \eqref{low}, \eqref{up1}, and \eqref{up2} yields 
		$$ C_1s\le \hat{h}(s)\le C_2s, \quad  \mbox{for } s\in(0,{4}),$$ where $C_1$, $C_2$ are some positive constants depending only on $n$. Therefore, \eqref{h_upper_lower} follows by reversing the scaling. {Finally, when $\varepsilon\ge {1/64}$, we apply a similar argument above with $h$ in place of $\hat h$ and $1/2$ in place of ${4}$. The lemma is proved.} 
	\end{proof}
	
	Next, we establish a linear lower bound for $u$ when $x_1\ge 0$ in the specific setting of Theorem \ref{main_thm} by constructing a subsolution.
	
	\begin{lemma}\label{linear_lower}
		For $n \ge 2$, $\gamma \ge 1$, $k_1>0$, $k_2>0$, $\varepsilon \in (0,1/4)$, let $\Omega = B_{5}$, $\cD_1 , \cD_2$ be balls of radius $1$ centered at $(0', 1 + \varepsilon/2)$ and $(0', -1- \varepsilon/2)$, respectively, $\varphi(x) = x_1$, and let $u$ be the solution of \eqref{robin}. Then we have
		\begin{equation*}
			u_+\ge x_1 \quad \mbox{in }\widetilde{\Omega}\cap\{x_1\ge0\}, \quad u_-\ge \frac{x_1}{1+k_i\gamma} \quad \mbox{in }\mathcal{D}_i\cap \{x_1\ge 0\}, \quad i=1,2,
		\end{equation*}
		\begin{equation*}
			u_+\le x_1 \quad \mbox{in }\widetilde{\Omega}\cap\{x_1\le0\}, \quad u_-\le \frac{x_1}{1+k_i\gamma} \quad \mbox{in }\mathcal{D}_i\cap \{x_1\le 0\}, \quad i=1,2.
		\end{equation*}
	\end{lemma}
	\begin{proof}
		Let $v_+:=x_1$ in $\widetilde\Omega$ and $v_-=\frac{x_1}{1+k_i\gamma} \mbox{in }\mathcal{D}_i$, $i=1,2$. Then clearly 
		$$
		\Delta v_+=0 \quad \mbox{in } \widetilde{\Omega}, \quad \Delta v_-=0 \quad \mbox{in }\mathcal{D}_i, \quad i=1,2.
		$$
		Note that on $\partial\mathcal{D}_1\cup\partial\mathcal{D}_2$,
		$\partial_\nu x_1=-x_1$. Therefore, on $\partial\mathcal{D}_i$, $i=1,2$,
		\begin{equation*}
			\partial_\nu v_++\gamma^{-1}(v_+-v_-)=\big(\frac{k_i}{1+k_i\gamma}-1\big)x_1,\quad
			k_i\partial_\nu v_-+\gamma^{-1}(v_+-v_-)=0.
		\end{equation*} 
		Since $\gamma\ge 1$ and $k_1,k_2>0$, $$\frac{k_i}{1+k_i\gamma}-1 <0.$$
		Then we set $w=v-u$, namely $w_+=v_+-u_+$, $w_-=v_--u_-$. It suffices to show that 
		$w\le 0$ in $\Omega\cap\{x_1\ge 0\}$ and $w\ge 0$ in $\Omega\cap\{x_1\le 0\}$.   
		Clearly, $w$ satisfies 
		$$  
		\Delta w_+=0 \quad \mbox{in } \widetilde{\Omega}, \quad \Delta w_-=0 \quad \mbox{in }\mathcal{D}_i, \quad i=1,2
		$$
		and 
		on $\partial\mathcal{D}_i$, $i=1,2$,
		\begin{equation*}
			\partial_\nu w_++\gamma^{-1}(w_+-w_-)=\big(\frac{k_i}{1+k_i\gamma}-1\big)x_1,
		\end{equation*}
		\begin{equation*}
			k_i\partial_\nu w_-+\gamma^{-1}(w_+-w_-)=0.
		\end{equation*} 
		Therefore, $w$ is a subsolution
		of \eqref{robin} (as defined in Lemma \ref{lem:maximum_principle}
		in $\Omega\cap \{x_1>0\}$ and a supersolution in $\Omega\cap \{x_1<0\}$. We also know that $$w_+=0 \quad \mbox{on } \partial\Omega,$$
		and by symmetry, $u=0$ when $x_1=0$ and therefore $$w=0 \quad \mbox{on } \{x_1=0\}.$$
		Then by the maximum principle Lemma \ref{lem:maximum_principle} applied in  $\Omega\cap \{x_1>0\}$ and $\Omega\cap \{x_1<0\}$ respectively, we have
		$w\le 0$ in $\Omega\cap\{x_1\ge 0\}$ and $w\ge 0$ in $\Omega\cap\{x_1\le 0\}$. The proof is complete.
	\end{proof}
	Under the assumptions of Lemma \ref{linear_lower}, by symmetry, we know that $u(0)=0$,
	$$
	g_1(0') = g_2(0') = 0,
	$$
	and
	$$
	\partial_{x_j}g_1(0') = \partial_{x_j}g_2(0') = 0 \quad \mbox{for}~j = 2,3,\ldots,n-1,
	$$
	where $g_i=u_-|_{\partial \cD_i}$.
	By Lemma \ref{linear_lower},
	we also have
	\begin{equation}\label{est:g_x1_lower}
		\partial_{x_1} g_i(0')\ge \frac{1}{1+k_i\gamma}>0  \quad \mbox{for } i=1,2. 
	\end{equation}

	Now we are ready to prove Theorem \ref{main_thm}. We separate the proof into two cases when $\gamma>1/\mu$ and when $\gamma=1/\mu$.
	
	\begin{proof}[Proof of Theorem \ref{main_thm} when $\gamma>1/\mu$]
		Recall that we assume $\mu = 1$. We will prove \eqref{grad_u_lower_bound} for $n \ge 3$ in the following four steps. The proof for $n = 2$ is similar and simpler. A brief remark on the case $n=2$ will be provided at the end.
		
		\textbf{Step 1.} By the elliptic theory, since $\widetilde\Omega$ is symmetric in $x_1$ and  $\varphi$ is odd in $x_1$, we know that {$u_+$ and $u_-$ are smooth up to the interfacial boundaries for fixed $\varepsilon>0$,} $|u| \le 5$, and $u$ is odd in $x_1$.
		Let $\tilde{u}$, $v$, $\bar{v}$ be defined as in Lemma \ref{lemma_reduction_2}.
		By Lemma \ref{lemma_reduction_2}, we know that $\bar{v}$ satisfies the equation \eqref{equation_v_bar_2}
		$$
		\dv((\varepsilon+|y'|^2)\nabla \bar v) - 2\gamma^{-1} \bar{v}(y') = \dv F + \tilde{G}_1+\tilde{G}_2 \quad \mbox{in}~B_{3/4} \subset \bR^{n-1}.
		$$
		For simplicity of notation, we still denote $\tilde{G}_i$ by $G_i$.
		From the proof of Theorem \ref{thm_upperbound}, we know that
		\begin{equation}\label{prelim_estimate}
			|\nabla_{y'} v(y)| \le C(\varepsilon +|y'|^2)^{-1/2}, \quad |\partial_n v(y)| \le C (\varepsilon + |y'|^2)/ \varepsilon\quad \mbox{in}~ Q_{3/4,\varepsilon},
		\end{equation}
		where $Q_{3/4,\varepsilon}$ is defined in \eqref{Q_s_t}. Since both $\cD_1$ and $\cD_2$ are balls, $e^i$ from \eqref{a_ij_formula} can be bounded by $C|y'|^4$. Therefore by \eqref{FG_bound_b},  for $|y'| < 3/4$ we have
		\begin{equation}\label{FG_bound_2}
			\left\{
			\begin{aligned}
				&|F(y')| \le C\Big(|y'|(\varepsilon + |y'|^2) + |y'|^4 (\varepsilon +|y'|^2)^{-1/2}\Big), \\
				&|G_1(y')| \le C(\varepsilon + |y'|^2),\quad
				|G_2(y')| \le C|y'|^{1+\beta}.
			\end{aligned}\right.
		\end{equation}
		We denote $Y_{k,i}$ to be {a normalized spherical harmonic of degree $k$} so that $\{Y_{k,i}\}_{k,i}$ forms an orthonormal basis of $L^2(\bS^{n-2})$, {$Y_{1,1}(\xi)=\xi_1/c_n$, where $c_n^2=\int_{\bS^{n-2}}\xi_1^2\,d\xi$}, and we write $y' = (r ,\xi)$ {in polar coordinates}. Since $\bar v$ is odd with respect to $y_1$, and in particular $\bar{v}(0) = 0$, we have the following decomposition
		\begin{equation}
			\label{v_bar_expansion}
			\bar{v}(y') = V_{1,1}(r)Y_{1,1}(\xi) + \sum_{k=2}^\infty \sum_{i=1}^{N(k)} V_{k,i}(r)Y_{k,i}(\xi), \quad { y'} \in B_{3/4}\setminus\{0\},
		\end{equation}
		where $V_{k,i}(r) = \int_{\bS^{n-2}} \bar{v}(r,\xi) Y_{k,i}(\xi) \, d\xi$ and $V_{k,i} \in C([0,3/4)) \cap C^\infty((0,3/4))$. Since $\bar{v}(0) = 0$ and $\varepsilon + |y'|^2$ is independent of $\xi$, we know that $V_{1,1}$ satisfies $V_{1,1}(0) = 0$, and
		\begin{equation*}
			LV_{1,1} = H_1(r)+H_2(r), \quad 0 < r <3/4,
		\end{equation*}
		where $L$ is the differential operator defined in \eqref{def_L},
		\begin{align*}
			H_1(r) &= \int_{\bS^{n-2}} \frac{(\dv F + G_1) Y_{1,1}(\xi)}{\varepsilon+r^2} \, d\xi \\
			&= \int_{\bS^{n-2}} \frac{\partial_r F_r {+\frac{n-2}{r}F_r}
				+ \frac{1}{r} {\operatorname{div}_{\bS^{n-2}}F_\xi} + G_1}{\varepsilon + r^2} Y_{1,1}(\xi) \, d\xi\\
			&= \partial_r \left(\int_{\bS^{n-2}} \frac{F_r}{\varepsilon + r^2} Y_{1,1}(\xi) \, d\xi \right) \\
			&\quad + \int_{\bS^{n-2}} {\Big(\frac{2r}{(\varepsilon+r^2)^2}+ \frac{n-2}{r(\varepsilon+r^2)}\Big)}F_r Y_{1,1}- \frac{{F_\xi\cdot\nabla_\xi Y_{1,1}}}{r(\varepsilon + r^2)} + \frac{G_1 Y_{1,1}}{\varepsilon+ r^2}\, d\xi\\
			&=: A'(r) + B(r), \quad 0 < r < 3/4.
		\end{align*}
		and
		\begin{align*}
			H_2(r) = \int_{\bS^{n-2}} \frac{G_2 Y_{1,1}}{\varepsilon+ r^2}\, d\xi, \quad 0<r<3/4.
		\end{align*}
		Hence $A(r),B(r), H_2(r) \in C^1((0,3/4))$ satisfy, in view of \eqref{FG_bound_2}, that
		\begin{equation}
			\label{AB_bounds}
			|A(r)| \le C(n)r , \quad |B(r)| \le C(n), \quad |H_2(r)|\le C(n)\frac{r^{1+\beta}}{\varepsilon+r^2},\quad 0 < r < 3/4.
		\end{equation}
		
		\textbf{Step 2.}
		We will prove, for some $\varepsilon$-dependent constant $C_1(\varepsilon)$ and positive $\varepsilon$-independent constant $C_2$, that
		\begin{equation}
			\label{V_11_formula}
			|V_{1,1}(r) - C_1 (\varepsilon) h(r)| \le  C_2 r^{1+\min\{\alpha,\beta\}}, \quad 0<r\le 1/2,
		\end{equation}
		where $h(r)$ is the solution of \eqref{Homogeneous_ODE} satisfying $h(0)=0$, $h(1/2) = 1/2$. We use the method of reduction of order to write down bounded functions $p_1$, $p_2$ satisfying $Lp_i = H_i$ in {$(0,1/2)$}, and then give the estimates $|p_1(r)| \le C r^{1+\alpha}$, $|p_2(r)| \le C r^{1+\beta}$.
		
		For $i=1,2$, let $p_i = h w_i$ and
		$$
		w_i(r) := \int_0^r\frac{1}{h^2(s) s^{n-2} (\varepsilon + s^2)} \int_0^s h(\tau) \tau^{n-2} (\varepsilon + \tau^2) H_i(\tau) \, d\tau ds, \quad 0 < r \le 1/2.
		$$
		By a direct computation,
		$$
		Lp_i = L(hw_i) = h w_i'' + \left[ 2h' + \left( \frac{n-2}{r} + \frac{2r}{\varepsilon + r^2} \right)h \right]w_i' = \frac{h}{G}(Gw_i')' = {H_i},
		$$
		where $G= h^2 r^{n-2}(\varepsilon + r^2)$. The estimate $|p_1(r)| \le C r^{1+\alpha}$ can be proved in the same way as in Step 2 of the proof of \cite{DongYangZhu26}*{{Theorem} 1.5}. Here we show the estimate $|p_2(r)| \le C r^{1+\beta}$ using the improved estimate \eqref{h_upper_lower}.
		Since $h$ is increasing, by \eqref{AB_bounds}, we have
		\begin{align*}
			|w_2(r)|\le C\int_0^r\frac{1}{h^2(s) s^{n-2} (\varepsilon + s^2)} \int_0^s h(\tau) \tau^{n-1+\beta}  \, d\tau ds\le C\int_0^r\frac{s^{2+\beta}}{h(s)  (\varepsilon + s^2)}\, ds.
		\end{align*}
		Using \eqref{h_upper_lower}, we can estimate that
		$$\frac{s^{2+\beta}}{h(s)  (\varepsilon + s^2)}\le C \frac{s^{2+\beta}}{\varepsilon^\frac{\alpha-1}{2}s\,\varepsilon}\le C s^{\beta-\alpha} \quad \mbox{for } s\in (0,{4}\sqrt{\varepsilon}).$$
		By \eqref{h_lowerbound},  we also have 
		$$\frac{s^{2+\beta}}{h(s)  (\varepsilon + s^2)}\le C \frac{s^{2+\beta}}{s^\alpha\,s^2}\le C s^{\beta-\alpha} \quad \mbox{for } s\in [{4}\sqrt{\varepsilon},1/2].$$
		Therefore, using \eqref{h_bounds},
		\begin{align*}
			|p_2(r)| &=   h(r) |w_2(r)|\le Ch(r)\int_0^r\frac{s^{2+\beta}}{h(s)  (\varepsilon + s^2)}\, ds\\
			&\le Cr^\alpha\int_0^r s^{\beta-\alpha}\,ds\le  Cr^{1+\beta}, \quad\quad 0<r\le 1/2.
		\end{align*}
		Since $V_{1,1} - p_1-p_2$ is bounded and satisfies $L(V_{1,1} - p_1-p_2) = 0$ in $(0,1/2)$ and $V_{1,1}(0) - p_1(0)-p_2(0) = 0$, by Lemma \ref{ODE_lemma}, there exists a constant $C_1(\varepsilon)$ such that \begin{equation*}
			V_{1,1} - p_1-p_2 = C_1(\varepsilon)h.
		\end{equation*}
		Hence \eqref{V_11_formula} follows.
		
		\textbf{Step 3.} We will show that 
		\begin{equation}\label{C_1_epsilon_lowerbound}
			C_1 (\varepsilon) > \frac{1}{C}
		\end{equation}
		for some positive $\varepsilon$-independent constant $C$.
		
		By the definition of $\tilde u_+$ in \eqref{subtr} {and the choice of $a$ and $b$},
		$$\tilde u_+=u_++\frac{1}{{2}(\gamma-1)} (\partial_{x_1} g_1(0')+\partial_{x_1}g_2(0'))x_1. $$
		Therefore, by Lemma \ref{linear_lower} and \eqref{est:g_x1_lower}, we have
		$\tilde u_+ \ge x_1$ in $\{x_1 \ge 0\}\cap\widetilde\Omega$ and $\tilde u_+ \le x_1$ in $\{x_1 \le 0\}\cap\widetilde\Omega$. Then
		$$\left\{\begin{aligned}
			&\bar{v}(y') \ge y_1, &&\mbox{when}~y_1\ge 0,\\ 
			&\bar{v}(y') \le y_1, &&\mbox{when}~y_1\le 0.\\ 
		\end{aligned}\right.
		$$
		Recalling that $Y_{1,1}(\xi)=\xi_1/c_n$, we obtain
		$$
		V_{1,1}(r) = \int_{\bS^{n-2}}  \bar{v}(r,\xi) Y_{1,1}(\xi) \, d\xi \ge c_n r, \quad 0 < r < 3/4,
		$$
		where $c_n:=(\int_{\bS^{n-2}}  {\xi_1^2} d\xi)^{1/2}>0$.
		Combining with \eqref{h_bounds} and \eqref{V_11_formula}, we have
		\begin{align*}
			c_n r \le V_{1,1}(r) \le C_1(\varepsilon)h(r) + \frac{1}{2}c_n r \le C_1(\varepsilon)r^\alpha + \frac{1}{2}c_n r \quad \forall 0 < r \le r_0,
		\end{align*}
		where ${r_0 = \min\{1/2,(c_n/(2C_2))^{1/\min\{\alpha,\beta\}}\}}$. This implies
		$$
		C_1(\varepsilon) \ge \frac{1}{2}c_n r_0^{1-\alpha}.
		$$
		
		\textbf{Step 4.} Completion of the proof. 
		
		It follows, in view of \eqref{h_bounds}, \eqref{V_11_formula} and \eqref{C_1_epsilon_lowerbound}, that
		\begin{equation}
			\label{V_11_lowerbound}
			V_{1,1}(r) \ge \frac{1}{C}  h(r) - C_2 r^{1+\min\{\alpha,\beta\}} \ge \frac{1}{2C} h(r) , \quad 0 < r < r_1
		\end{equation}
		for some small $\varepsilon$-independent constant $r_1$. By \eqref{h_lowerbound},
		\begin{equation}
			\label{h1_lowerbound}
			h(2c\sqrt{\varepsilon}) > 2^{\alpha-1}(2^\alpha - 1)c^\alpha \varepsilon^{\alpha/2},
		\end{equation}
		where $c$ is the constant in Lemma \ref{lemma_subsolution}. By \eqref{v_bar_expansion}, \eqref{V_11_lowerbound}, and \eqref{h1_lowerbound}, we have
		\begin{align*}
			\left( \int_{\bS^{n-2}} |\bar{v}(2c\sqrt{\varepsilon},\xi)|^2 \, d\xi \right)^{1/2} &\ge |V_{1,1}(2c\sqrt{\varepsilon})| \ge \frac{1}{C} h(2c\sqrt{\varepsilon}) \ge \frac{1}{C}\varepsilon^{\alpha/2} \quad \forall 0 < \varepsilon < \varepsilon_0
		\end{align*}
		for some small positive $\varepsilon_0$ depending only on $n$ and $\gamma$. Then, there exists a $\xi_0 \in \bS^{n-2}$ such that $|\bar{v}(2c\sqrt{\varepsilon}, \xi_0)| \ge \frac{1}{C} \varepsilon^{\alpha/2}$. Since $\bar v$ is the vertical average of $v$, and recall that $v(y) = \tilde u_+(x)$, there exists an $x_n \in (-\varepsilon/2 + h_2(x'), \varepsilon/2 + h_1(x'))$ such that
		\begin{equation}
			\label{u_lower_bound}
			|\tilde u_+(2c\sqrt{\varepsilon}, \xi_0, x_n)| \ge \frac{1}{C} \varepsilon^{\frac{\alpha}{2}}.
		\end{equation}
		Estimate \eqref{grad_u_lower_bound} follows from \eqref{u_lower_bound}, $u(0) = 0$, and the fact that $\partial_{x_1} g_i(0')$ is uniformly bounded. This concludes the proof when $n \ge 3$.
		
		When $n = 2$, by Lemma \ref{lemma_reduction}, $\bar{v}$ satisfies the equation
		$$
		((\varepsilon+|y_1|^2)\bar v')'- 2\gamma^{-1} \bar{v} = F' + {G_1+G_2}\quad \mbox{in}~(-3/4,3/4),
		$$
		where $F$, {$G_1$, and $G_2$} satisfy the same bounds as in \eqref{FG_bound_2}. By using $\bar{v}$ in place of $V_{1,1}$ and repeating steps 2-4 with minor modifications, we can establish the estimate \eqref{grad_u_lower_bound} for $n = 2$.
	\end{proof}

	\begin{proof}[Proof of Theorem \ref{main_thm} when $\gamma=1/\mu$]
		As in the proof of the case when $\gamma>1/\mu$, we assume that $\mu=1$ and focus on the proof for $n\ge 3$. {We use the same definitions} and notation as those in the previous case when $\gamma>1/\mu$. {Note that in this case, the linear function in \eqref{subtr} is identically equal to zero.}  By Lemma \ref{lemma_reduction}, we know that $\bar{v}$ satisfies the equation \eqref{equation_v_bar}
		$$
		\dv((\varepsilon+|y'|^2)\nabla \bar v) - 2\gamma^{-1} \bar{v}(y') = \dv F + {G}_1+{G}_2 \quad \mbox{in}~B_{3/4} \subset \bR^{n-1}.
		$$
		Since \eqref{prelim_estimate} still holds in this case, by \eqref{FG_bound}, for $|y'|<3/4$ we have
		\begin{equation}\label{FG_bound_3}
			|F(y')| \le C\Big(|y'|(\varepsilon + |y'|^2) + |y'|^4 (\varepsilon +|y'|^2)^{-1/2}\Big), \quad
			|G_1(y')| \le C(\varepsilon + |y'|^2).
		\end{equation} 
		Since $\bar v$ is odd with respect to $y_1$, and in particular $\bar{v}(0) = 0$, we have the decomposition \eqref{v_bar_expansion}.
		Since $\bar{v}(0) = 0$ and $\varepsilon + |y'|^2$ is independent of $\xi$, we know that $V_{1,1}$ satisfies $V_{1,1}(0) = 0$, and
		\begin{equation*}
			LV_{1,1} = H_1(r) + H_2(r), \quad 0 < r <3/4,
		\end{equation*}
		where $L$ is defined in \eqref{def_L},
		\begin{align*}
			H_1(r) &= -\frac{1}{\gamma} \int_{\bS^{n-2}} \frac{(\partial_{y_1}g_1(0) + \partial_{y_1}g_2(0))y_1 Y_{1,1}(\xi)}{\varepsilon+r^2} \, d\xi \\
			&= - {c_n}\frac{\partial_{y_1}g_1(0) + \partial_{y_1}g_2(0)}{\gamma} \frac{r}{\varepsilon+r^2},\quad {c_n>0,}
		\end{align*}
		and
		\begin{align*}
			H_2(r) &= \int_{\bS^{n-2}} \frac{(\dv F + {G_1}) Y_{1,1}(\xi)}{\varepsilon+r^2} \, d\xi\\
			& -\sum_{i=1}^2 \frac{1}{\gamma} \int_{\bS^{n-2}} \frac{ \big( \sqrt{1+ |\nabla_{y'}h_i|^2} g_i(y') - \partial_{y_1}g_i(0)y_1 \big) Y_{1,1}(\xi)}{\varepsilon+r^2} \, d\xi\\
			&= \partial_r \left(\int_{\bS^{n-2}} \frac{F_r}{\varepsilon + r^2} Y_{1,1}(\xi) \, d\xi \right) + \int_{\bS^{n-2}} \Big(\frac{2r}{(\varepsilon+r^2)^2}{+\frac{n-2}{r(\varepsilon+r^2)}}\Big) F_r Y_{1,1}\\
			&\quad - \frac{F_\xi\cdot\nabla_\xi Y_{1,1}}{r(\varepsilon + r^2)} + \frac{{G_1} Y_{1,1}}{\varepsilon+ r^2}\, d\xi\\
			&\quad  -\sum_{i=1}^2 \frac{1}{\gamma} \int_{\bS^{n-2}} \frac{ \big( \sqrt{1+ |\nabla_{y'}h_i|^2} g_i(y') - \partial_{y_1}g_i(0)y_1 \big) Y_{1,1}(\xi)}{\varepsilon+r^2} \, d\xi\\
			&=: A'(r) + B(r) + D(r), \quad 0 < r < 3/4.
		\end{align*}
		By \eqref{FG_bound_3}, \eqref{fg_0}, \eqref{def:c_2}, and $\|\nabla g_i\|_{C^{\beta}}\le C/k_i$ as a consequence of \eqref{u-_1+holder},
		one can show that $A(r),B(r), D(r) \in C^1((0,3/4))$ satisfy
		\begin{equation}
			\label{ABD_bounds}
			|A(r)| \le C(n)r, \quad |B(r)| \le C(n), \quad |D(r)| \le C(n,k_1,k_2) \frac{r^{1+\beta}}{\varepsilon+r^2}
		\end{equation}
		for $r\in(0,3/4)$ and $\beta \in(0, 1)$.
		
		{ Note that $h(r)=r$ is a homogeneous solution for $L$ when $\gamma=1$. One can also use the same method of reduction of order as in the proof of \cite{DongYangZhu26}*{Lemma 7.1} to show that $h(r) = r$ is the unique solution of $Lh=0$ in $(0,1/2)$, such that $h(0)=0$ and $h(1/2)=1/2$.}
		Next, we use the method of reduction of order again to find two bounded particular solutions $p_1$ and $p_2$ satisfying $Lp_i = H_i$ in $(0,1/2)$ for $i = 1,2$. 
		We define
		\begin{align*}
			w_1(r) :=& -\int_r^{\frac{1}{2}}\frac{1}{h^2(s) s^{n-2} (\varepsilon + s^2)} \int_0^s h(\tau) \tau^{n-2} (\varepsilon + \tau^2) H_1(\tau) \, d\tau ds\\
			=& {c_n} \frac{\partial_{y_1}g_1(0) + \partial_{y_1}g_2(0)}{2\gamma (n+1)} \Big[ {\log(\varepsilon + \frac{1}{4})} - \log (\varepsilon + r^2) \Big],
		\end{align*}
		we know that
		$$
		p_1(r) := h(r) w_1(r) = {c_n}\frac{\partial_{y_1}g_1(0) + \partial_{y_1}g_2(0)}{2\gamma (n+1)} r\Big[ \log(\varepsilon + \frac{1}{4}) - \log (\varepsilon + r^2) \Big]
		$$
		is a bounded particular solution to ${L}p_1=H_1$. Using \eqref{ABD_bounds}, we can construct and estimate $|p_2| \le Cr^{1+\min\{\alpha,\beta\}}$ as in the case $\gamma>1$. Then there exists a constant $C_1(\varepsilon)$ such that
		$$
		V_{1,1}(r) = p_1(r) + p_2(r) + C_1(\varepsilon)r.
		$$
		Plugging in $r = 1/2$, we can see that $|C_1(\varepsilon)| \le C$ independent of $\varepsilon$. Therefore, $p_1(r)$ is the leading order term. Namely,
		\begin{equation}\label{V11_lower_2}
		\begin{aligned}
		V_{1,1}(r)&\ge p_1(r)-Cr^{1+\beta}-Cr\\
		&\ge -\frac{{c_n}((1+k_1)^{-1}+(1+k_2)^{-1})}{4\gamma(n+1)}r\log(\varepsilon+r^2),\\
		&\hspace{2em} r\in(0,r_1),\quad {0<\varepsilon<r_1^2}.
		\end{aligned}
		\end{equation}
		where $r_1=r_1(n,\gamma,k_1,k_2)$ is a positive constant independent of $\varepsilon$ {and in the last inequality we used \eqref{est:g_x1_lower}}.
		By \eqref{v_bar_expansion} and \eqref{V11_lower_2}, we have
		\begin{align*}
			\left( \int_{\bS^{n-2}} |\bar{v}(\sqrt{\varepsilon},\xi)|^2 \, d\xi \right)^{1/2} &\ge |V_{1,1}(\sqrt{\varepsilon})| \ge -\frac{1}{C} \sqrt{\varepsilon}\log(\varepsilon)\quad \forall 0 < \varepsilon < \varepsilon_0
		\end{align*}
		for some small positive $\varepsilon_0$ depending only on $n$, $\gamma$, $k_1$, $k_2$. Then {arguing as in the case when $\gamma>1/\mu$,}
		{we conclude the proof} when $n \ge 3$.
		
		When $n = 2$, by Lemma \ref{lemma_reduction}, $\bar{v}$ satisfies the equation
		$$
		((\varepsilon+|y_1|^2)\bar v')'- 2\gamma^{-1} \bar{v} = F' + G_1+G_2\quad \mbox{in}~(-3/4,3/4),
		$$
		where $F$ and $G_1$ satisfy the same bounds as in \eqref{FG_bound_3}. By using $\bar{v}$ in place of $V_{1,1}$ and repeating steps 2-4 with minor modifications, we can establish the estimate \eqref{grad_u_lower_bound_new} for $n = 2$.
	\end{proof}

	\section{Variable bonding parameters}\label{sec:variable-bonding}
	
	In applications, the bonding properties along the inclusion surfaces may vary with position. It is therefore natural to allow the bonding parameters to depend on $x$. In this section, we briefly explain how the preceding results and proofs extend to this case.
	
	For $i=1,2$, assume that $\gamma_i\in C^{1,1}(\partial\cD_i)$ and that
	\begin{equation}
		\label{eq9.38}
		\gamma_i(x)\geq \gamma_0>0\quad\mbox{on}~\partial\cD_i.
	\end{equation}
	We consider the problem
	\begin{equation}\label{robin_variable}
		\begin{cases}
			\Delta u_+ =0&\mbox{in}~\widetilde\Omega,\\
			\Delta u_- =0&\mbox{in}~\mathcal{D}_1\cup\mathcal{D}_2,\\ \partial_\nu u_+=k_i \partial_\nu u_-= -\gamma_i(x)^{-1}(u_+-u_-)&\mbox{on}~\partial \mathcal{D}_i, \quad i=1,2,\\
			u_+= \varphi&\mbox{on}~\partial\Omega.
		\end{cases}
	\end{equation}
	The existence, uniqueness, and the maximum principle results established in Section \ref{sec:prelim} still hold with minor modifications.
	
	We define $\bar\gamma$ as the harmonic mean of the two bonding parameters evaluated at the closest points
	\begin{equation}\label{def:gamma_bar}
		\bar\gamma:= \frac{2}{\dfrac{1}{\gamma_1(0', \varepsilon/2)} + \dfrac{1}{\gamma_2(0', -\varepsilon/2)}}.
	\end{equation}
	We discuss separately the case with symmetry and the general case without it.\\
	
	\textbf{Case 1: Symmetric domains and data.}
	
	Under the assumption on symmetry imposed in Theorem \ref{thm_upperbound}, the conclusions of Theorem \ref{thm_upperbound} remain valid, and the optimality of the resulting estimates follows from Theorem \ref{main_thm}.
	
	\begin{theorem}
		{Under the assumptions of Theorem \ref{thm_upperbound}, let} $\Omega_{R_0}$ be symmetric in $x_j$, { let $\gamma_1(x)$ and $\gamma_2(x)$ be even in $x_j$,} and let $u$ be the solution of \eqref{robin_variable} that is odd in $x_j$ for some $1 \le j \le n-1$. Then estimates \eqref{grad_u_upperbound1}--\eqref{grad_u_upperbound2} remain valid with $\gamma$ replaced by $\bar\gamma$ defined in \eqref{def:gamma_bar}.
	\end{theorem}
	
	Since the proof essentially follows that of Theorem \ref{thm_upperbound} with minor modifications, we point out only the necessary modifications and give a brief proof sketch. In the proof of Theorem \ref{thm-1/2}, we need to change the auxiliary function $\psi$ to be
\[
\begin{aligned}
\psi(y)
={}&F_0(y')\left[
\frac{\gamma_1^{-1}+\gamma_2^{-1}}{4r^2}(y_n^2-r^4)
+\frac{\gamma_1^{-1}-\gamma_2^{-1}}{2}y_n
\right]\\
&+\frac{y_n^2-r^4}{4r^4}
\left[
\frac{y_n+r^2}{\gamma_1}E_1(y)
+\frac{r^2-y_n}{\gamma_2}E_2(y)
\right],
\end{aligned}
\]
where
\[
F_0(y'):=\frac{\varepsilon+h_1(y')-h_2(y')}{2r^2},
\qquad E_i(y):=F_i(y)-F_0(y'),\quad i=1,2.
\]
Then we can take
\[
c_L(y):=-\frac{a^{nn}(y)F_0(y')}{2r^2}
\left(\frac{1}{\gamma_1}+\frac{1}{\gamma_2}\right).
\]

Another main change occurs in the dimensional-reduction argument of Section \ref{sec4}. When the bonding parameters $\gamma_i$ are nonconstant, $G_1$ contains an extra term
	\begin{equation} \label{extra_G1}
		\Big(\frac{1}{\gamma_1(y')} - \frac{1}{\gamma_1(0')} + \frac{1}{\gamma_2(y')} - \frac{1}{\gamma_2(0')} \Big) \bar{v},
	\end{equation}
	which can be controlled by $$C (\|\nabla \gamma_1 \|_{L^\infty} + \|\nabla \gamma_2 \|_{L^\infty}) |y'| |\bar v|.$$ 
	Since $u$ is odd in $x_j$, so is $\bar v$, and hence $\bar v(0')=0$. Together with the preceding bound {of $\nabla \bar v$}, the extra term \eqref{extra_G1} will serve as a higher-order error  relative to the leading homogeneous part of the equation \eqref{equation_v_bar}.  Consequently, the additional term does not affect the iteration process in the proof of Theorem \ref{thm_upperbound}, and the same estimates remain valid with $\gamma$ replaced by $\bar\gamma$.
	
	\textbf{Case 2: General domains and data with ``nearly constant" bonding parameters.}
	
	In the absence of symmetry assumptions on the solution $u$, we instead suppose that
	$$
	\nabla_{x'} \gamma_1(0', \varepsilon/2) = \nabla_{x'} \gamma_2(0', -\varepsilon/2) = 0.
	$$
	That is, $\gamma_1$ and $\gamma_2$ are ``nearly constant" at the closest points. In this case, the extra term \eqref{extra_G1} can be controlled by
	$$C (\|\nabla^2 \gamma_1 \|_{L^\infty} + \|\nabla^2 \gamma_2 \|_{L^\infty}) |y'|^2 |\bar v|.$$ 
	Even though $\bar v$ may not vanish at $y' = 0'$, the $|y'|^2$ term is sufficient to make \eqref{extra_G1} a higher-order error relative to the leading homogeneous part of the equation \eqref{equation_v_bar}. As a consequence, the conclusions of Theorems \ref{thm_upperbound_3} and \ref{thm_upperbound_2} remain valid with $\gamma$ replaced by $\bar\gamma$.
	
	\textbf{Case 3: General domains, data, and bonding parameters.}
	
	Now we discuss the most general case {without symmetry assumptions}, and $\nabla_{x'} \gamma_1(0', \varepsilon/2)$ and $\nabla_{x'} \gamma_2(0', -\varepsilon/2)$ are {not both zero}. In this case, the extra term \eqref{extra_G1} generally satisfies only the first-order bound
	$$|\eqref{extra_G1}| \le C (\|\nabla \gamma_1 \|_{L^\infty} + \|\nabla \gamma_2 \|_{L^\infty}) |y'| |\bar v|.$$ 
	Moreover, $\bar v$ need not vanish at $y'=0'$. Consequently, when $\bar\gamma$ takes the critical value $2/(n\mu)$, {similar to the $G_2$ term in the symmetric case when $\gamma=1/\mu$,} this extra term has the same order as the leading homogeneous part of \eqref{equation_v_bar}. This would produce an additional logarithmic factor in the iteration process. To summarize, we have the following theorem.
	
	\begin{theorem}
		For $n \ge 3$, $\varepsilon \in (0,R_0/4)$, $\mu > 0$, let $h_1,h_2$ satisfy \eqref{fg_0}, \eqref{fg_2}, and \eqref{fg_3}. For $i=1,2$, let $\gamma_{i}\in C^{1,1}(\partial\cD_i)$ satisfy \eqref{eq9.38}. Let $u$ be the solution of \eqref{robin_variable}. Then there exists a positive universal constant $C$ depending on $\bar{\gamma}$, such that
		\begin{itemize}
			\item when $0 <\bar \gamma < 2/(n\mu)$,
			\begin{equation*}
				|\nabla u(x)| \le C \|\varphi\|_{L^\infty(\partial\Omega)}\quad \mbox{in}~ \Omega_{R_0/2},
			\end{equation*}
			\item when $\bar\gamma = 2/(n\mu)$,
			\begin{equation*}
				|\nabla u(x)| \le C \|\varphi\|_{L^\infty(\partial\Omega)} |\log (\varepsilon+|x'|^2)|\quad \mbox{in}~ \Omega_{R_0/2},
			\end{equation*}
			\item when $\bar\gamma >2/(n\mu)$,
			\begin{equation*}
				|\nabla u(x)| \le C \|\varphi\|_{L^\infty(\partial\Omega)}(\varepsilon + |x'|^2)^{\frac{\alpha_0-1}{2}} \quad \mbox{in}~  \Omega_{R_0/2},
			\end{equation*}
		\end{itemize}
		where $\alpha_0$ is given in \eqref{alpha_n_zero}, with $\gamma$ replaced by $\bar \gamma$.
	\end{theorem}

	{Finally, the flattening map $\Phi^{-1}$ depends only on the geometry. The assumed $C^{1,1}$ regularity of the bonding parameters is used in the Robin transformation and in estimates for the resulting coefficients and source terms.}

	
	\appendix
	
	\section{\texorpdfstring{{The normal-preserving flattening map}}{The normal-preserving flattening map}}\label{sec_appendix}
	In this appendix, we provide the definition of the flattening map introduced in \cites{DYZ23,DongYangZhu26} and related estimates. The flattening map preserves the normal directions on the upper and lower boundaries.
	
Choose $\chi\in C_c^\infty(B_{R_0})$ with $0\le\chi\le1$ and $\chi=1$ on $B_{7R_0/8}$. For $i=1,2$, define $\tilde h_i=\chi h_i$ in $B_{R_0}$ and extend it by zero outside $B_{R_0}$. Then $\tilde h_i\in C^{1,1}(\bR^{n-1})$, with norm controlled by $R_0$ and $\|h_i\|_{C^{1,1}(B_{R_0})}$.
	Recall that for any $x_0 \in \Omega_{R_0/2}$, 
	$$
	Q_{s,t} = \{ y = (y', y_n) \in \bR^n:~ |y'-x_0'| < s,\, |y_n| < t \}.
	$$
	For any $y \in \overline{Q_{2r,r^2}}$, where $r=\frac{1}{4}(\varepsilon+|x'_0|^2)^{\frac{1}{2}}$, we define the map $x = \Phi(y)$ by
	\begin{equation}\label{Phi}
		\left\{
		\begin{aligned}
			x' &= y' - h(y),\\
			x_n &= \frac{1}{2} \Big[\frac{y_n}{r^2} (\varepsilon + h_1(y') - h_2(y')) + h_1(y') +h_2(y') \Big],
		\end{aligned}
		\right.
	\end{equation}
	where
	\begin{equation*}
		h(y) = (y_n - r^2)(y_n + r^2)(\Theta y_n + \Xi),
	\end{equation*}
	$$
	\left\{
	\begin{aligned}
		\Theta &= \frac{1}{8r^6} [\varepsilon + h_1(y') - h_2(y')] \nabla_{y'} [\widehat h_1(y) + \widehat h_2(y)],\\
		\Xi &= \frac{1}{8r^4} [\varepsilon + h_1(y') - h_2(y')] \nabla_{y'} [\widehat h_1(y) - \widehat h_2(y)],
	\end{aligned}
	\right.
	$$
	$\widehat h_1$ and $\widehat h_2$ are mollifications of $\tilde h_1$ and $\tilde h_2$ given by
	\begin{equation}\label{mollification}
		\widehat h_1(y) := \int_{\bR^{n-1}} \tilde h_1(y' - \kappa z') \eta(z') \, dz',\quad \widehat h_2(y) := \int_{\bR^{n-1}} \tilde h_2(y' - \kappa z') \eta(z') \, dz',
	\end{equation}
	where $\eta$ is a nonnegative smooth function with unit integral supported in $B_1 \subset \bR^{n-1}$,
	and 
	$$\kappa = \frac{r^4 - y_n^2}{r}\ge 0.$$
	{Note that for sufficiently small $r$, we have \eqref{Phi} 
$\Phi(Q_{1.9r, r^2})\subset \Omega_{x_0,2r}$ and $\Omega_{x_0,r} \subset \Phi(Q_{1.1r, r^2})$.}
    For $y \in Q_{2r,r^2}$, some key estimates for $\widehat h_1(y)$ and $\widehat h_2(y)$ are
	\begin{equation}\label{est:mollification}
		\begin{aligned}
			&|\nabla_{y'}^k \widehat h_i(y)| \le C r^{2-k} \quad \mbox{for}~k = 0,1,2,\\
			&|\partial_{y_n} \nabla_{y'} \widehat h_i(y)| \le C r,\\
			&|\nabla_{y'}^3 \widehat h_i(y)| \le \frac{Cr}{r^4 - y_n^2},\\
			&|\partial_{y_n}\nabla_{y'}^2 \widehat h_i(y)| \le \frac{Cr^2}{r^4 - y_n^2},\\
			&|\partial_{y_n}^2\nabla_{y'} \widehat h_i(y)| \le \frac{Cr^3}{r^4 - y_n^2}
		\end{aligned}
	\end{equation}
	for $i = 1,2$. We refer to \cite{DYZ23}*{Section 3} for the details. Note that the map $\Phi$ is invertible when $r$ is sufficiently small.
	
	\bibliographystyle{amsplain}

\end{document}